\documentclass{imsart}

\usepackage[round]{natbib}
\usepackage{amsmath,amssymb,amsthm,bm,enumerate,mathrsfs,mathtools}
\usepackage{latexsym,color,verbatim,multirow}
\usepackage{graphicx}
\usepackage{caption}
\usepackage{subcaption}

\usepackage{graphics,amssymb,color}
\usepackage{graphicx}
\usepackage{algorithm}
\usepackage{algorithm,algorithmic}
\usepackage{xfrac}

\usepackage{hyperref}
\hypersetup{colorlinks,
citecolor=blue,
linkcolor=blue,
urlcolor=black}

\newcommand{\argmax}{\mathop{\mathrm{argmax}}}
\newcommand{\argmin}{\mathop{\mathrm{argmin}}}
\newcommand{\minimize}{\mathop{\mathrm{minimize}}}
\newcommand{\maximize}{\mathop{\mathrm{maximize}}}

\newcommand{\real}{\mathbb{R}}

\newcommand{\Ee}{{\mathbb E}}

\newtheorem{theorem}{Theorem}
\newtheorem{lemma}{Lemma}

\newtheorem{remark}{Remark}
\newtheorem{Example}{Example}[section]

\newcommand{\K}{\mathcal{K}}

\newcommand{\grad}{\nabla}

\newcommand{\hbeta}{\hat{\beta}}
\newcommand{\hz}{\hat{z}}

\newcommand{\solutionset}{{\cal S}}

\begin{document}

\title{Selective sampling after solving a convex problem}

\author{Xiaoying Tian Harris, Snigdha Panigrahi, Jelena Markovic, Nan Bi, Jonathan Taylor}
\runauthor{Tian Harris et al.}
\runtitle{Selective sampling}

\begin{abstract}
We consider the problem of selective inference after solving
a (randomized) convex statistical learning program in the form of a penalized or constrained
loss function.
Our first main result is a change-of-measure formula that 
describes many conditional sampling problems of interest
in selective inference.
Our approach is model-agnostic
in the sense that users may provide their own statistical model for inference, we simply provide
the modification of each distribution in the model after the selection. 

Our second main result describes the geometric
structure in the Jacobian appearing in the change of measure,
drawing connections to curvature measures appearing 
in Weyl-Steiner volume-of-tubes formulae. This Jacobian
is necessary for problems in which the convex penalty is not
polyhedral, with the prototypical example being group LASSO or the
nuclear norm. We derive explicit formulae for the Jacobian 
of the group LASSO.

To illustrate the generality of our method, we consider many examples throughout, varying
both the penalty or constraint in the statistical learning problem as well as the 
loss function, also considering selective inference after solving multiple statistical
learning programs. Penalties considered include LASSO, forward stepwise, stagewise algorithms, marginal screening and generalized LASSO.
Loss functions considered include squared-error, logistic, and log-det for covariance matrix estimation.

Having described the appropriate distribution
we wish to sample from through our first two results, we outline a framework for sampling using
a projected Langevin sampler
in the (commonly occuring) case that the distribution is log-concave. 
\end{abstract}

\maketitle


\section{Introduction}
\label{sec:introduction}

Based on the explosion of freely available and high quality
statistical software, practicing data scientists or statisticians 
can rely on  untold numbers of statistical learning methods that explore their
data. 
Having found an interesting pattern through such methods, in order
to report their findings in the scientific literature, 
the data scientist is confronted in defending the significance
of their findings. Naively assessing significance using methods
that ignore their exploration is recognized as flawed. A common, though
not often used, solution is to use data splitting to evaluate the significance. 
There has been significant recent research on developing methods of inference that
have the same type of guarantees as data splitting but offer more power.
These methods are generally referred to as methods for selective inference
\citep{benjamini_simultaneous,posi,knockoffs,optimal_inference,exact_lasso,selective_sqrt,randomized_response}. 
Loosely speaking, selective inference
recognizes the inherent selection biases in reporting the most ``significant'' results from various statistical models and 
attempts to adjust for the bias in a rigorous framework.

At a high level, selective inference involves two stages: first, query the data by applying some function.
Often this function might be the solution of some convex
optimization problem. Second, posit a model based on the outcome of the query and
perform inference for the parameters or statistical functionals in such models. 
In order to perform valid inference in this second stage,
a common approach espoused above is to condition on the result of the query. An alternative
approach is to address the problem via reduction to a problem of simultaneous inference \cite{posi}.
Under the conditional approach, when the query involves solving a convex problem, we are
interested in distributions of the form
\begin{equation}
\label{eq:inference:problem}
S | \hat{\beta}(S) \in A
\end{equation}
where $S \sim F$ represents our data, $\hat{\beta}$ denotes the solution to a convex
optimization problem and $A$ is some event of interest, set by the data analyst.

In this paper, we address computational problems in the conditional approach to
selective inference. In adopting this conditional approach, it quickly becomes
apparent distributions of the form \eqref{eq:inference:problem}
can be rather complex. In this work, we propose a unified sampling approach that is feasible for 
a wide variety of problems. Combined with the randomization idea in \cite{randomized_response},
this significantly improve both the applicability of selective inference in practice,
and the power of the selective tests.

\subsection{A canonical example}
\label{sec:canonical:example}

As a concrete example, we consider solving the LASSO \cite{lasso} at some fixed $\lambda$, as considered in
\cite{exact_lasso}. Suppose we observe data $(X, y)$, with $X \in \real^{n \times p}$ and $y \in \real^n$.
The query by the data analyst returns $\text{sign}(\hat{\beta}(X,y;\lambda))$ (with $\text{sign}(0)=0$)
where
\begin{equation}
\label{eq:lasso}
\hat{\beta}(X,y;\lambda) = \argmin_{\beta} \frac{1}{2} \|y-X\beta\|^2_2 + \lambda \|\beta\|_1.
\end{equation}
The sign vector can be expressed in terms of $\hat{E}$, the support of $\hat{\beta}$ and $\hat{z}_E$ the signs
of the non-zero coefficients of $\hat{\beta}$.
In principle, after observing $(\hat{E}, \hat{z}_E) = (E_{obs}, z_{E, obs})$,
there is no restriction on the choice of models. For simplicity, we posit a parametric normal model that is based 
on $E$ (dropping the $`obs'$ notation) and fixed $X$, 
\begin{equation}
\label{eq:param:normal}
{\cal M}_{E} = \left\{N(X_{E}b_{E}, \sigma^2_{{E}}I): b_{{E}} \in \real^{{E}}\right\}
\end{equation}
with $\sigma^2_E$ known. Target of inference is naturally $b_{j|E}$, the parameters of this model.
Selective inference essentially takes any distribution $F \in {\cal M}_E$ and consider its corresponding
selective distribution by conditioning on $\{(\hat{E},\hat{z}_E)=(E_{obs},z_{E,obs})\}$. In this case,
our selective distribution is exactly of the form of \eqref{eq:inference:problem}, with $A$ being the
quandrant specified by the nonzero coefficients $E$ and their signs $z_E$. 

Despite the simple form of $A$, the constraint induced a much more intricate set in the space of $y$.
In this special case, \cite{exact_lasso} has worked out the constraint set on $y$ and selective tests
can be computed explicitly. However, this is not generally true, especially in the context of randomized
selection procedures proposed by \cite{randomized_response, optimal_inference}. In many cases, 
the formsof selective distributions
are not explicitly computable and we have to use sampling methods to approximate the conditional distribution. 
Given a large enough sample from such distribution, we can construct the selective tests. In the following
paragraph, we first introduce the randomized Lasso procedure which is a special case of randomized
selection procedures proposed in \cite{randomized_response}. 
Then we would illustrate what is the appropriate conditional distribution to sample
and given a large sample from this distribution, how we carry out the tests for $b_{j|E}$.

First, we propose to incorporate randomness into \eqref{eq:lasso}. In particular, we consider solving
\begin{equation}
\label{eq:random:lasso}
\hat{\beta}(X,y,\omega;\lambda) = \argmin_{\beta} \frac{1}{2} \|y-X\beta\|^2_2 + \lambda \|\beta\|_1 - \omega^T\beta + \frac{\epsilon}{2} \|\beta\|^2_2
\end{equation}
where $\omega \sim G$ is a random vector independent of $(X,y)$, 
whose distribution is chosen by the data analyst, hence known. We will assume that 
$G$ is supported on all of $\real^p$ with density $g$. The
ridge term with small parameter $\epsilon$ ensures the problem above has a solution.
This resembles the ridge term in the elastic net proposed by \cite{elastic_net}.

This type of randomized convex program has been considered in \cite{randomized_response}.
It is espoused as it significantly increase the power in the inference stage without much loss
in the quality of the selected model. In particular, increasing
the scale of the randomization inevitably deteriorates the model selection quality, similar to using a smaller 
training set to choose a model in the context of data splitting. However, 
even a small amount of randomization has empirically shown a fairly noticable increase in power.
Another benefit of randomization include a form of robustness to rare selection events in the selective CLT of 
\cite{randomized_response}.

After solving \eqref{eq:random:lasso}, we use its solution $\hbeta(X,y,\omega;\lambda)$
for model selection. Similar to unrandomized version, the appropriate law after model selection is 
$$
y \mid \hbeta(X,y,\omega; \lambda) \in A, \quad y \sim F,
$$
where $F$ is a member of the parametric normal model specified by \eqref{eq:param:normal}.

Various selective tests can be constructed based on this law including the goodness-of-fit
tests and the p-values and confidence intervals for $b_{j|E}$s. We discuss how to construct
tests for $b_{j|E}$s here as they are a special case of testing a single parameter in a
multi-parameter exponential family. The general approach is laid out in \cite{optimal_inference}.

Selective inference recognizes the fact that interest in $b_{j|E}$ is due to the query returning $E$
as the nonzero variables of the solution to \eqref{eq:random:lasso}. 
In order to provide valid inference, we therefore consider only the part of sample
space that will yield the same model, i.e. conditioning on $\{(\hat{E},\hat{z}_E)=(E_{obs},z_{E,obs})\}$.
Conditioning each model in ${\cal M}_E$ yields a new parametric model,
the {\em selective model}
\begin{equation}
\label{eq:selective:model}
{\cal M}^*_{(E_{obs},z_{E,obs})} = \left\{F^*: \frac{dF^*}{dF}(\cdot) \propto 1_{\{(\hat{E},\hat{z}_E)=(E_{obs},z_{E,obs}), F \in {\cal M}_E\}}(\cdot)\right\}.
\end{equation}
where $(E_{obs}, z_{E,obs})$ are the observed variables and signs.

Note that this approach is not tied to the parametric model in \eqref{eq:param:normal}. In particular, 
when the parametric normal model is not appropriate, one might replace ${\cal M}$ with some other statistical model (i.e.~some
other collection of distributions) and carry out statistical inference in this new selective model. This approach was laid out
in detail in \cite{optimal_inference}, with some asymptotic justification for nonparametric models developed in
\cite{randomized_response}. In fact, in the \cite{exact_lasso} the (pre-conditioning) statistical model proposed was not ${\cal M}_E$, rather it was
the {\em saturated model}
\begin{equation}
\label{eq:sat:model}
{\cal M}= \left\{N(\mu,\sigma^2 I): \mu \in \mathbb{R}^n \right\}.
\end{equation}
A fairly simple calculation shows that the hypothesis tests and confidence intervals
are exactly the same whether we had started with \eqref{eq:sat:model} or \eqref{eq:param:normal}. However this is not the case for examples like forward stepwise, discussed below and in
\cite{sequential_selective,spacings}.

Let us now consider how we might go about inference in \eqref{eq:selective:model}. For instance,
suppose we are simply interested in a goodness-of-fit test and we will consider the model ${\cal M}=\{F_0\}$ so that
${\cal M}^*$ consists of $F_0$ restricted to the event that fixes the active set and signs 
of the LASSO to be $(E_{obs}, z_{E,obs})$. Call this distribution $F_0^*$. A natural way
to test $H_0:F=F_0$ would be follow Fisher's approach by choosing some test statistic $T=T(X,y)$ and compare our observed value
$T_{obs}$ to the distribution of $T$ under $F_0^*$. 
In order to do this, it is sufficient to 
describe the law $\hat{\beta}_{\lambda,*}(F_0)$, i.e. the push forward of $F_0$ and condition this law on the
event that fixes the active set and signs to be $(E_{obs},z_{E,obs})$. Unfortunately, this is no
easy task. Of course, a natural alternative is to use Monte Carlo. A large enough sample
from $F_0^*$ is sufficient to carry out this goodness-of-fit test. 
A version of this goodness-of-fit test
is considered in Section \ref{sec:goodness:of:fit} below.

The issue of what Monte Carlo method to use remains. Our main contribution in this work is an explicit
description of how to sample from distributions such as $F_0^*$. To be precise, we describe
how to sample from conditional distributions where the conditioning
depends on the solution to a convex problem. A naive way to sample from
this distribution would be to take a Monte Carlo sampler to draw from $F_0$, retaining only those points
where the active set and signs of the LASSO agree with $(E_{obs},z_{E,obs})$. This requires solving
\eqref{eq:lasso} at each sample point, and is clearly infeasible. We might say that this is a direct
way to condition the {\em push-forward} distribution $\hat{\beta}_{\lambda,*}(F_0)$ on the selection event. 

Our approach here relies on what is essentially the {\em pull-back} of the conditional distribution itself. 
In concrete terms, we construct an explicit inverse map to problems like \eqref{eq:lasso} on the selection
event. Then, instead of sampling \eqref{eq:inference:problem}, we
realize the law $F_0^*$ by sampling $\hbeta_{\lambda}$ together with some auxillary variables
on the conditioning event.
After aquiring these samples, we can use the inverse map to reconstruct $y$ that is essentially
distributed according to the law $F_0^*$. 
The construction for \eqref{eq:lasso} is described in Section \ref{sec:gaussian:lasso}. 

We will see that
the pull-back construction is particularly simple for the randomized lasso \eqref{eq:random:lasso}.
The randomization was inspired by the differential privacy literature \cite{reusable_holdout}, though
we will see below that it is already very similar to data splitting.
Besides making the pull-back simple to compute, we advocate the use of randomization in the query
stage in that conditional inference in the second stage is often more powerful if the selection stage 
is carried out with randomization than without \cite{optimal_inference,randomized_response}. Increasing
the scale of the randomization inevitably deteriorates the model selection quality, similar to using a smaller 
training set to choose a model in the context of data splitting. However, 
even a small amount of randomization has empirically shown a fairly noticable increase in power.
Another benefits of randomization include a form of robustness to rare selection events in the selective CLT of 
\citep{randomized_response}.

Inference for linear functionals is slightly more complicated than the simple goodness-of-fit test. 
To conclude this section, let us describe the form of the pull-back construction when we
have solved \eqref{eq:random:lasso} and observed the active set and signs $(E,z_{E})$ (dropping the $`obs'$ notation).
In principle, there is no restriction on the choice of models. For simplicity, we posit a parametric normal model that is based 
on a subset of variables $\bar{E}$:
$$
{\cal M}_{\bar{E}} = \left\{N(X_{\bar{E}}\beta_{\bar{E}}, \sigma^2_{\bar{E}}I): \beta_{\bar{E}} \in \real^{\bar{E}}\right\}
$$
with $\sigma^2_{\bar{E}}$ known. Often, the data analyst may opt to take $\bar{E}=E_{obs}$, though
this is not strictly necessary. Based on $(E,z_{E})$ she might consult the relevant
literature (or data independent of $y$) and choose to include or delete some variables from the set $E$. 
When the variance is unknown, one can either plug in a consistent estimate of $\sigma^2_E$ 
(c.f.~Lemma 14 of \cite{randomized_response}) or include the parameter $\sigma^2_{\bar{E}}$ in the model
and solve instead the square-root Lasso \cite{sqrt_lasso}.
If a parametric model is not appropriate then one might use a normal approximation for 
the pair $X^Ty$, and rewrite the {\em selection event} in terms of this statistic as opposed to just $y$. 
Such an approach is considered in \cite{randomized_response}, appealing to the selective CLT.

Having fixed model ${\cal M}_{\bar{E}}$, and knowing the density $g$ for the randomization $\omega$, suppose the
data analyst now wants to test $H_0:\beta_{j|\bar{E}}$ under the assumption $F \in {\cal M}_{\bar{E}}$.
Let $P_{\bar{E} \setminus j}$ denote orthogonal projection onto $\text{col}(X_{\bar{E}\setminus j})$.
In the normal $z$-test for $\beta_{j|\bar{E}}$, we essentially consider the distribution of
$$
r_j = y - \mu_{\bar{E} \setminus j}, \quad \mu_{\bar{E} \setminus j}=P_{\bar{E} \setminus j}y.
$$
$r_j$ can be recognized as the residual from the model with variables $\bar{E} \setminus j$. 
Selective tests for $\beta_{j|\bar{E}}$ is essentially considering the law of $r_j$
under the conditional distributions. 
An application of our main result Theorem \ref{thm:change:measure} implies that this law
(with the appropriate augmentation variables) has density proportional to
\begin{equation}
\label{eq:example:lasso}
(r_j, \hbeta_E, \hz_{-E}) \mapsto \exp \left(-\frac{1}{2 \sigma^2_{\bar{E}}} \|\mu_{\bar{E} \setminus j} + r_j\|^2_2 \right) 
\cdot g \left(X^TX_E\hbeta_E + \lambda \cdot \begin{pmatrix} \hz_E \\ \hz_{-E} \end{pmatrix} + \epsilon \cdot \begin{pmatrix} \hbeta_E \\ 0 \end{pmatrix}  \right)
\end{equation}
supported on 
$$\text{col}(P_{\bar{E} \setminus j})^{\perp} \times \left\{(\hbeta_E,\hz_{-E}): \text{sign}(\hbeta_E)=\hz_E, \|\hz_{-E}\|_{\infty} \leq 1 \right\}.$$ 
The augmentation variables $(\hbeta_E, \hz_{-E})$ are variables
involved in the optimization problem \eqref{eq:random:lasso}. 
Their support ensures we stay on the selection event of interest.
In fact, their support is recognizable as a subset of
$$
\left\{(\hbeta,\hz): \hz \in \partial (\| \cdot \|_1)(\hbeta) \right\},
$$
with $\partial$ denoting the subdifferential with respect to $\hbeta$. In particular, this is the 
subset for which the non-zero coefficients of $\hbeta$ are $E$ with signs $\hz_E$. 

In this simple case, the construction of the inverse map is remarkably simple.
If we write out the KKT condition for the optimization problem \eqref{eq:random:lasso},
we have
$$
X^TX\hbeta - X^Ty + \lambda \cdot \hz - \omega + \epsilon \cdot \hbeta = 0
$$
which can be rewritten as
$$
\omega = X^TX\hbeta - X^Ty + \lambda \hz + \epsilon \hbeta.
$$
where $(\hbeta, \hz)$ are the optimization variables and the subgradient of \eqref{eq:random:lasso}.
In this sense, we have inverted the KKT conditions on the selection event of interest, explicitly
parameterizing the pairs $(y,\omega)$ that yield active set and signs $(E,z_E)$. 

We call this
density the pull-back as this is recognizable precisely as the differential geometric pull-back of a 
differential form on the interior of the selection event, a subset of the manifold $\real^n \times \real^p$.
As $g$ is assumed to have a density, the boundary of the selection event is ignorable.
To those readers who do not recognize the term pull-back, 
Theorem \ref{thm:change:measure} essentially just applies a standard change of variables to realize a conditional law for $(y,\omega)$ via
a law for $(r_j,\beta_E,\hz_{-E})$. Note that the unselective law of $r_j$ is that of $y$ projected onto
the residual space form model $\bar{E} \setminus j$.

The reader may ask what we have gained through this exercise. By constructing an
inverse map for $(y, \omega)$ on the selection event, we do not have to check whether the pair
$(y,\omega)=(y,X^TX\hbeta-X^Ty+\lambda \cdot \hz + \epsilon \cdot \hbeta)$ satisfy the KKT conditions
for the pair $(E,z_E)$. That is, we have
removed the difficulty of conditioning from the problem. We are still left with a sampling problem, though
we have an explicit density as well as a fairly simple support. Further, the density \eqref{eq:example:lasso}
is log-concave in $(r_j,\beta_E, z_{-E})$ with a simple support hence methods such as projected Langevin \cite{bubeck_langevin} may
be run efficiently as the projection operator is cheap to compute. Section \ref{sec:implementation} discusses several
other examples.

Finally, and most importantly,
the reader who may also be the data scientist may ask what to do with samples from this density in order
to carry out a test of $H_0:\beta_{j|\bar{E}}(F)=0$.
Let $\hat{\beta}_{j|\bar{E}}:\real^n \rightarrow \real$ denote the
map that computes OLS coefficient $j$ in the model with variables $\bar{E}$. Then, given a
sufficiently large sample $(r_{j,b}, \beta_{E,b}, z_{-E,b})_{b=1}^B$ from density \eqref{eq:example:lasso},
the data analyst will compare the empirical distribution
of $$\left(\hat{\beta}_{j|\bar{E}}(\mu_{\bar{E} \setminus j} + r_{j,b})\right)_{b=1}^B
=\left(\hat{\beta}_{j|\bar{E}}(r_{j,b})\right)_{b=1}^B 
$$
to the observed value $\hat{\beta}_{j|\bar{E}}(y)$. 
Selective confidence intervals can be constructed by tilting the empirical distribution, though if the
true parameter is far from 0, then a reference distribution other than \eqref{eq:example:lasso} is
perhaps more appropriate. One might try replacing $r_{j,\perp}$ in the density above with $r_{j,\perp} - \bar{\beta}_{j|E}$
where $\bar{\beta}_{j|E}$ is an approximate selective MLE or pseudo MLE \cite{selective_bayesian}. We do
not pursue this further here, leaving this for future work. A data analyst concerned about slow mixing
in a given MCMC scheme to draw from \eqref{eq:example:lasso} may take some reversible MCMC algorithm
and carry out the exact tests described in \cite{besag_clifford}.

\subsection{General approach}

Having described what might be the canonical example, the LASSO with a parametric
Gaussian model for inference, we now lay out our general approach.

We consider a randomized version of the optimization problem in \eqref{eq:inference:problem} formulated as
\begin{equation}
\label{eq:optimization:problem:rand}
\hat{\beta}(S,\omega) = \argmin_{\beta \in \real^p} \ell(\beta;S) + {\cal P}(\beta) -\omega^T\beta + \frac{\epsilon}{2} \|\beta\|^2_2
\end{equation}
where $\ell$ is some smooth loss involving the data, ${\cal P}$ is some
structure inducing convex function, $\epsilon > 0$ is some small parameter that is sometimes necessary in order to assure the program has a solution
and $\omega \sim G$ is a randomization chosen by the data analyst. 
Our main goal is to sample
\begin{equation}
\label{eq:inference:problem:rand}
(S,\omega) | \hat{\beta}(S,\omega) \in A
\end{equation}
Our reasons for considering the problem \eqref{eq:optimization:problem:rand} rather than an unrandomized problem are described in the LASSO example above
and also hold for the optimization and sampling problems.

Namely, we expect an increase in power in the second stage following even a small randomization.
Further, as in the LASSO example above, this randomization
often allows us to cast the sampling problem as sampling
from a distribution on a space that is much simpler than if we had not 
randomized. That is, the law induced by the {\em pull-back} measure in a randomized program is often supported on a simpler region, as opposed to the push-forward measure of the non-randomized program. For inference in the non-randomized case, one might take the approach
of sending the scale of randomization to 0, though we do not pursue this here.

Another feature of the sampling problem related to \eqref{eq:optimization:problem:rand} is that the sampler
somewhat decouples the statistical model from the optimization variables. In this
sense, our main result provides ways to sample in a model-agnostic fashion: data analysts
can supply their own model ${\cal M}$, resulting in selective model ${\cal M}^*$. Of course,
for inference in ${\cal M}^*$, the analyst may have to use other techniques to reduce their problem
to sampling from a particular distribution in ${\cal M}^*$. In the example above, standard exponential
family techniques were used to eliminate nuisance parameters $\beta_{\bar{E} \setminus j|\bar{E}}(F)$ and
the problem was reduced to sampling from only one distribution, constructed
by conditioning a distribution in ${\cal M}^*$ on the sufficient statistic corresponding to the nuisance parameters.

By decoupling, we do not mean statistical independence in any sense. We mean that
each distribution in the corresponding selective model is supported on
$\Omega \times {\cal C}$ where $\Omega$ is the original
probability space for our data $S$ and ${\cal C}$ is a set of optimization variables related to the
structure inducing function ${\cal P}$. Formally, we should note that in the pull-back variables, the selective model
describes the distribution of tuples $(s,\beta,z)$ rather than the original probability space $\Omega \times \real^p$.
Hence, our sampler produces tuples $(s,\beta,z)$ rather than pairs $(s,\omega)$ though $\omega$ can always be reconstructed
via the map $\omega=\nabla \ell(\beta;s) + z + \epsilon \cdot \beta$.

The two most common examples of interest in statistical learning are 
$$
{\cal P}(\beta) = h_K(\beta) = \sup_{\nu \in K} \nu^T\beta
$$
for some convex $K \ni 0$, i.e. a seminorm. In this case, ${\cal C}$ is typically a subset of
the normal bundle of $K$ (c.f. \cite{rfg,schneider} )
$$
\left\{(\beta, z): \beta \in N_zK \right\}
$$
where $N_zK$ is the normal cone of $K$ at $z$, polar to the support cone of $K$ at $z$.
The other common example is a constraint on a seminorm, i.e.
\begin{equation}
\label{eq:seminorm}
{\cal P}(\beta) = I_{K}(\beta) = \begin{cases} 0 & \beta \in K \\
  \infty & \beta \not \in K \end{cases}
\end{equation}
where $K=\left\{b:\|b\| \leq 1\right\}$ for some seminorm $\|\cdot\|$.
In this case ${\cal C}$ is again typically a subset of $N(K)$ of the form
\begin{equation}
\label{eq:constraint}
\left\{(\beta,z): z \in N_{\beta}K \right\}.
\end{equation}

Such sets arise naturally from the KKT conditions of 
\eqref{eq:optimization:problem:rand}:
$$
\omega = \nabla \ell(\beta;S) + z + \epsilon \cdot \beta.
$$
Queries related to active sets of variables in the case of the LASSO or group LASSO or the rank of matrices in the
case of the nuclear norm correspond to smooth subsets of the corresponding normal bundles. When the structure
inducing penalty is polyhedral in nature, the sets ${\cal C}$ are typically polyhedral. For norms
with curved unit balls this is no longer the case, and curvature comes in to play. We treat the group LASSO
as a canonical example of this in Section \ref{sec:curvature}.

\subsection{Related work}
Most of the theoretical work on high-dimensional data focuses on 
consistency, either the consistency of solutions \cite{negahban_unified_2010,
van2008high} or the consistency of the models \cite{wainwright2009sharp, 
lasso_consistency}.

In the post selection literature, \cite{posi} proposed the PoSI approach,
which reduce the problem to a simultaneous inference problem. Because of
the simultaneity, it prevents data snooping from any selection procedure,
but also results in more conservative inference. In addition, the PoSI
method has extremely high computational cost, and is only applicable when
the dimension $p < 30$ or for very sparse models. 
The authors \cite{multi_split} proposed a method for computing
p-values that controls false discovery rate (FDR) among all variables. 
The knockoff filter of \cite{knockoffs} provides similar control of FDR for all variables
in the full model.
What distinguishes the conditional approach from these simulataneous approaches
is that the hypotheses tested, or parameters for which intervals are formed, in selective inference are chosen as a function
of the data. Hence, the methods of inference
are not always directly comparable.


\subsection{Outline of paper}

We propose the main Theorem \ref{thm:change:measure} of this paper in Section \ref{sec:randomized:convex:programs}, putting forth the sampling density conditional on a selection event. The highlight of this theorem is the reparametrization map that allows to reconstruct the randomization as a function of optimization variables and data and allows us to compute
the required conditional density explicitly, up to a normalizing constant. 
This is followed by a variety of examples of convex optimization programs. The support of the selective density in the first set of examples- LASSO and variants in Section \ref{sec:lasso:variants}, graphical models in Section \ref{sec:graphical:models}, forward stepwise in Section \ref{sec:kac:rice:poly} can be described by polyhedral geometry. We follow this up with a section describing more complex problems where the Jacobian involves a curvature component, the group lasso illustrated as the prototypical example. We extend the selective sampler to the case in which the data analyst
considers multiple views, or queries, of the data in Section \ref{sec:multiple:views}.
A selective version of Fisher's exact test in which the data analyst chooses the sufficient statistics based on the data is described in Section \ref{sec:goodness:of:fit}.    Finally, we advocate the projected Langevin  sampling technique \cite{bubeck_langevin} in Section \ref{sec:implementation} to sample from a log-concave selective density as described in earlier examples. Each update in such an implementation involves a projection onto a set of constraints induced by the selection event. The computational cost of each step is often minimal due to a much simpler constraint region using our reparametrization.

\section{Inverting the optimization map}

In this section, we consider the general problem of constructing
an explicit inverse to the solution of a convex problem. In this section,
we focus on convex problems without additional randomization, deferring
randomization to Section \ref{sec:randomized:convex:programs}.

Consider a statistical learning problem of the form
\begin{equation}
\label{eq:canonical:program0}
\minimize_{\beta \in \real^p} \ell(\beta;S) + {\cal P}(\beta), \qquad S \sim F,
\end{equation}
where $F$ is some
distribution in some model ${\cal M}$ and
${\cal P}$ is some structure inducing convex function, typically of the form
\eqref{eq:seminorm} or \eqref{eq:constraint}.
Cone constraints are also easily handled.

The subgradient equation for such a problem at a solution $\hat{\beta}(S)$ reads
\begin{equation}
\label{eq:canonical:subgrad}
0 \in \hat{\alpha}(S) + \hat{z}(S)
\end{equation}
with
\begin{equation}
\begin{aligned}
(S, \hat{\beta}(S), \hat{\alpha}(S), \hat{z}(S)) \in \solutionset^F(\ell, {\cal P}),
\end{aligned}
\end{equation}
where
\begin{equation}
\label{eq:canonical:set}
\begin{aligned}
\solutionset^F(\ell, {\cal P}) &\overset{\text{def}}{=} \biggl \{(s, \beta, \alpha, z): \\
& \qquad s \in \text{supp}(F), \\
& \qquad \ell(\beta;s) < \infty, \\
& \qquad \alpha \in \partial \ell(\beta; s),\\
& \qquad {\cal P}(\beta) < \infty, \\
& \qquad z \in \partial {\cal P}(\beta) \biggr\}.
\end{aligned}
\end{equation}
Above, and throughout, $\partial$ and $\nabla$ will denote subdifferentials
and derivatives with respect to $\beta$ unless otherwise noted.

The set $\solutionset^F(\ell, {\cal P})$ can be described by the 
``base space'' $\text{supp}(F)$ and ``fibers''
$$
\left\{(\beta,\alpha,z): \ell(\beta;s) < \infty, \alpha \in \partial \ell(\beta;s), \mathcal{P}(\beta)<\infty, z \in \partial {\cal P}(\beta) \right\}.
$$

We call the map
\begin{equation}
\label{eq:opt:map}
s \overset{\hat{\theta}}{\mapsto} (s, \hat{\beta}(s), \hat{\alpha}(s), \hat{z}(s)) \in \solutionset^F(\ell, {\cal P})
\end{equation}
the {\em optimization map}. 
In practice, given data $S$,
 a computer solves the problem, i.e.~produces a point in 
$\hat{\theta}(S) \subset \solutionset^F(\ell, {\cal P})$. Specifically,
a solver produces a point in 
the range of the optimization map:
\begin{equation}
\solutionset_{0}^F(\ell,{\cal P}) = \left\{(s,\beta,\alpha,z) \in \solutionset^F(\ell,{\cal P}): \alpha+z=0 \right\}.
\end{equation}
Formally speaking, the program \eqref{eq:canonical:program0}
may have no solutions. On this set
$\hat{\theta}(s) = \emptyset$.
In all of our examples  except the dual problem considered in Section \ref{sec:qp}, we will assume enough so that our convex programs 
have unique solutions, when they have any.

The selection events \cite{exact_lasso,optimal_inference} we are most interested in are typically of the form
\begin{equation}
\label{eq:canonical:event}
\solutionset^F_{\cal B}(\ell, {\cal P}) = \left\{(s,\beta,\alpha,z) \in \solutionset^F_0(\ell, {\cal P}): (\beta,\alpha,z) \in {\cal B}(s) \right\}
\end{equation}
for some nice set-valued function ${\cal B}(s)$ which could be specified
by the zero-set of a function  $h^{\cal B}$:
$$
(\beta,\alpha,z) \in {\cal B}(s) \iff h^{\cal B}(s,\beta,\alpha,z)=0.
$$
In all examples
below except basis pursuit in Section \ref{sec:basis:pursuit}, ${\cal B}(s)={\cal B}$ does not depend on 
$s$. 
This set also has the form of a bundle with base space $\text{supp}(F)$
and fibers
$$
\left\{(\beta,\alpha,z) \in {\cal B}(s): \alpha + z = 0 \right\}.
$$

We call $\solutionset^F_{\cal B}(\ell, {\cal P})$
{\em parametrizable} 
if there exists a measurable
parametrization
$\psi$ defined on some domain
$D$ with range $\solutionset^F_{\cal B}(\ell, {\cal P})$.
The map $\psi$ is typically constructed
to be an inverse of the optimization map $\hat{\theta}$ on 
$\solutionset^F_{\cal B}(\ell, {\cal P})$. 

Given a parameterization, we will typically
construct a change of measure using $\phi= \pi_{\cal B} \circ \psi $ to simplify sampling from
$\left\{s: \pi_{\cal B}^{-1}(s) \neq \emptyset \right\}$
by sampling from $\solutionset^F_{\cal B}(\ell, {\cal P})$ itself,
where $\pi_{\cal B}$ is the projection onto the base
of $\solutionset^F_{\cal B}(\ell, {\cal P})$.
Our construction is similar to what \cite{zhou_montecarlo_lasso} called estimator augmentation.
By construction, then, $\pi_{\cal B}(s,\beta,\alpha,z)$ is such that
$$\hat{\theta} \circ \pi_{\cal B} = \text{id}_{|\solutionset^F_{\cal B}(\ell, {\cal P})}$$
with $\text{id}$ the identity map. 

Transforming the probability space and constructing new data vectors that solve \eqref{eq:canonical:program0} as functions of optimization problems is key to our approach of sampling. 
The explicit parametrization changes with each problem. Typically, in selective inference examples, we condition on some function of 
$\hat{\beta}$, perhaps its support $E$ and possibly the signs $z_E$ of the non-zero coefficients. 
More generally, might condition on something besides $(E,z_E)$ which we might denote by $q$.
The selection event, i.e. the quantity we condition on determines the set of constraints $\mathcal{B}_q$.
We shall denote the parametrization map in our problems as $\psi_{q}$, so often we will write $\psi_{(E,z_E)}$. The map $\psi_{q}$ always produces a point $(s,\beta,\alpha,z)\in \solutionset^F_{{\cal B}_q}(\ell, {\cal P})$, while the domain of this parametrization can vary with the problem at hand, that is 
$$\psi_{q}: D_{q} \to \solutionset^F_{{\cal B}_q}(\ell, {\cal P}),$$
for a problem specific domain $D_{q}$.


\subsection{Inverting the MLE}
\label{sec:mle}

We now begin to describe our approach to inverting
the optimization map.
Our first example is classical: the density of the MLE of 
the natural parameters of an exponential family. The formula
is not new, going back at least to Fisher \citep{barndorff_nielsen,efron_hinkley}. Nevertheless, it serves to illustrate the general approach
we take for the general cases later.

Our loss function is
$$
\ell(\beta;S) = \Lambda(\beta) - \beta^TS
$$
where 
$$
e^{\Lambda(\beta)} = \Ee_{F_0}[e^{\beta^TS}]
$$
is the moment generating function in the exponential family with reference measure $F_0$ and sufficient statistic $S$ with
$F_{\beta}$  the law of $S$ above and in what follows. We also 
assume that $F_0$ has a density $f_0$ with respect to Lebesgue measure.

As we are computing
the MLE, our penalty function is ${\cal P}(\beta) \equiv 0$.
The convex program we solve is
\begin{equation}
\label{eq:mle:problem}
\minimize_{\beta \in \real^p} \Lambda(\beta) - \beta^TS.
\end{equation}
The KKT conditions or subgradient equation here is just the usual score equation
$$
\nabla \Lambda(\hat{\beta}(S)) = S
$$
where
$$
\nabla \Lambda(\beta) = \int_{\real^p} s f_{\beta}(ds).
$$

Our parameterization of $\mathcal{S}_0^F(\ell,\mathcal{P}=0)$ is
$$
\psi(\beta) = (\nabla \Lambda(\beta), \beta, 0, 0)
$$
which can be interpreted as reconstructing $S$ given $\beta$.

Standard multivariate calculus then tells us that the density
of $\hat{\beta}(S)$ under the distribution $F_{\beta_0}$ is
\begin{equation}
\label{eq:mle:density}
e^{\nabla \Lambda(\beta)^T\beta_0 - \Lambda(\beta_0)} 
\det(\nabla^2 \Lambda(\beta)) f_0(\nabla \Lambda(\beta)).
\end{equation}
where $\nabla^2 \Lambda(\beta)$ is the observed information
\citep{efron_hinkley}.

If $\beta_0$ is the true parameter, then, at the cost of changing the reference measure by a factor of
$\exp(s^T\hat{\beta}(s) - \Lambda(\hat{\beta}(s)) \det(\nabla^2 \Lambda(\beta))^{1/2}$ yielding
a new Lebesgue density $h_0$,
 we can
rewrite this
as 
\begin{equation}
\label{eq:mle:density:saddlepoint}
e^{\Lambda(\beta) - \Lambda(\beta_0) + \nabla \Lambda(\beta)^T(\beta_0 - \beta) } \det(\nabla^2 \Lambda(\beta))^{1/2} h_0(\nabla \Lambda(\beta)).
\end{equation}
As pointed out in \citep{barndorff_nielsen} ignoring the term $h_0(\nabla \Lambda(\beta))$ (which
is the constant 1 in the Gaussian case) yields the usual
saddle-point approximation to the density of the MLE, up to the constant
of integration. The exponential above can be rewritten as
$$
-\frac{1}{2}(\beta-\beta_0)^T \nabla^2 \Lambda(\beta) (\beta-\beta_0) + R(\beta;\beta_0).
$$
Note that even ignoring the remainder, this is quadratic in $\beta_0$ the 
parameter, and not $\beta$ the variable of integration in the density.

In principle,
nothing above really relies on the exponential family structure for the model, 
though it does rely somewhat on the fact that the loss
we use came from an exponential family. It relies on this
in that we use the form of the loss to reconstruct data $S$ from
optimization variables $\beta$. This is similar to what we see 
in the LASSO example below.

Nevertheless, the same argument above shows that if we solve the
program \eqref{eq:mle:problem} then, so long
as $S \sim F$ has a Lebesgue density the
density of $\hat{\beta}(S)$ is
\begin{equation}
\label{eq:mle:misspec}
f(\nabla \Lambda(\beta)) \cdot \left|\det(\nabla^2 \Lambda(\beta))\right|.
\end{equation}
In this sense, the above display provides an exact recipe to 
compute the density of the MLE under model misspecification. This
is somewhat similar to the general approach taken in \cite{exact_density_mle}, though we are considering this only in a very restricted setting.


\subsection{Pull-Back of the LASSO with fixed design matrix}
\label{sec:gaussian:lasso}

As a second example of a pull-back, we look at the canonical example in the class of regularized convex optimization problems: the LASSO \cite{tibs_lasso}.
The LASSO program is defined for each $(X, y, \lambda) \in \real^{n \times p} \times \real^n \times (0,\infty)$ as
\begin{equation}
\label{eq:lasso:program}
\minimize_{\beta \in \real^p} \frac{1}{2} \|y - X\beta\|^2_2 + \lambda \|\beta\|_1.
\end{equation}

In this example, $X$ is considered fixed and $F={\cal L}(P_Cy|X)$ where 
$P_C$ is projection onto $\text{col}(X)$. The law
$F$ is supported on $\text{col}(X)$ because the optimization map depends only on
$P_{C}y$. In our general notation, we can take
\begin{equation}
\label{eq:canonical:lasso}
\begin{aligned}
S &= P_Cy \\
\ell(\beta;y) &= \frac{1}{2} \|y-X\beta\|^2_2 \\
{\cal P}(\beta) &= \lambda \|\beta\|_1 \\
\solutionset^F(\ell, {\cal P}) =
\Big\{(& y,\beta,\alpha,z): P_Cy=y, \beta \in \real^p, \\
&\alpha=X^TX\beta - X^Ty, z \in \partial (\lambda \| \cdot \|_1)(\beta) \Big\}.
\end{aligned}
\end{equation}

The familiar subgradient equations of the LASSO are
\begin{equation}
\label{eq:lasso:KKT}
-\hat{\alpha}(y) = X^T(y-X\hat{\beta}(y)) = \hat{z}(y), \qquad \hat{z}(y) \in \partial (\lambda \| \cdot \|_1)(\hat\beta (y)).
\end{equation}
Note also that \eqref{eq:lasso:KKT} contains the implicit constraint $\hat{z}(y) \in \text{row}(X)$ where $\text{row}(X)$ 
is the rowspace of $X$.
This can be seen from the structure of 
$$
\begin{aligned}
\solutionset^F_0(\ell,{\cal P}) &= \biggl\{(y,\beta,\alpha,z) \in \solutionset^F(\ell, {\cal P}): X^T(X\beta-y) + z = 0 \biggr\}
\end{aligned}
$$
as each point in $\solutionset^F_0(\ell, {\cal P})$ has $z \in \text{row}(X)$.

In \cite{exact_lasso}, the authors assume general position so the
map $\hat{\theta}$ is single valued \cite{uniqueness_lasso}. The
authors
then considered the active set and signs of the LASSO solution,
conditioning on their value $(E,z_E)$ and the design $X$.
This is the event
\begin{equation}
\label{eq:lasso:event}
\begin{aligned}
\left\{y : (y, \hat{\beta}(y), \hat{\alpha}(y), \hat{z}(y))\in\mathcal{S}_0^F(\ell, \mathcal{P}),  \text{diag}(z_E)\hat{\beta}_E(y)>0, \hat{\beta}_{-E}(y) = 0 \right \}.
\end{aligned}
\end{equation}
Note that this event is equivalent to 
$$
\begin{aligned}
\hat{\theta}(y) \in \left\{(y,\beta,\alpha,z) \in \solutionset^F_0(\ell, {\cal P}): \text{diag}(z_E)\beta_E > 0, \beta_{-E} = 0 \right \}. 
\end{aligned}
$$

\noindent This is our canonical example of a parameterizable set where
$$
{\cal B} = {\cal B}_{(E,z_E)} = \left\{(\beta,\alpha,z): \text{diag}(z_E)\beta_E > 0, \beta_{-E} = 0
, z_E = \lambda\ \textrm{sign}(\beta_E), \|z_{-E}\|_{\infty} \leq \lambda \right \}.
$$
The above constraints on $\beta$ induce further restrictions on the range of the solver
\begin{equation*}
\begin{aligned}
	\mathcal{S}_{\mathcal{B}_{(E,z_E)}}^F(\ell,\mathcal{P}) 
	= \{ (& s,\beta,\alpha,z)  \in \mathcal{S}_0^F(\ell, \mathcal{P}) : \beta\in\mathcal{B}_{(E, z_E)}\} \\
	= \{ (& s, \beta, \alpha, z) : \text{diag}(z_E)\beta_E > 0, \beta_{-E} = 0,
	\\ &  \alpha=-X^T(s-X\beta)=-z, z_{E}=\lambda\ \textrm{sign}(\beta_E), \|z_{-E}\|_{\infty}<\lambda \}.
\end{aligned}
\end{equation*}

The authors in \cite{exact_lasso} then carry out selective inference for linear functionals $\eta^T\mu$ in the {\em 
saturated model}
\begin{equation}
\label{eq:param:lasso:model}
{\cal M} = \left\{N(\mu, \sigma^2): \mu \in \real^n\right\}
\end{equation}
with $\sigma^2 > 0$ considered known. In this context, selective inference corresponds to taking each
$F \in {\cal M}$ and conditioning it on the event \eqref{eq:lasso:event} which can be rewritten as
$$
\pi_S \left(\mathcal{S}_{\mathcal{B}_{(E,z_E)}}^F(\ell,\mathcal{P})\right)
$$
where $\pi_S$ denotes projection onto the data coordinate. The resulting {\em selective model}
$$
{\cal M}^* = \left\{F^*: \frac{dF^*}{dF}(y) \propto \begin{cases} 1 & y \in \pi_S \left(\mathcal{S}_{\mathcal{B}_{(E,z_E)}}^F(\ell,\mathcal{P})\right) \\ 0 & \text{otherwise.} \end{cases} \right\}
$$
is an exponential family and sampling is generally not necessary in the saturated model as valid inference
typically requires conditioning on sufficient statistic related to nuisance parameters as described in 
\cite{exact_lasso,optimal_inference}.

Nevertheless, if sampling were necessary, a naive accept reject sampling scheme
for inference in this setting
draws vectors $y$ on $\real^n$ according to $N(\mu,\sigma^2)$ solves
the LASSO with the triple $(y,X,\lambda)$ and checks whether the result lies in
\eqref{eq:lasso:event}. Due to the nature of the LASSO, this check 
can be reduced to verifying whether $y$ satisfies a set of affine inequalities \cite{exact_lasso}.

\subsubsection{Parametrization and pull-back of the LASSO}

What if we did not have to check these affine inequalities
in our sampling scheme? This is the essence of what we propose in this work. We will
ultimately see that our approach is essentially
equivalent to that of \cite{exact_lasso} but 
the probability space of our sampler is different. The approach we take
is similar to \citep{zhou_montecarlo_lasso}.
We will see later
that, after randomization, the parameter space is generically simpler
than if we had not randomized.

Our first example of a parameterization is
\begin{equation}
\label{eq:embed:map}
\psi_{(E,z_E)}(\beta, z) = (X\beta + (X^T)^{\dagger} z, \beta, -z, z)
\end{equation}
with domain
\begin{equation}
\label{eq:embed:domain}
\begin{aligned}
D_{(E,z_E)}  = \{(& \beta, z): \text{diag}(z_E)\beta_E >0, \beta_{-E}=0, \\
 & z \in \text{row}(X), z_{E}=\lambda\ \textrm{sign}(\beta_E), \|z_{-E}\|_{\infty} \leq \lambda  \}
\end{aligned}
\end{equation}
and corresponding data reconstruction map $\phi_{(E,z_E)}:D_{(E,z_E)} \rightarrow \text{col}(X)$
defined by
$$
\phi_{(E,z_E)}(\beta, z) = X\beta + (X^T)^{\dagger} z.
$$

We see that $\psi_{(E,z_E)}$ is indeed a parameterization
of $\solutionset_0^F(\ell, {\cal P})$. Hence, on the range of 
$\phi_{(E,z_E)}$ we do not actually have to check the affine inequalities
of \citep{exact_lasso} as they are satisfied by construction.

The set $D_{(E,z_E)}$ is a subset of a $2p$-dimensional set but is in fact of dimension $\text{rank}(X)$ whenever $|E| < \text{rank}(X)$. Hence,
it has Lebesgue measure $0$ in $\real^{2p}$ but it is a subset of an affine space so that it inherits a Hausdorff measure ${\cal H}_{\text{rank}(X)}$.

In a formal sense, the map $\phi_{(E,z_E)}$ is the inverse of the optimization map.
\begin{lemma}
\label{lem:pullback:chart}
On the set of $y$ where that the
solution to \eqref{eq:lasso:program} is unique ($X$ being considered
fixed), the following equality holds
\begin{equation}
\hat{\theta}_{-(y, \alpha)} \circ \phi_{(E,z_E)} (\beta, z) = (\beta, z)\;\; \textnormal{  for }(\beta, z) \in D_{(E,z_E)},
\end{equation}
where $\hat{\theta}_{-(y,\alpha)}(y) = (\hat{\beta}(y), \hat{z}(y))$.
Hence, $\hat{\theta}_{-(y,\alpha)} \circ \phi_{(E,z_E)}$ is equivalent to the 
identity map on $D_{(E,z_E)}$.

More generally, the set \eqref{eq:lasso:event} is equal to 
$\phi_{(E, z_E)}(D_{(E,z_E)}).$
\end{lemma}
\begin{proof}
Direct verification of \eqref{eq:lasso:KKT}.
\end{proof}

Now choose an element in ${\cal M}$, i.e.~fix some $\mu \in \real^n$ and consider the $N(\mu,\sigma^2I)$ density. 
We define its {\em pull-back} to be
the law $F^*_{\mu}$  supported on $D_{(E,z_E)}$ with ${\cal H}_{\text{rank}(X)}$ density 
\begin{equation}
f^*_{\mu}(\beta, z) \propto \exp\left(-\frac{1}{2 \sigma^2} \|\phi_{(E,z_E)}(\beta, z)-\mu\|^2_2\right)
\cdot \left|J \phi_{(E,z_E)}(\beta, z)\right|,
\end{equation}
where $J\phi_{(E,z_E)}$ is the Jacobian of the projection of the
parameterization, which depends
only on $X$ in this case (and can be treated as a constant). We use the name pull-back as this
density is precisely the pull-back of the  measure with Lebesgue density $f_{\mu}$ under the reconstruction map $\phi_{(E,z_E)}$.

As each density has a pull-back, the model itself has a pull-back ${\cal M}^*$ comprising the pull-back of each density. The model
also inherits the exponential family structure of ${\cal M}$. Taking $f_0^*$ to be the reference measure, we can choose
the natural parameter to be $\mu/\sigma^2$ and sufficient statistic to be 
$\phi_{(E,z_E)}(\beta, z)$.

Finally, note that nothing about our construction so far has anything to do with the parametric model
\eqref{eq:param:lasso:model}. 
In fact, as long as $F$ has a Lebesgue density the pull-back is well-defined. 

\begin{lemma}
\label{lem:lasso:sign:conditional}
Suppose $F= {\cal L}(P_Cy|X)$ has density $f$ on $C=\text{col}(X)$
and the solution to
\eqref{eq:lasso:program} is $F$-a.s. unique. Then, the following equality holds
\begin{equation}
\begin{aligned}
&{\cal L}_{y\sim F}\left (P_Cy  \: \big\vert \:(y,\hat{\beta}(y), \hat{\alpha}(y),\hat{z}(y))\in\mathcal{S}_0^F(\ell, \mathcal{P}), \text{diag}(z_E)\hat{\beta}_E(y)>0, \hat{\beta}_{-E}(y) = 0, X \right) \\
&={\cal L}_{(\beta, z) \sim F^*}\left(\phi_{(E,z_E)}(\beta, z) | X\right),
\end{aligned}
\end{equation}
where $F^*$ is supported on $\text{relint}(D_{(E,z_E)})$ and has ${\cal H}_{\text{rank}(X)}$ density 
$$
f^*(\beta, z) \propto f \left( \phi_{(E,z_E)}(\beta, z) \right) \cdot \left|J \phi_{(E,z_E)}(\beta, z)\right|.
$$
\end{lemma}
\begin{proof}
Follows from Lemma \ref{lem:pullback:chart} and standard multivariate
calculus combined with the fact that the existence of the 
density $f$ implies that
$$
F ( \phi_{(E,z_E)}(D_{(E,z_E)} \setminus \text{relint}(D_{(E,z_E)}))) = 0.
$$
\end{proof}


\begin{remark}
One of the consequences of the above is that one can
draw response vectors $y$ that have the same active set and signs
from a density with relatively simple support when $\text{row}(X) = \real^p$. 
Hence, sampling 
IID from some density on this support allows one to use importance
sampling with the above explicit density as numerator in the importance weight.

When $\text{row}(X) \subsetneq \real^p$, the support is still somewhat
complex. For instance, the subgradients $z$ must be in the intersection of a face of the 
$\ell_{\infty}$ ball which may be a difficult set to describe. 
We will see that after randomization, 
this complexity often disappears.
\end{remark}

\begin{remark}
In theory, one might want to drop the uniqueness assumption above.
If uniqueness of the solution does not hold then the parameterization
is not injective and the event we condition on should be replaced with the
event $\hat{\theta}(y) \cap \mathcal{S}_{\cal B}^F(\ell,\mathcal{P}) \neq \emptyset$.

In this case, it may still be possible to 
derive a formula for the law of $P_Cy$
by invoking the co-area formula \cite{federer,diaconis_sampling_manifold}. We do not 
pursue this generalization here.
\end{remark}

In order to carry out selective inference for a linear functional $\eta^T\mu$, the authors of  \cite{exact_lasso} conditioned
on $\mathcal{P}_{\eta}^{\perp}y$. This is also possible in the pull-back model. 
Fix a linear subspace $L \subset \text{col}(X)$ considered to be the model subspace so that
$$
{\cal M}_L = \left\{N(\mu,\sigma^2 I): \mu \in L \right\}.
$$ 
For each linear function of interest $\eta \in L$, one generally must condition on $(\mathcal{P}_L-\mathcal{P}_{\eta})y$ to eliminate the
nuisance parameter $(\mathcal{P}_L-\mathcal{P}_{\eta})\mu$.
If $w$ is the observed value of $(\mathcal{P}_L-\mathcal{P}_{\eta})y$, a straightforward modification of the argument above shows
that the appropriate distribution is supported on the $\text{rank}(X) - \text{dim}(L) + 1$ dimensional set
\begin{equation*}
\label{eq:conditional:support}
D_{(E,z_E)}(w) = \{(\beta, z) \in D_{(E,z_E)}:  (\mathcal{P}_L-\mathcal{P}_{\eta})\phi_{(E,z_E)}(\beta, z)=w\}
\end{equation*}
with ${\cal H}_{\text{rank}(X)-\text{dim}(L)+1}$ density proportional to $f_{\mu}^*$.
In this case, even if $\text{row}(X)=\real^p$, the support is somewhat complex.


\section{Inverting the optimization map of a randomized convex program} \label{sec:randomized:convex:programs}

In this section, we describe how to invert the KKT conditions for our family of randomized convex programs.
As mentioned in the introduction, randomization comes with advantages like enhanced statistical power in the inference stage, as
well as a simplification of the support of the relevant reference distribution.
In what follows, all of our convex programs have random variables appearing
linearly in the subgradient, so that these random variables
can be reconstructed from optimization variables. The random
variable is one introduced by a data analyst through  additional randomization as described in \citep{reusable_holdout,randomized_response}.

Let $G$ be a distribution on $\real^p$ and $\epsilon \geq 0$ a 
small parameter. 
Given
a canonical problem specified by $(F,\ell,{\cal P})$ we define
its randomized version
as
\begin{equation}
\label{eq:canonical:random:program}
\minimize_{\beta \in \real^p} \ell(\beta;S) + {\cal P}(\beta)  -  \omega^T\beta+ \frac{\epsilon}{2} \|\beta\|^2_2, \qquad (S,\omega) \sim F\times G.
\end{equation}

A natural question to ask
at this point is: why randomize the program in the 
above way?
One of the inspirations for selective inference after randomization, the topic 
considered in \citep{randomized_response} are techniques used in differential
privacy \citep{reusable_holdout}. The other was the empirical finding
in \citep{optimal_inference}
that holding out some data before carrying out selective inference as in 
\cite{exact_lasso} also showed an improvement in selective power. 

This increase in selective power can be attributed to the fact
that choosing a model with a randomized response has more leftover information as defined in \cite{optimal_inference} after model selection than choosing a model with the original response. 

The choice to add $\frac{\epsilon}{2} \|\beta\|^2_2 - \omega^T\beta$ to the objective is not the only
reasonable choice. The main property we want of the perturbation 
$\Delta(\beta,\omega)$ is that $\omega=\phi(\beta, v)$ where
$(\beta, v) \in \partial \Delta(\beta,\omega)$. That is, given the value of 
$\beta$ and the subgradient $v$ we can reconstruct $\omega$. 

\begin{remark}
Addition of the term $\frac{\epsilon}{2} \|\beta\|^2_2$ to the objective
ensures that the set of $(s,\omega)$ such that \eqref{eq:canonical:random:program} has a solution contains the
set
$$
\{(s,\omega): \hat{\theta}(s) \neq \emptyset, \omega \in \text{supp}(G)\}.
$$
To see this, note that solving \eqref{eq:canonical:random:program} is equivalent to computing
$$
\hat{\beta}(s,\omega) = \text{prox}_{\frac{1}{\epsilon}(\ell(\cdot;s) + {\cal P}(\cdot))}(\omega / \epsilon).
$$
Whenever $\hat{\theta}(s) \neq \emptyset$, the objective $\ell(\cdot;s)+{\cal P}(\beta)$ is a closed proper convex function, hence its proximal mapping is well-defined and $\hat{\beta}(s,\omega)$ is well-defined.

If $\ell$ is strongly convex a.s.-$F$ and ${\cal P}$ is closed and proper then we can 
take $\epsilon=0$.
\end{remark}

\subsection{Data splitting as randomization} \label{sec:data:splitting}

The suggested randomization above may strike some readers as lacking motivation.
A randomization mechanism probably familiar to most readers is 
data splitting \citep{cox}, i.e.~randomly partitioning the dataset into
two pieces of size $(n_1,n-n_1)$. We will see that solving
a convex program after data splitting is essentially
an example of \eqref{eq:canonical:random:program}.

For a specific example, suppose we fit the 
graphical LASSO \cite{glasso} based on a data matrix
$X \in \real^{n \times p}$
but we first randomly split the data. That is,
we form
$$
S_1(g) = \frac{1}{n_1} X_1(g)^T
\left(I_{n_1} - \frac{1}{n_1} 1_{n_1}1_{n_1}^T \right)X_1(g),
$$
with  $g$ denoting
the random partition of the data into two groups and $X_1(g)$ denoting the
data in the first group.

Next, we solve the program
$$
\minimize_{\Theta: \Theta^T=\Theta, \Theta > 0} - \log \det (\Theta) + \text{Tr}(S_1(g) \Theta) + \lambda {\cal P}(\Theta)
$$
with the usual caveat that the penalty does not charge diagonal elements of $\Theta$.

This program is equivalent to solving the problem
$$
\minimize_{\Theta: \Theta^T=\Theta, \Theta > 0} - \log \det (\Theta) + \text{Tr}(S \Theta) + \text{Tr}((S_1(g)-S)\Theta) + \lambda {\cal P}(\Theta),
$$
where 
$$
S = \frac{1}{n} X^T\left(I - \frac{1}{n}1_n1_n^T \right)X.
$$
This is very close to \eqref{eq:canonical:random:program} with
$\omega = \omega(g) =  S- S_1(g)$ which is orthogonal to, but not necessarily
independent of the full covariance $S$.
As described in \cite{optimal_inference}, data splitting can be interpreted as
simply conditioning on the value 
$(S_1(g),g)$ as the only variation 
in ${\cal L}(S|S_1(g),g)$ is $S_2(g)$, the second stage data.

\subsection{The selective sampler}

In this section, we describe our main tool for inference after solving
a convex program randomized in the above fashion. We call this
tool the {\em selective sampler}.

A solver for the program \eqref{eq:canonical:random:program} produces a point in
\begin{equation}
\label{eq:canonical:randomized_set}
\begin{aligned}
\bar{\solutionset}^F(G, \epsilon,\ell,{\cal P}) &\overset{\text{def}}{=} \biggl \{(s, \omega, \beta, \alpha, z): \\
& \qquad s \in \text{supp}(F), \\
& \qquad \omega \in \text{supp}(G), \\
& \qquad \ell(\beta;s) < \infty, \\
& \qquad \alpha \in \partial \ell(\beta; s),\\
& \qquad {\cal P}(\beta) < \infty, \\
& \qquad z \in \partial {\cal P}(\beta) \biggr\}.
\end{aligned}
\end{equation}
More precisely, it produces a point in 
\begin{equation}
\bar{\solutionset}_0^F(G, \epsilon,\ell,{\cal P}) = \left \{(s, \omega, \beta, \alpha, z)\in \bar{\solutionset}^F(G, \epsilon,\ell,{\cal P}): \epsilon \cdot \beta + \alpha + z - \omega = 0 \right\}.
\end{equation}

Our statistical learning task is typically to infer something about $F$ in some
model ${\cal M}$. 
As we are free to choose $G$ however we want, a 
natural choice is to choose $G$ to have a Lebesgue density supported on all
of $\real^p$. In this case, inspection of the KKT conditions or subgradient equation
of \eqref{eq:canonical:random:program} read
\begin{equation}
\omega = \hat{\alpha}(s,\omega) + \hat{z}(s,\omega) + \epsilon \cdot \hat{\beta}(s,\omega).
\end{equation}
The subgradient equation determines a canonical
map
\begin{equation}
\label{eq:canonical:map}
\psi : \solutionset^F(\ell, {\cal P}) \rightarrow \bar{\solutionset}^F_0(G, \epsilon,\ell, {\cal P})
\end{equation}
defined naturally as
\begin{equation}
\psi(s,\beta,\alpha,z) = (s,  \alpha + z + \epsilon \cdot \beta,\beta,\alpha,z).
\end{equation}
As in the non-randomized case, we are typically interested in some selection events of the form 
\begin{equation*}
\bar{\mathcal{S}}_{\mathcal{B}}^F(G,\epsilon,\ell,\mathcal{P}) = \{(s,\omega,\beta,\alpha,z)\in \bar{\mathcal{S}}_0^F(G,\epsilon,\ell,\mathcal{P}): (\beta,\alpha,z)\in\mathcal{B}(s) \}.
\end{equation*}
The corresponding randomization reconstruction map is then
$$
\phi(\beta,\alpha,z)=\alpha + z + \epsilon \cdot \beta ,
$$
defined on a domain  
$$
\{(\beta, \alpha,z)\in {\cal B}(S):  \alpha + z + \epsilon \cdot \beta \in \text{supp}(G) \},
$$
for a suitable ${\cal B}(S)$, defined by the selection event.
It turns out that many selection events of interest are such that the restriction of $\psi$ to 
these events have a simple structure which allows for straightforward
description of the selective model.
The canonical example of conditioning on the set of active variables 
and signs of the LASSO as in \citep{exact_lasso} was described in Section \ref{sec:canonical:example}.


\begin{theorem}[Selective sampler]
\label{thm:change:measure}
Suppose that $\omega$ is independent of $S$ with distribution $G$ such that $\text{supp}(G) \subset \real^p$ has non-empty interior with Lebesgue density $g$ on $\text{supp}(G)$.
Then, the map
$$
\psi(s,\beta,\alpha,z) = (s, \epsilon \cdot \beta + \alpha + z,\beta,\alpha,z)=(s,\phi(\beta, \alpha,z),\beta,\alpha,z)
$$
restricted to 
$$
D_G = \left\{(s,\beta,\alpha,z) \in \solutionset^F(\ell,{\cal P}): 
\epsilon \cdot \beta + \alpha + z \in \text{supp}(G) \right\}
$$
is onto $\bar{\solutionset}^F_0(G, \epsilon,\ell,{\cal P})$.
Further, the law
\begin{equation}
\label{eq:change:measure}
\begin{aligned}
&{\cal L}_{F \times G}((S,\omega) | (S,\omega,\hat{\beta}(S,\omega), \hat{\alpha}(S,\omega), 
\hat{z}(S,\omega)) \in \bar{\mathcal{S}}_{\mathcal{B}}^F(G,\epsilon,\ell,\mathcal{P})) \\
&= {\cal L}((S, \epsilon \beta +\alpha+z) | (s,\beta, \alpha, z) \in D_G, (\beta,\alpha,z)\in {\cal B}(S))
\end{aligned}
\end{equation}
for suitable ${\cal B}(S)$
and $(S,\beta,\alpha,z)$ has density proportional to
\begin{equation}
\begin{aligned}
\label{eq:change:measure:density}
& f(s) \cdot g(\epsilon \cdot \beta +\alpha + z) \cdot \left|J\psi(s,\beta,\alpha,z)\right| \cdot 1_{D_G}(s,\beta,\alpha,z) \cdot 1_{\mathcal{B}(s)}(\beta,\alpha,z) \\
&=f(s) \cdot g(\phi(\beta,\alpha,z))\cdot  \left|\det(D_{(\beta,\alpha,z)} \phi)\right| \cdot 1_{D_G}(s,\beta,\alpha,z) \cdot 1_{\mathcal{B}(s)}(\beta,\alpha,z) 
\end{aligned}
\end{equation}
with the Jacobian denoting the derivative of the map $\psi$ with respect
to $(\beta,\alpha,z)$ on the fiber over $s$.
\end{theorem}

\begin{proof}
The fact that $\psi$ restricted to $D_G$ is onto follows from its construction.
Let
${\cal F}(s) = \left\{(\beta,\alpha,z): \alpha \in \partial \ell(\beta;s), z \in \partial 
{\cal P}(\beta) \right\}$
denote the fiber over $s$ and assume that ${\cal B}(s)={\cal B}$ does not
vary with $s$. In this case, in any local coordinates
on ${\cal F}(s)$ standard multivariate calculus
yields the density \eqref{eq:change:measure:density} as the
derivative of the map
$$
(s,\beta,\alpha, z) \mapsto (s, \alpha + z + \epsilon \cdot \beta )
$$
takes the form
$$
\begin{pmatrix}
I_{n\times n} & 0_{n\times p} \\
D_s \phi & D_{(\beta,\alpha,z)} \phi
\end{pmatrix}
$$
with determinant $\det(D_{(\beta,\alpha,z)} \phi) $.
The result then follows by integrating over $\text{supp}(F)$ with density $f(s)$.
If $S$ does not have a density  then \eqref{eq:change:measure}
can be derived via the Kac-Rice formula \cite{kac_rice,rfg}. 

If ${\cal B}$ depends on $s$ then the Kac-Rice formula may often be used to
derive the above density given the ${\cal B}$ can be described as
the zero set of some smooth function $h^{\cal B}$. As most of our examples are such that ${\cal B}$ does not
depend on $s$ we omit the details.
\end{proof}


\begin{remark}
The correct interpretation of \eqref{eq:change:measure} has
as conditioning event
$$
\bar{\theta}_{-(s,\omega)}(s,\omega) \cap {\cal B}(s) \neq \emptyset
$$
where $\bar{\theta}$ is the map
$$
\bar{\theta}(s,\omega)=(s,\omega, \hat{\beta}(s,\omega), \hat{\alpha}(s,\omega), \hat{z}(s,\omega)).
$$
\end{remark}

\begin{remark} When $\text{supp}(G)=\real^p$ then $D_G=\solutionset^F(\ell, {\cal P})$.
\end{remark}


\begin{remark}
Above, we have constructed $\omega$ as a function of $(s,\beta,\alpha,z)$.
This is similar to the LASSO case in Section \ref{sec:gaussian:lasso}
in that we construct new 
random variables out of optimization variables. Often, it is also possible to construct $s$ from $(\omega,\beta,\alpha,z)$ though there is no
canonical embedding unless we provide more structure to the map $\ell$.
If $\ell$ is an exponential family negative log-likelihood so that
$$
\ell(\beta;s) = \Lambda(\beta) - s^T\beta
$$
then this is certainly possible. However, in order to
have a change of measure result as in Theorem \ref{thm:change:measure}, the 
law of $S$ should have a density on $\real^n$.
\end{remark}


\begin{remark}
Often, we will want to condition on some functions of 
$S$. Theorem \ref{thm:change:measure} formally holds unchanged for any 
distribution supported on a lower dimensional subset of $S$. 
One simply replaces the law $F$ with the appropriate law supported on
a lower dimensional set. This device was used
in Section \ref{sec:canonical:example}. A further example of this is considered
in Appendix \ref{sec:conditional} below. 
\end{remark}


\begin{remark}
The theorem assumes $\omega$ is independent of the data $S$. It is straightforward
to see that a similar result holds if we replace throughout the density $g$ above with 
$K(\omega;s)$ a kernel for the conditional density of $\omega|s$.
\end{remark}



\section{Polyhedral examples} \label{sec:polyhedral}

We now begin describing several instances of the selective sampler. 
In this section, the penalties or constraints are polyhedral.
In this case, the reconstruction maps are typically
affine in the optimization
variables but may be non-trivial 
in the data.
Some of these examples were also considered in \cite{selective_sampler_nips}. 
We repeat them here, in more explicit detail, as concrete
examples of the selective sampler.

\subsection{LASSO with Gaussian errors and fixed design matrix}
\label{sec:random:lasso}

As is often the case, the LASSO serves as a canonical example. We denote the parametrization, based on the active set and signs as $\psi_{(E,z_E)}$ on domain $D_{(E,z_E)}$ and the reconstruction map for randomization $\omega$ as $\phi_{(E,z_E)}$ throughout. In all the below examples, we are implicitly
thinking of cases when $\text{supp}(G)=\real^p$ and $G$ has a Lebesgue density, the canonical example being $N(0, \sigma^2_{\omega} I_p)$.

This example was addressed in Section \ref{sec:canonical:example}, though we present it here in the general notation
developed so far.
The embedding in Theorem \ref{thm:change:measure}
plays the role of $\psi_{(E,z_E)}$ in the
parametric LASSO example. 
The randomized LASSO program \cite{randomized_response} 
with randomization $G$ 
is defined for each $(y, X, \lambda, \omega) \in \real^n \times \real^{n \times p} \times (0,\infty) \times \real^p$ as
\begin{equation} \label{eq:lasso:randomized:program}
\minimize_{\beta \in \real^p} \frac{1}{2} \|y - X\beta\|^2_2 + \frac{\epsilon}{2} \|\beta\|^2_2 - \omega^T \beta + \lambda \|\beta\|_1, \qquad y|X \sim F, \omega \sim G
\end{equation}
with $\omega$ independent of $(X,y)$.

When $G = \delta_0$, we recover the parametric LASSO with $\epsilon=0$ and
the parameteric Elastic Net if $\epsilon > 0$ \cite{elastic_net}. 
Supposing then that $\text{supp}(G)=\real^p$, and the canonical selection event given by
$$
{\cal B}_{(E,z_E)} = \left\{ \beta : \text{diag}(z_E)\beta_E > 0, \beta_{-E} = 0 \right \}.
$$
A parametrization of 
$$
\{(s,\omega,\beta, \alpha, z)\in  \bar{\mathcal{S}}_{0}^F(G,\epsilon,\ell,\mathcal{P}) : \beta \in\mathcal{B}_{(E,z_E)} \}
$$ 
is given by
\begin{equation*}
\begin{aligned}
	\psi_{(E,z_E)}(y, \beta_E, u_{-E}) 
 = \biggl(y, \:\epsilon &\begin{pmatrix} \beta_E \\0 \end{pmatrix}-X^T(y-X_E\beta_E)+\lambda\begin{pmatrix}  z_E \\ u_{-E} \end{pmatrix}, \\
& \begin{pmatrix} \beta_E \\0\end{pmatrix}, -X^T(y-X_E\beta_E), \lambda\begin{pmatrix}
	 z_{E}\\ u_{-E} \end{pmatrix}\biggr)
\end{aligned}
\end{equation*}
with the domain
\begin{equation} \label{eq:lasso:fixedX:support}
D_{(E,z_E)}=\left\{(y,\beta_E,u_{-E}): \text{diag}(z_E)\beta_E> 0, \|u_{-E}\|_{\infty} \leq 1 \right\}.
\end{equation} 
The reconstruction map for $\omega$ is given by
\begin{equation}
\begin{aligned}
\phi_{(E,z_E)}(y, \beta_E, u_{-E}) 
 = \epsilon &\begin{pmatrix} \beta_E \\0 \end{pmatrix}-X^T(y-X_E\beta_E)+\lambda\begin{pmatrix}  z_E \\ u_{-E} \end{pmatrix},
\end{aligned}
\end{equation}
again with the same support as the parametrization map.
For the canonical event ${\cal B}_{(E,z_E)}$, 
we therefore need to sample from a density proportional to 
\begin{equation} \label{eq:randomized:lasso:fixedX:density}
f_{\mu}(y) \cdot g \left(\epsilon \begin{pmatrix}\beta_E \\ 0 \end{pmatrix} - X^T(y-X\beta_E) + \lambda  \begin{pmatrix}  z_E \\ u_{-E} \end{pmatrix} \right) \cdot \left|\det(X_E^TX_E + \epsilon I)\right|
\end{equation}
supported on $D_{(E,z_E)}$, where $f_{\mu}$ is the $N(\mu,\sigma^2 I_n)$ density. As above, 
the same expression holds if our model for $y|X$ is not from the normal family. \\
A common variant of the LASSO implemented in \cite{glmnet} is
\begin{equation}
\minimize_{\beta \in \real^p} \frac{\|y-X\beta\|^2_2}{2 \|X^Ty\|_{\infty}}  + \lambda \|\beta\|_1.
\end{equation}
A randomized version of this is easily handled, changing the density to be proportional to
$$
f_{\mu}(y) \cdot g \left(\epsilon \begin{pmatrix}\beta_E \\ 0 \end{pmatrix} - \frac{X^T(y-X_E\beta_E)}{\|X^Ty\|_{\infty}} +  \lambda\begin{pmatrix}  z_E \\ u_{-E} \end{pmatrix} \right) \cdot \left|\det\left(\frac{X_E^TX_E}{\|X^Ty\|_{\infty}} + \epsilon I \right)\right|
$$
and supported on the same set as in \eqref{eq:lasso:fixedX:support}.
Another variant \cite{sabourin_nobel_permutation}
replaces $\|X^Ty\|_{\infty}$ with a resampling
based quantity $\text{median}(X^Ty)$ resulting in a similar
change in the sampling density.


\begin{remark}
\emph{\textit{Conditioning on Nuisance Statistics:}}
As detailed in \cite{optimal_inference}, we condition not just on the selection event ${\cal B}_{(E,z_E)}$, but also on the sufficient statistics corresponding to nuisance parameters to obtain optimal UMPU selective tests/ intervals. In such a case, the support for $y$ in the selective sampling density is restricted to a set, denoted as $D_{\text{obs}}$. This has been sketched for interested readers in \ref{sec:conditional} in the appendix.
\end{remark}


\subsection{The selective sampler is model agnostic: LASSO without parametric assumptions}

Up this point, we have assumed so far that $X$ is fixed and 
the law of $y|X$ was from the parametric model $\{N(\mu,\sigma^2 I): \mu \in \real^n\}$.

We now remove this assumption, assuming that the law $F$ is now just a law for the 
pair $(X, y) \in \real^{n \times p}\times \real^n$. 
A common assumption is the pairs model: $(x_i, y_i) \overset{IID}{\sim} \tilde{F}$ for some distribution
$\tilde{F}$ on $\real^p \times \real$, where $x_i^T$ are the rows of $X$. While we keep
this model in mind, it is not necessary in order to define
the appropriate sampler. This section serves as an example of the sense in which
our main result Theorem \ref{thm:change:measure} is agnostic to the underlying statistical model. 

In the notation established so far
$$
\begin{aligned}
S &= (X, y) \in \real^n \times \real^{n \times p}, \\
\ell(\beta;(X,y)) &= \frac{1}{2} \|y-X\beta\|^2_2, \\
{\cal P}(\beta) &= \lambda \|\beta\|_1.
\end{aligned}
$$

Therefore, we must sample from a density proportional to 
\begin{equation*} 
\label{eq:randomX:sampler}
\begin{aligned}
 f(X, y) \cdot g \left(\epsilon \begin{pmatrix} \beta_E \\ 0 \end{pmatrix} - X^T(y-X_E\beta_E) +  \lambda \begin{pmatrix} z_E \\ u_{-E} \end{pmatrix} \right) \cdot\left| \det(X_E^TX_E + \epsilon I)\right|
\end{aligned}
\end{equation*} 
supported on
\begin{equation} \label{eq:lasso:randomX:support}
\left\{(X,y, \beta_E,u_{-E}): \text{diag}(z_E)\beta_E> 0, \|u_{-E}\|_{\infty} \leq 1 \right\}.
\end{equation}
In order to use this result for particular applications of inference, one may have to reduce the problem to
sampling from a particular reference distribution perhaps 
by conditioning on appropriate sufficient statistics, as described in Appendix \ref{sec:conditional}.

\subsection{Selective CLT: $\ell_1$-penalized logistic regression} \label{sec:logistic}

In this section, we describe an application of the selective CLT of \cite{randomized_response}, applied
to the case of logistic regression with random design matrix $X\in\mathbb{R}^{n\times p}$ with rows $x_i^T$, $i=1,\ldots, n$. Suppose $$x_i \overset{iid}{\sim} F_X,\; x_i \in \mathbb{R}^p, \; y_i|x_i \sim \text{Bernoulli}(\pi(x_i^T b)),$$ where 
$\pi(x) = \exp(x) / (1+\exp(x))$ and $b \in\mathbb{R}^p$ is unknown, $p$ fixed, $S = (X, y)$. 
With logistic loss,
$$
\ell(\beta;(X, y)) 
= -\frac{1}{\sqrt{n}} \left[\sum_{i=1}^n y_i \log \pi(x_i^T \beta) + (1-y_i) \log(1-\pi(x_i^T \beta)) \right],
$$
we solve randomized $\ell_1$-penalized logistic regression introduced in \cite{randomized_response},
\begin{equation}
    \label{eq:randomized_logistic}
   \minimize_{\beta \in \mathbb{R}^p} \ell(\beta; (X,y)) - \omega^T\beta + \frac{\epsilon}{2}\|\beta\|_2^2+ \lambda\|\beta\|_1 .
\end{equation}
On the usual selection event of observing active set and signs $(E,z_E)$, the randomization reconstruction map is
$$
\phi_{(E,z_E)}(y,\beta_E, u_{-E}) = \epsilon \begin{pmatrix}\beta_E\\ 0\end{pmatrix} - \frac{1}{\sqrt{n}}X^T (y - \pi(X_E{\beta_E})) + \lambda\begin{pmatrix} z_E\\ u_{-E}\end{pmatrix}.
$$ 
And, the sampling density becomes proportional to 
\begin{equation} 
\label{eq:randomX:logistic:sampler}
\begin{aligned}
 f(X, y) &\cdot g \left(\epsilon \begin{pmatrix} \beta_E \\ 0 \end{pmatrix} - \frac{1}{\sqrt{n}}X^T(y-\pi(X_E\beta_E)) +  \lambda \begin{pmatrix} z_E \\ u_{-E} \end{pmatrix} \right) \\
  &\cdot\left|\det\left(\frac{1}{\sqrt{n}}X_E^TW(X_E\beta_E)X_E + \epsilon I\right)\right|,
\end{aligned}
\end{equation}
where $W(X\beta)=\text{diag}(\pi(X\beta)(1-\pi(X\beta)))$\footnote{With slight abuse of notation, we allow $\pi: \mathbb{R}^n \to \mathbb{R}^n, x \mapsto \pi(x)$ to be the function
applied on each coordinate of $x \in \mathbb{R}^n$.} is the Hessian matrix of the loss and the density above is supported on the same set as in \eqref{eq:lasso:randomX:support}. 
Sampling $(X, y)$ jointly is not feasible when $F_X$ is unknown.
Denoting 
\begin{equation} \label{eq:logistic:unpenalized:mle}
	\bar{\beta}_E = \underset{\beta \in\mathbb{R}^{|E|}}{\textnormal{argmin}} -\sum_{i=1}^n y_i\log\pi(x_{E,i}^T\beta) +(1-y_i)\log(1-\pi(x_{E,i}^T\beta)),
\end{equation}
the MLE for the unpenalized logistic regression with only the variables in $E$,
a Taylor series expansion of $\nabla \ell(\beta;(X,y))$ gives\footnote{Alternatively, we might take $\bar{\beta}_E$ to be the one-step estimator in the selected model starting from $\hat{\beta}_E$ \cite{penalized_l1}.}
\begin{align*} 
\nabla \ell(\beta;(X,y)) &\approx \sqrt{n}\begin{pmatrix} Q(\beta_E-t_{E}) \\ -t_{-E}+C(\beta_E-t_{E}) \end{pmatrix}.
\end{align*}
\noindent Here, $t$ is the observed value of the random vector
\begin{equation} \label{eq:logistic:T:defined}
T=
\begin{pmatrix}
\bar{\beta}_E \\
 \frac{1}{n}X_{-E}^T( y - \pi(X_E\bar{\beta}_E))
\end{pmatrix},
\end{equation}
$$Q=\frac{1}{n}X_{E}^TW(X_E\bar{\beta}_E)X_E \text{ and } C=\frac{1}{n}X_{-E}^TW(X_E\bar{\beta}_E)X_E.$$ Since $Q$ and $C$ converge by the law of large numbers to fixed quantities, we can either treat them as fixed or bootstrap them. For more details, see \cite{randomized_response}. \\
With $p$ is fixed, pre-selected, $T$ properly scaled is asymptotically normal 
\begin{equation}
\label{eq:asymptotic:normality}
 \sqrt{n} \left(T - \begin{pmatrix}b_E \\ \gamma_{-E} \end{pmatrix} \right) \overset{D}{\to} N(0, \Sigma),
\end{equation}
where $\Sigma$ is estimable from the data. When the selected parametric model is correct ($E \supseteq \textnormal{supp }{b}$)
we note that $\gamma_{-E}=0$ and $\Sigma$ is in fact block-diagonal \citep{penalized_l1}.
When the selected parametric model is not correct, one can estimate this covariance nonparametrically, using bootstrap (see Section \ref{sec:implementation}).
Since asymptotically $T$ is from an exponential family with parameters $b_E$, we could base inference on $b_E$ by sampling from the distribution of $T$ instead of $(X, y)$.


In this case, the parametrization map becomes
\begin{equation}
\begin{aligned}
	\psi_{(E,z_E)}(t, \beta_E, u_{-E}) = \biggl(t,  \epsilon &\begin{pmatrix}
 	\beta_E \\ 0  \end{pmatrix} +\sqrt{n}\begin{pmatrix} Q(\beta_E-t_{E}) \\ -t_{-E}+C(\beta_E-t_{E}) \end{pmatrix}+\lambda\begin{pmatrix}
 		 z_E \\ u_{-E}
 	\end{pmatrix}, \\
 	&\begin{pmatrix}
 \beta_E \\ 0	
 \end{pmatrix}, \sqrt{n}\begin{pmatrix} Q(\beta_E-t_{E}) \\ -t_{-E}+C(\beta_E-t_{E}) \end{pmatrix}, \lambda\begin{pmatrix} z_E\\ u_{-E} \end{pmatrix}
 \biggr),
 \end{aligned}
\end{equation}
with the domain $\mathbb{R}^p\times \mathbb{R}^{|E|}_{z_E}\times [-1,1]^{p-|E|}$. Hence, the density we sample from is proportional to 
\begin{equation} \label{eq:logistic:density}
	f(t)\cdot g\left(\epsilon \begin{pmatrix}
 	\beta_E \\ 0  \end{pmatrix} +\sqrt{n}\begin{pmatrix} Q(\beta_E-t_{E}) \\ -t_{-E}+C(\beta_E-t_{E}) \end{pmatrix}+\lambda\begin{pmatrix}
 		 z_E \\ u_{-E}
 	\end{pmatrix}\right),
\end{equation}
restricted on the domain of $\psi_{(E,z_E)}$, where $f$ here is the density of $\mathcal{N}\left(\begin{pmatrix}
	b_E \\ 0
\end{pmatrix}, \Sigma\right)$.
We note that, in order to construct a valid test about some linear functional of $b_{E}$ we can always condition on the observed $u_{-E}$ and / or $\beta_E$ if we desire.
The upside to such conditioning is that the sampling problem becomes easier, with the downside usually being a loss in selective power.

\subsection{First step of forward stepwise}
\label{sec:kac:rice:poly}

The Kac-Rice tests, described in \cite{kac_rice} are based on the solution of 
\begin{equation}
\label{opt-Kac Rice}
\maximize_{\eta \in \K}\eta^T X^T y, \; \text{ where } y\sim F,
\end{equation}
$X\in \real^{n\times p}$ is a fixed design, $\K$ a convex set that can be stratified into smooth disjoint manifolds and the process $\eta^TX^T y$ is Morse for almost every $y\in \real^n$. The simplest example of such
tests is one step of forward stepwise model selection, in which case $\K$ is the $\ell_1$ ball of radius 1.
In this case $X$ will be usually centered and scaled so that $X^Ty$ corresponds
to the marginal $Z$ statistics for $p$ different simple linear regressions.

Inference in this broad class of problems, for the global null, is based on the test statistic $\eta^{*T}X^Ty$, where
\begin{equation}
\label{test-stat-Kac Rice}
\eta^*=\argmax_{\eta \in \K}\eta^T X^T y,
\end{equation}
with $\K$ being the polar set of convex set $C$. Since, the null distribution of the above test statistic is intractable, we could instead provide inference based on the conditional law
$$\mathcal{L}\left(\eta^{*T}X^Ty \:\bigg |\: \eta^* =\argmax_{\eta \in \K}\eta^T X^T y\right).$$

With this brief description of inference based on Kac-Rice tests, we focus back on the randomized versions of the Kac-Rice objective in \eqref{opt-Kac Rice}. The randomized objective is given by
\begin{equation}
\begin{aligned}
\label{opt-rand-Kac Rice}
\maximize_{\eta\in\real^p}\eta^T (X^T y+\omega)-I_{\K}(\eta), \text{ where } y\times \omega\sim F\times G,
\end{aligned}
\end{equation}
with $\omega\in \real^p$ and penalty manifests as the characteristic function of convex set $\K$, that is 
\[ I_\K(\eta) = \begin{cases} 
      0 & \text{if } \eta\in\K,\\
      \infty & \text{otherwise.}
   \end{cases}
\]
Here we set $\epsilon=0$, as the above optimization problem does always have a solution. Having $\epsilon > 0$ would allow
for several variables to be selected.

As mentioned above, perhaps the simplest example of the randomized optimization problem in \eqref{opt-rand-Kac Rice} 
is forward stepwise problem with a single step. Specifically, that is consider
$$\maximize_{\eta \in \K}\eta^T (X^T y+\omega),$$
for $$ \K=\{\eta\in \real^p: \|\eta\|_1\leq 1\}.$$
The above optimization problem yields the optimal direction
\begin{equation*}
	\eta^*_j=\begin{cases}
s^* & \textnormal{ if } j=j^*\\
0 & \textnormal{ otherwise, }	
\end{cases}
\end{equation*}
$j=1,\ldots,p$, where 
$$
j^*= \argmax_{1\leq j\leq p} |X_j^Ty +w_j|, 
$$
coordinate with the maximum absolute value, and 
$$
s^*=\text{sign}(X_{j^*}^Ty +w_{j^*}),
$$
the corresponding sign. Thus, we condition on the first active direction (both $s^*$ and $j^*$), which gives rise to selection event
\begin{equation*}
\begin{aligned}
	\hat{E}_{(s^*, j^*)} = \bigl \{ (y,\omega)\in\mathbb{R}^n\times\mathbb{R}^p \::\: \text {sign}(X_{j^{*}}^T y+\omega_{j^*})=s^*, \\ 
	s^* (X_{j^*}^{T}y+\omega_{j^*})\geq \underset{1\leq j\leq p}{\max}|X_j^Ty+\omega_j| \bigr \}.
\end{aligned}
\end{equation*}
The subgradient equation yields the reconstruction map
\begin{equation*}
\label{KKT0}
\phi_{(j^*,s^*)}(y,z)=z-X^T y,
\end{equation*}
where $z \in \partial I_{\K}(\eta^*)$, the set of sub-gradients to $\K$ at $\eta^*$, given by the normal cone
\begin{equation*}
\begin{aligned}
\label{subgrad-FST}
\partial I_{\K}(\eta^*)&=\{c(u_1, \ldots, u_{j^*-1}, s^*, u_{j^*+1},\ldots,u_p): u_i \in\mathbb{R}, |u_i|\leq 1, c>0 \}.
\end{aligned}
\end{equation*}

We can reparametrize the set 
$$\left\{ (y, \omega, \beta, \alpha, z)\in  \bar{\mathcal{S}}^F_0(G, \epsilon = 0, \ell,\mathcal{P}): (y,\omega)\in \hat{E}_{(s^*,j^*)} \right\}$$
using
\begin{equation*}
\label{FST-one step-map}
\psi_{(s^*, j^*)}(y,z)=(y,z-X^Ty, \eta^*, X^Ty, z),
\end{equation*}
($\eta^*_j = s^* \mathbb{I}_{\{j=j^*\}}$, $j=1,\ldots, p$), with the domain $\mathbb{R}^n\times \partial I_{\mathcal{K}}(\eta^*)$.

With the above reparametrization, we sample $(y,z)$ from a density proportional to
\begin{equation}
\label{sampler-density-FST}
f(y)\cdot g(z-X^Ty),
\end{equation}
supported on $\mathbb{R}^n\times  \partial I_{\mathcal{K}}(\eta^*) $.


\begin{remark} \label{remak:cone}
The set of sub-gradients at $\eta^*$ for $s^*>0$ can be identified as the epigraph of the $\ell_{\infty}$ norm (modulo a permutation of the maximum coordinate $j^*$) and for $s^*<0$, $\partial I_{\mathcal{K}}(\eta^*)$ is the polar cone of the epigraph of $\ell_1$ norm.
\end{remark}

\begin{remark} \label{remak:moresteps}
Of course, it is usually of interest to take more than one step of forward stepwise. 
Inference after several steps of forward stepwise is considered in \cite{spacings,sequential_selective}. We consider
several steps of forward stepwise in Section \ref{sec:multiple:views}.
\end{remark}

\begin{remark} \label{remak:adjust2}
The data analyst may have some set of variables $\bar{E}$ that she insists on controlling for. 
In this case, instead of just assuming $X$ is centered, we might assume that $\mathcal{P}^{\perp}_{\bar{E}}X=X$ and its columns normalized,
where $\mathcal{P}_{\bar{E}}$ denotes projection onto the column space of $X_{\bar{E}}$.
In this way, centering $X$  corresponds to the common practice controlling for an intercept in the model.
\end{remark}

\begin{remark}
\label{remark:FS:glm}
Forward stepwise is also used in generalized linear models such as logistic regression. In this case, the quantity
$X^T\mathcal{P}^{\perp}_{E'}y$ that appears in the least squares setting can be replaced either by the Wald Z-statistics or the
score Z-statistics. For example, for logistic regression, the score Z-statistics take the form
\begin{equation}
\label{eq:score:glm}
Z_j = \frac{X_j^T(y - \pi(X_{E'}\bar{\beta}_{E'}))}{(X_j^TW(X_{E'}\bar{\beta}_{E'})X_j)^{1/2}}, \qquad 1 \leq j \leq p
\end{equation}
where $\bar{\beta}_{E'}$ is the unpenalized MLE for the model with variables $E'$ and
$$
W(X_{E'}\bar{\beta}_{E'}) = \text{diag}(\pi(X_{E'}\bar{\beta}_{E'}) (1 - \pi(X_{E'}\bar{\beta}_{E'}))
$$
is a consistent estimate of the variance of $y$ under the model with variables $E'$. 
While this is the typical parametric estimate of variance forward stepwise would use, one might prefer
using a jackknife or bootstrap estimate of this variance if one is unsure whether the model with
variables $E'$ is correctly specified, as would be the case early on in building a model 
via forward stepwise. Using Wald type $Z$ statistics would require fitting $p$ different logistic regression models which 
has some computational burden. 

Having computed the $Z_j$'s one might then consider a randomized
version of the problem
$$
\maximize_{\eta: \|\eta\|_1 \leq 1} \eta^TZ
$$
and proceed as above with $S$ being the vector $Z$.
\end{remark}


\subsection{Marginal screening} 
\label{sec:marginal:screening}

Marginal screening computes marginal $Z$ statistics
$$
Z_j = \frac{X_j^T y}{\sigma\|X_j\|_2}
$$
or $T$ statistics
$$
T_j = \frac{X_j^T y}{\hat{\sigma}_j\|X_j\|_2}
$$
 for each of $p$ centered  variables $X_j$ and thresholds their absolute value at some threshold, perhaps
$z_{1 -\alpha/2}$ where $\alpha$ is some nominal $p$-value threshold.
This can be expressed in optimization form as
$$\minimize_{\eta: \|\eta\|_{\infty} \leq c}\frac{1}{2}\|\eta-T\|^2_2.$$
Selective inference in the nonrandomized setting for this problem was considered in 
\citep{lee_screening}.

A natural randomized version would be
$$\minimize_{\eta: \|\eta\|_{\infty} \leq c}\frac{1}{2}\|\eta-T\|^2_2 - \omega^T\eta.$$
Conditioning on the set achieving the threshold $c$ and their signs to be $(E,z_E)$, we see that this event is
\begin{equation} \label{eq:ms:support}
\left\{(T,\eta,z): \eta_E = c\cdot z_E, \text{diag}(z_E) z_E \geq 0, \|\eta_{-E}\|_{\infty} < c, z_{-E}=0 \right\}
\end{equation}
and the randomization reconstruction map becomes
\begin{equation*}
	\phi_{(E, z_E)}(T, \eta_{-E}, z_E) = \begin{pmatrix}
		c \cdot z_E \\ \eta_{-E}
	\end{pmatrix} -T+\begin{pmatrix}
		z_E \\ 0
	\end{pmatrix}.
\end{equation*}
We thus sample $(y, X, \eta_{-E}, z_{E})$ (or $(T,\eta_{-E}, z_E)$ if $X$ is random) from a selective density proportional to  
\begin{equation}
\label{eq:ms:density}
f(T)\cdot g\left(\begin{pmatrix}
	 c\cdot z_E \\ \eta_{-E} \end{pmatrix} - T + \begin{pmatrix}
 	z_E \\ 0
 \end{pmatrix} \right)
\end{equation}
and supported on the  event in \eqref{eq:ms:support}, where $f$ is the unselective law of $T$.
For logistic regression, one can replace the $T$ statistics above with the score statistics
as described in Remark \ref{remark:FS:glm}.


\subsection{Full model screening} \label{sec:full:model:screening}

Another possible way to screen variables is to threshold the coefficients from the output of an unpenalized, randomized selection program given by
\begin{equation} \label{eq:full:ms:objective}
	\minimize_{\beta\in \real^p}\frac{1}{2}\|y-X\beta\|_2^2 + \frac{\epsilon}{2}\|\beta\|_2^2-\omega^T\beta, \;\; ((X,y),\omega)\sim F\times G,
\end{equation}
Based on the solution $\hat{\beta}((X,y), \omega)$, we define the threshold model $(E,z_E)$ as
$$E=\{i:|\hat{\beta}_i((X,y),\omega)|\geq a\sigma\}$$ and signs $z_E=\textnormal{sign}(\hat{\beta}_E((X,y),\omega))$, where $a$ is a constant (e.g.~$1-\alpha$ quantile of the standard normal or the $t$-distribution for some nominal level $\alpha$) and $\sigma$ is the scaling. The canonical event of interest is
\begin{equation*}
	\mathcal{B}_{(E,z_E)} = \{\beta\in\mathbb{R}^p:\: 
	 |\beta_i|\geq a \sigma \; \forall i\in E,
	 \|\beta_{-E}\|_{\infty}<a\sigma,
	 \textnormal{diag}(z_E)\beta_E>0 \}.
\end{equation*}

A parametrization of $$\{((X,y),\omega, \beta, \alpha,z)\in  \bar{\mathcal{S}}_{0}^F(G,\epsilon,\ell,\mathcal{P}) : \beta \in\mathcal{B}_{(E,z_E)} \}$$ 
is given by
\begin{equation*}
	\psi_{(E,z_E)}((X,y), \beta) 
 = \left((X,y), \: \epsilon \beta-X^T(y-X\beta), \beta, -X^T(y-X\beta), 0\right)
\end{equation*}
with the domain
\begin{equation*}
D_{(E,z_E)}=\left\{((X,y),\beta): \text{diag}(z_E)\beta_E> 0, |\beta_i|\geq a\sigma \; \forall i\in E,\|\beta_{-E}\|_{\infty} < a\sigma \right\}.
\end{equation*} 

For the canonical event ${\cal B}_{(E,z_E)}$, 
we therefore need to sample $((X,y), \beta)$ from a density proportional to 
\begin{equation}
f(X,y) \cdot g \left(\epsilon\beta - X^T(y-X\beta) \right) \cdot \left|\det(X^TX + \epsilon I)\right|
\end{equation}
supported on $D_{(E,z_E)}$, where $f$ and $g$ are the densities of $F$ and $G$, respectively. If $X$ is random, the sampling of the data simplifies as in Section \ref{sec:logistic}. The scaling $\sigma$ can be estimated as the variance of noise from the selected model with response $y$ and matrix of predictors $X_E$. As long as the estimate of $\sigma$ is consistent, we can treat it as a constant in the selection event, hence have the selection event as the polyhedral region.

\subsection{Selective sampler based on a penalized MLE: Fisher's exact selective test}
\label{sec:goodness:of:fit} 


In Section \ref{sec:mle}, we considered the pull-back of the unpenalized MLE for
an exponential family. 
Allowing for penalization in the problem \eqref{eq:mle:problem}
 as well as randomization yields programs of the form
\begin{equation*}
\label{eq:canonical:program}
\minimize_{\beta \in \real^p} \Lambda(\beta) -S^T\beta  +\frac{\epsilon}{2}\|\beta\|_2^2-\omega^T\beta + {\cal P}(\beta), \;\; (S,\omega)\sim F\times G.
\end{equation*}
We describe a version of Fisher's exact test based on observing some function
of the solution to the above program.
As we saw in the LASSO, the penalty ${\cal P}$ is often chosen such that the solution set
\begin{equation*}
\label{eq:solution:set}
\left\{(\beta, z): z \in \partial {\cal P}(\beta) \right\}
\end{equation*}
possesses a nice stratification into a discrete collection of bundles ${\cal C}$,
 the canonical example being the $\ell_1$ norm, or perhaps non-negative constraints. In particular, if 
${\cal P}$ is the support function of some convex set $K$, then \eqref{eq:solution:set}
is seen to be the normal bundle of $K$ which itself often has a nice stratification.

Fixing ${\cal C}$ to be one of these strata, the subgradient equations of the optimization problem yields
a reconstruction map
$$
\phi_{\mathcal{C}}(s,\beta,z) =  \nabla \Lambda(\beta) -s + z + \epsilon \cdot \beta
$$
where $(\beta, z)\in \mathcal{C}$ and $s \in\textnormal{supp}(F)$.
In this case, the parametrization map of $$\{(s,\omega, \beta, \alpha, z)\in \bar{S}_0(G,\epsilon, \ell, \mathcal{P}): (\beta,z)\in\mathcal{C}\}$$ becomes
\begin{equation*}
	\psi_{\mathcal{C}}(s,\beta,z) = (s,\nabla \Lambda(\beta) -s+\epsilon\beta +z,\beta,\nabla \Lambda(\beta) -s , z )
\end{equation*}
and the sampling density proportional to
\begin{equation*}
	f(s)\cdot g(\nabla \Lambda(\beta) -s+\epsilon\beta +z) \cdot \left|J\phi_{\mathcal{C}}(s, \beta, z) \right|
\end{equation*}
and restricted to $(s,\beta,z)\in M\times \mathcal{C}$.

We now describe a stylized instance of such a problem used in the context of goodness-of-fit tests.
Much of the effort in selective inference
has been focused on finding tools for inference about specific
parameters in a model, i.e. inference about statistical functionals
$\theta:{\cal M} \rightarrow \mathbb{R}$. 
Two examples in the literature that
employ such goodness-of-fit tests are \cite{choi_pca,sequential_selective}.

In this case, we assume 
the data analyst will use the data to decide what sufficient statistics
to use in constructing the goodness-of-fit test.
Suppose our data $S$ is a poisson process $N$ with points in ${\cal X}$ and 
intensity measure modeled with density
$$
\frac{d\lambda(\beta)}{d\mu}(x) = \exp \left(\sum_{j=1}^p \beta_j h_j(x)\right)
$$
with respect to some reference probability measure $\mu$ on ${\cal X}$.

If ${\cal X}=[0,1]$ and $h_j = 1_{[l_j,u_j]}$ were a collection of 
indicator functions, this would correspond to a change-point model.
We might then use
something like the fused LASSO  in which we can take the intervals to be the step functions $[l_j,u_j] = [(j-1)/p, 1]$ or perhaps
a multiscale approach as in \cite{frick,walther1,walther2} and use the penalty 
$$
{\cal P}(\beta) = \sum_{j=1}^p \lambda_j |\beta_j|
$$
a weighted $\ell_1$ penalty\footnote{The multiscale approach described in \cite{frick,walther1,walther2} is formally
a testing approach. The penalized version above was proposed in 
\cite{selective_sqrt} based on the dual of the test
statistic in the multiscale literature.}
. Usually, the background rate 
is included in such a model by fixing $h_1$ to be the constant
function 1. An unpenalized fit for this variable corresponds to its corresponding
$\lambda$ being 0.

Depending
on our choice of penalty, we might then solve the following
randomized program  
$$
\minimize_{\beta \in \real^p} \Lambda(\beta) -  \sum_{j=1}^p \left(\beta_j \int_{\cal X} h_j(x) \; N(dx) \right)
+ {\cal P}(\beta) + \frac{\epsilon}{2} \|\beta\|^2_2 - \omega^T\beta
$$
where
$$
\Lambda(\beta) = \log \left[\int_{\cal X} \left[\exp \left(\sum_{j=1}^p \beta_j h_j(x) \right) -1\right] \; \mu(dx)\right].
$$

Having solved the above problem, the data analyst now observes that 
$(\hat{\beta}(N,\omega),\hat{z}(N,\omega))$ are in some subbundle ${\cal C}$ of the
set of variable-subgradient pairs.
When ${\cal P}$ is a weighted $\ell_1$ penalty, 
the conditioning event
$$
\left\{(N,\omega):\hat{E}(N,\omega) = E, z_{\hat{E}(N,\omega)}=z_E\right\}
$$
where
$$
\hat{E}(N,\omega) = \left\{j: \hat{\beta}_j(N,\omega) \neq 0, \lambda_j > 0 \right\}
$$
are the selected penalized coordinates. 
The corresponding subbundle is the set
$$
{\cal C} = \left\{(\beta_E, \beta_U, z_{-E}): \text{diag}(z_E)\beta_E \geq 0, \|z_{-E}\|_{\infty} \leq 1,
\beta_U\in\real^U \right\}
$$
where $U=\{j:\lambda_j=0\}$ is the set of unpenalized coordinates. With some abuse of notation, we write
$z=(z_E, z_U, z_{-E}) = (z_E,0,z_{-E})$ as the full subgradient where $z_E$ are held fixed on ${\cal C}$ and 
$z_U=0$ as these coordinates are unpenalized.

The data analyst now decides to test the null hypothesis
$$
H_{0,j|E}: \beta_{j|E} = 0, \qquad j \neq 1
$$
in the model
$$
 \log \lambda(\beta) = \sum_{j \in E} \beta_j \int_{\cal X} h_j(x) \; \mu(dx).
$$
Following the exponential family setup in \cite{optimal_inference}, we
might condition on the nuisance sufficient statistics
$$
\int_{\cal X} h_i(x) N(dx), \qquad i \in E \setminus j.
$$
Assuming we have included a background rate in the model, this
fixes the total number of points in the Poisson process to be
$N({\cal X})$ the observed number of points.

If we had not used the data to select the intervals, conditioning on these sufficient
statistics and resampling points is exactly Fisher's exact test modulo the choice
of test statistic.
That is, the appropriate reference measure can be constructed by 
sampling from a Binomial process $\tilde{N}$ with $n=N({\cal X})$ 
points and distribution
$\mu$ conditioned to satisfy
$$
\int_{\cal X} h_i(x) \; \tilde{N}(dx) = \int_{\cal X} h_i(x) \; N(dx), \qquad i \in E \setminus j
$$
where the right hand side are the values observed in the data.
When ${\cal X}$ is a discrete space, this is a generalization of Fisher's exact
test. Sampling for such problems have a fairly rich literature (c.f. \citep{diaconis_sturmfels}). When ${\cal X}$ is Euclidean or a manifold, the conditioning
event above corresponds to a subset of the configurations of ${\cal N}({\cal X})$
points on ${\cal X}$ and is generally a nontrivial task. See \citep{diaconis_sampling_manifold} for further discussion and examples.

Having used the data to choose which sufficient statistics to use, we must use an appropriate
selective distribution.
Beyond just sampling $\tilde{N}$ from the conditional density, we must
sample $(\beta_E, \beta_U, z_{-E})$. Conditional on $\tilde{N}$ we see that
the appropriate density of the joint law with
respect to the product of the law of $\tilde{N}$ and $p$-dimensional Hausdorff
measure on ${\cal C}$ is proportional to
$$
(\tilde{N}, \beta_E, \beta_U, z) \mapsto g \left( \nabla \Lambda(\beta) - \int_{\cal X} h(x) d\tilde{N}(x) + \epsilon \cdot \beta + \text{diag}(\lambda) z \right)
$$
with 
$$
\nabla \Lambda(\beta) = 
\frac{\int_{\cal X} h(x) \exp \left(\sum_{i \in E \cup U} \beta_i h_i(x) \right) \; \mu(dx)}{\int_{\cal X}  \exp \left(\sum_{i \in E \cup U} \beta_i h_i(x) \right) \; \mu(dx)}
$$
For computational reasons, to evaluate the
integrals over ${\cal X}$ above, it may be simpler to use a discretization of 
${\cal X}$ as in Lindsey's method \cite{efron_tibs_lindsey}.

Finally, while we have considered using a weighted LASSO to choose the
sufficient statistics, one might use a penalty with some curvature as well, requiring
the modifications discussed in Section \ref{sec:curvature}.


\subsection{Graphical models} \label{sec:graphical:models}

Gaussian graphical models are a popular way to study network structures. In particular, it has often been 
used on many types of genome data, e.g. gene expression, metabolite concentrations, etc. Specifically,  
consider the $p$-dimensional normally distributed random variable
$$
X = (x_1, \dots, x_p) \sim N(\mu, \Sigma).
$$
It is of interest to study the conditional independence structure 
of the variables $\{1,2, \dots, p\}$. The conditional independence structure  
is conveniently represented by an undirectional graph $(\Gamma, \mathcal{E})$, where the nodes are $\Gamma = \{1,2,\dots,p\}$, and there is an edge $(i,j)\in\mathcal{E}$ if and only if $x_i \not\perp x_j$ conditional on all the other variables $\Gamma \backslash \{i, j\}$. Moreover, 
assuming the covariance matrix $\Sigma$ is not singular, we denote the inverse covariance matrix $\Theta = \Sigma^{-1}$, then
$$
x_i \perp x_j | X_{\Gamma \backslash \{i, j\}} \iff \Theta_{ij} = 0.
$$
In many applications of Gaussian graphical models, we assume the sparse edge structure, where we can hope to uncover 
the network structure even in the high-dimensional setting. 
We discuss applying the selective sampler
to the graphical LASSO \cite{glasso}. Another graphical model selection, neighborhood selection,
is discussed in Appendix \ref{sec:neighborhood:selection}.

The randomized graphical LASSO with randomization $\omega\sim G$
is the convex program
\begin{equation}
\label{eq:graphical:lasso}
\minimize_{\Theta: \Theta^T=\Theta, \Theta > 0} - \log \det (\Theta) + \text{Tr}((S - \omega)\Theta) + \lambda {\cal P}(\Theta), \qquad (S,\omega) \sim F\times G,
\end{equation}
where
$$
{\cal P}(\Theta) = \sum_{(i,j): i \neq j} |\Theta|_{ij}
$$
and 
$$
S \in M_k = \left\{A: A^T=A, A \geq 0, \text{rank}(A) = k \right\}, \qquad k \leq p.
$$
Usually, 
$$
S=\frac{1}{n}X^T\left(I - \frac{1}{n}11^T\right)X, \qquad X \in \real^{n \times p}
$$
for some data matrix
so that $k$ is generically $\text{min}(n-1,p)$. Note that $\Theta, S,\omega$ are $p\times p$ matrices.

The subgradient equations read as
$$
 \omega = \hat{\Theta}^{-1}(S,\omega) + S + \hat{z}(S,\omega),
$$
where $\hat{z}(S,\omega)\in \partial \mathcal{P}(\hat{\Theta}(S,\omega))$.
For $E \subset \{(i,j): i \neq j\}$, the natural selection event of interest is
$$
\mathcal{S}=\left\{(S,\omega): \exists \, \Theta=\hat{\Theta}(S,\omega) \text{ s.t. } \ \text{sign}(\Theta_E) = z_E, {\Theta}_{-(E\cup \{(i,i): 1\leq i\leq p \}) } = 0 \right\},
$$
Owing to symmetry in ${\Theta}$, on $\mathcal{S}$ we have 
\[ \hat{z}_{ij}(S,\omega) = \cfrac{\partial \mathcal{P}(\hat{\Theta}(S,\omega))}{\partial \Theta_{ij}}=\begin{cases} 
      2\lambda s_{ij} & (i,j)\in E \\
      0 & i=j \\
      2 \lambda u_{ij} & (i,j)\notin E \text{ and } i\neq j,
   \end{cases}
\]
where $$ u_{ij}= u_{ji} \text{ satisfies } \| u\|_{\infty} \leq 1 \text{ for } {(i,j)\notin E, i\neq j},$$
and $$ s_{ij}=\text{sign}(\Theta_{ij}).$$
Due to symmetry, we can parametrize $\omega$ by the active elements in $\Theta$ along with its diagonal elements (dimension $|E|/2+p$) and the inactive elements of the penalty subgradient (dimension $(p^2-p-|E|)/2$), corresponding to the inactive elements, both restricted to the upper triangular part of $\Theta$. To fix notations, denote
$$E^*=\{(i,j)\in E: i\leq j\}\cup \{(i,j): i=j\},$$
the set of active coordinates in the upper triangular part of $\Theta$ along with its diagonal elements.

This leads us to the following parameterization of $\{(S,\omega, \Theta, \alpha,z)\in \bar{\mathcal{S}}_0(G,\epsilon=0, \ell,\mathcal{P}):(S,\omega)\in \mathcal{S}\}$
\begin{equation}
\begin{aligned}
\psi_{(E,z_E)} {(S,\Theta_{E^*}, u_{-E^*})} = (S, \:& \Theta^{-1}(E^*) +  S + z(u_{-E^*}),\Theta(E^*), \\
& \Theta^{-1}(E^*) + S, z(u_{-E^*})),
\end{aligned}
\end{equation}
where $\Theta(E^*)$ is the matrix with $E^*$ entries from $\Theta_{E^*}$ and zeroes elsewhere, $z_{ij}(u) = 2\lambda (s_{ij}\mathbb{I}_{\{(i,j)\in E\}}+u_{i,j}\mathbb{I}_{\{(i,j)\not\in E, i\neq j\}})$ and the domain of $\phi_{(E,z_E)}$ is restricted to $\|u_{-E^*}\|_{\infty}\leq 1$.

The randomization reconstruction map can be written as
$$
\phi_{(E,z_E)}(S,\Theta_{E^*},u_{-E^*})=\Theta^{-1}(E^*) +  S + z(u_{-E^*}).
$$
With this parametrization, we can sample from the density proportional to
\begin{equation*}
\begin{aligned}
	f(S) \: & \cdot  g(\Theta^{-1}(E^*) +  S + z(u_{-E^*}))\cdot \left|J\psi(S,\Theta_{E^*},u_{-E^*})\right|.
\end{aligned}
\end{equation*}
and restricted to $\|u_{-E^*}\|_{\infty}\leq 1$.
The computation of above Jacobian involves (see the Appendix of \cite{penalized_l1})
$$\left(\cfrac{\partial \phi_{(E,z_E)}}{\partial \Theta_{ij}}\right)_{kl} =2(\sigma_{il}\sigma_{jk}+\sigma_{ik}\sigma_{jl})\text{ for } (i,j)\in E^*,$$
$$\left(\cfrac{\partial \phi_{(E,z_E)}}{\partial u_{ij}}\right)_{kl} = 2\lambda \delta_{ik}\delta_{lj}\text{ for } (i,j)\notin E^*, i\leq j,$$
where $\sigma_{ij}$ are the elements of $\Sigma$ and $\delta_{ij}$ is Kronecker delta function.

The above computations give the $(k,l)$ coordinate of the $(i,j)$th upper triangular matrix and as again, we simply restrict attention to $\{(k,l): k\leq l\}$, the upper triangular half of randomization map.

More polyhedral examples, as well as the generalized LASSO are presented in Appendix \ref{sec:lasso:variants}.

\section{Curvature} \label{sec:curvature}

Up to now, our examples have been polyhedral in nature.
In most examples
above, the selection event was some part of the normal bundle of such a polyhedral convex set.
Not all problems can be described by such polyhedral sets. In this 
section we give some examples with non-trivial curvature. As above,
many of the interesting selection events are still parts of the normal
bundle.

\subsection{Projection onto a convex set}

The problem of projection onto a closed convex set $K$ is canonical
example of a convex program. Hence, it seems to be a natural place
to start to introduce curvature into
our discussion.
(Indeed two variants of the LASSO above can be expressed
in terms of projection onto a convex set.) Let
$$
\begin{aligned}
\ell(\beta;s) &= \frac{1}{2} \|\beta-s\|^2_2 \\
{\cal P}(\beta) &= 
\begin{cases} 0 & \beta \in K \\
\infty & \text{otherwise.}
\end{cases}
\end{aligned}
$$

In this case, with $\epsilon=0$ the randomized program takes the form
$$
\minimize_{\beta \in K} \frac{1}{2} \|\beta-s\|^2_2 - \omega^T\beta.
$$
with solution
$$
\hat{\beta}(s,\omega) = \mathcal{P}_K(s + \omega),
$$
where $\mathcal{P}_K$ is the Euclidean metric projection  onto the set $K$ (i.e.~the projection  with metric induced by the $\ell_2$ norm).

The subgradients to our penalty are
$$\partial {\cal P}(\beta) = N_{\beta}(K),$$
the normal cone of $K$ at $\beta$. 

In this case, the subgradient equations take the form
$$
s + \omega = \beta + z, \qquad \beta \in K, z \in N_{\beta}(K).
$$
The normal cone $N_{\beta}(K)$ can be parameterized by its unit vectors
$$
S(N_{\beta}(K)) = \left\{u \in N_{\beta}(K): \|u\|_2=1\right\}
$$
and their length  $r \in [0,\infty)$.

We can therefore draw from the joint distribution of 
$(s,\omega)$ with a density proportional to 
$$
f(s) \cdot g( \beta + r \cdot u - s) \cdot |J_{\psi}(\beta, r, u)|,
$$
where\footnote{
This parameterization is slightly different than our usual one, in that it 
constructs $s + \omega$ instead of $\omega$, but the difference
is not important.
}
$$
\psi(s,\beta, r, z) = (s,\beta+ r \cdot u-s,\beta,\beta-s, r\cdot u) \qquad \beta \in K, u \in S(N_{\beta}(K)), r \geq 0.
$$

Integrating out $s$ yields the density
$$
(f \ast g)(\beta + r \cdot u) \cdot |J_{\psi}(\beta, r, u)|.
$$
where $f \ast g$ is the density of the convolution $F \ast G$.

The Jacobian here is recognizable as that in the Weyl-Steiner
volume of tubes formula \cite{weyl,hotelling,johnstone_tube,takemura_kuriki,sun_tubes,rfg}
\begin{equation}
\label{eq:tube:jacobian}
J_{\psi}(\beta,r,u) = \det \left(I + r \cdot C_{-u} \right) =
\det \left(I + C_{-r \cdot u}\right),
\end{equation}
where $C_{\eta}$ is the curvature matrix of $\partial K$ at $\beta$ 
in the direction $\eta$ normal to $\partial K$.

Suppose now $K$ is compact and that $F \ast G$ is the uniform distribution
on $K \oplus \delta B_2$. This can be achieved
by taking $F=\delta_0$, say and $G$ the uniform
distribution itself. In this case, $\text{supp}(G) = K \oplus \delta B_2$
and we arrive
at Weyl-Steiner's tube formula
$$
\begin{aligned}
1 &= \frac{1}{{\cal H}_p(K \oplus \delta B_2)}
 \int_{(\beta, r, u): \beta + r \cdot u \in K \oplus \delta B_2}
J_{\psi}(\beta, r, u) \; d\beta \; dr \; du \\
&=
 \int_0^{\delta} \left[\int 
J_{\psi}(\beta, r, u) \; d\beta \; du\right] \; dr. \\
\end{aligned}
$$
The formula above must be interpreted as a sum over pieces or strata of 
$K$ of different dimensions so that $d\beta$ properly refers
to Hausdorff measure of differing dimensions
and the matrix in  \eqref{eq:tube:jacobian}  is also of
differing dimensions. We refer
the readers to \cite{takemura_kuriki,rfg} for further details and 
points of entry into the volume-of-tubes literature. When $K$
is a smooth body, i.e.~its boundary is a 
differentiable $(p-1)$-dimensional hypersurface then there is only stratum
and the tube formula reads
$$
{\cal H}_p(K \oplus \delta B_2)
= \int_0^{\delta} \int_{\partial K}  
J_{\psi}(\beta, r, \eta_{\beta}) \; d\beta \; dr,
$$
where $\eta_{\beta}$ is the outward pointing unit normal vector field.

When $F \ast G$ is not the uniform distribution on $K \oplus \delta B_2$ we
might instead try to compute
$$
\begin{aligned}
(F \ast G)(K \oplus \delta B_2)
&= 
 \int_{(\beta, r, u): \beta + r \cdot u \in K \oplus \delta B_2}
(f \ast g)(\beta + r \cdot u) \cdot J_{\psi}(\beta, r, u) \; d\beta \; dr \; du \\
&=
 \int_0^{\delta} \left[\int 
(f \ast g)(\beta + r \cdot u) \cdot J_{\psi}(\beta, r, u) \; d\beta \; du\right] \; dr. \\
\end{aligned}
$$
This problem is considered in some generality when
$f \ast g$ is a smooth density in Chapter 10 of \cite{rfg}.
When $F \ast G = N(0, I_p)$, their probability,
expanded in a power series in $\delta$ plays an important
role in the Gaussian Kinematic Formula \cite{GKF1,GKF2}.

Unlike in the Weyl-Steiner formula, while the first steps are similar
and are what Weyl said was something any student of calculus could do \cite{weyl}, the 
above probability are not Riemannian invariants of $K$ but depend on 
how the law $F \ast G$ relates to $K$.

While the connections to curvature measures and volume-of-tubes formulae are
enlightening, we feel concrete examples are also very important.
In this section, we consider the group LASSO as a canonical example
of a statistical learning problem with curvature.


\subsection{Forward stepwise with groups of variables: Kac-Rice with groups}

This example is a generalization of the first-step of
forward stepwise when we allow groups of variables to enter, and is a second
example of the Kac-Rice test described above. 
Non-randomized approaches to this problem are considered in \cite{loftus_q,barber_group}.

The randomized version of the Kac-Rice objective with a group LASSO penalty determined by a partition $\{1,2,\ldots,p\}=\cup_{g\in G}g $ can be written as
 \begin{equation}
 \label{kac:rice:group:lasso}
\maximize_{\eta\in \real^p} \eta^T(X^T y+\omega)-I_\mathcal{K}(\eta), \; y\times \omega\in F\times H
 \end{equation}
 where 
 $$
 \mathcal{K}=\{\eta\in \real^p:\sum_{g\in G}\lambda_g \|\eta_g\|_2\leq 1\}
 $$
 and $I_{\mathcal{K}}$ is the usual characteristic function of set $\mathcal{K}$. 
In this section we have denoted the law of $\omega$ by $H$ with density $h$ to distinguish
the density from the group index $g$.
In this case, the optimal $\eta=\eta^*$ is given by
 \[ \eta^*_g =\begin{cases} 
      \cfrac{X_g^Ty +\omega_g}{\lambda_g \|X_g^Ty +\omega_g\|_2} & \text{ if } g=g^* \\
      0 & \text{ otherwise},
   \end{cases}
\]
where 
$$
g^*=\underset{g\in G}{\argmax} \frac{1}{\lambda_g}\|X_g^Ty +\omega_g\|_2.
$$ 
Conditioning on the selection event 
$$
\hat{E}_{g^*}=\left\{(y,\omega): g^*=\underset{g\in G}{\argmax} \|X_g^Ty +\omega_g\|_2/\lambda_g\right\},
$$
the subgradient equation leads to the reconstruction map
$$
\omega=\phi_{g^*}(y,c, z_{g^*}, (z_g)_{g\neq g^*})=-X^Ty +c\begin{pmatrix} z_{g^*} \\ (z_g)_{g\neq g^*}\end{pmatrix},
$$
subject to constraints 
$$
c>0, \; \|z_{g^*}\|_2=\lambda_{g^*}, \; \|z_g\|_2\leq \lambda_g \text{ for } {g\neq g^*}.
$$
Thus we sample $(y,c,z_{g^*},(z_g)_{g\neq g^*})$ from a density proportional to 
$$ 
f(y)\cdot h(cz-X^T y)\cdot c^{|g^*|-1}
$$
and supported on 
$$
\textnormal{supp}(F)\times \mathbb{R}_{+}\times \lambda_{g^*}S(\mathbb{R}^{|g^*|})\times \prod_{g\in G, g\neq g^*}\lambda_g B_2(\mathbb{R}^{|g|}),
$$ 
where $\lambda_{g^*}S(\mathbb{R}^{|g^*|})$ is a sphere with radius $\lambda_{g^*}$ in $\mathbb{R}^{|g^*|}$ and $\lambda_gB_2(\mathbb{R}^{|g|})$ is an $\ell_2$ ball in $\mathbb{R}^{|g|}$.
The Jacobian here, is the determinant of the derivative of the map
$$
\begin{pmatrix} z_{g^*} & c \end{pmatrix}\mapsto -X^T_{g^*}y + cz_{g^*},
$$
which equals $\textnormal{det}\begin{pmatrix} cV_{g^*} & z_{g^*}\end{pmatrix}=\lambda_{g^*}c^{|g^*|-1}$, where
$V_{g^*}\in\mathbb{R}^{|g^*|}\times\mathbb{R}^{|g^*|-1}$ be an orthonormal basis completion of $z_{g^*}$.

\subsection{Group LASSO}

As mentioned above, the LASSO \eqref{eq:lasso:program} can be expressed
in terms of metric projection, as can many problems when 
$\ell$ is squared-error loss and ${\cal P}$ is either
a (semi-)norm in or a constraint on a seminorm. When
the seminorm is polyhedral the Jacobian 
in our parameterization typically depends only on the design matrix $X$.
If the unit ball of the seminorm is not polyhedral, the 
Jacobian of the parameterization involves a curvature term which
depends on the design matrix $X$.

We consider the group LASSO \cite{group_lasso}.
The group LASSO norm, defined by a partition 
$$
  \cup_{g \in G} g = \{1, \dots, p\}
$$
and weights $(\lambda_g)_{g \in G}$ is defined as
\begin{equation}
\label{eq:group:lasso:norm}
{\cal P}(\beta) = {\cal P}_{G, (\lambda_g)_{g \in G}}(\beta) = \sum_{g \in G} \lambda_g \|\beta_g\|_2,
\end{equation}
where $\beta_g=\beta[g]$ are the coefficients in the $g$-th group.

The group LASSO problem is defined by
\begin{equation}
\label{eq:group:lasso}
\minimize_{\beta \in \real^p}
\frac{1}{2} \|y - X\beta\|^2_2 + \sum_{g \in G} \lambda_g \|\beta_g\|_2.
\end{equation}
Losses other than squared-error
can of course be used with this penalty -- it is the penalty
that is the group LASSO. The square-root group LASSO \cite{sqrt_group_lasso}
replaces the squared-error loss above with
the $\ell_2$ loss $\|y-X\beta\|_2$. 

The randomized version of the group LASSO is
\begin{equation}
\label{eq:group:lasso:random}
\minimize_{\beta \in \real^p}
\frac{1}{2} \|y - X\beta\|^2_2 - \omega^T\beta + \frac{\epsilon}{2} \|\beta\|^2_2 +  \sum_{g \in G} \lambda_g \|\beta_g\|_2
\end{equation}
where $\omega \sim H$ with density $h$.

Having solved \eqref{eq:group:lasso:random} with solution $\hat{\beta}(X,y,\omega)$, we define the set of
active groups as
$$
\hat{E}(X, y,\omega) = \left\{g: \hat{\beta}_g(X, y,\omega) \neq 0 \right\}.
$$
The canonical event of interest here is
\begin{equation}
\label{eq:group:event}
\left\{(X, y, \omega): \hat{E}(X, y,\omega) = E \right\}.
\end{equation}

On this event, the subgradient equation gives the randomization reconstruction map
\begin{equation} \label{eq:group:lasso:randomization}
\begin{aligned}
\omega &= \phi_E(y, (\beta_g)_{g \in E}, (z_h)_{h \in -E})\\
&= (X^TX + \epsilon I) \begin{pmatrix}
	(\beta_g)_{g\in E} \\ 0
\end{pmatrix} -X^T y + \begin{pmatrix}
(\lambda_g \beta_g /\| \beta_g \|_2)_{g \in E} \\ (z_h)_{h\in -E} \end{pmatrix},
\end{aligned}
\end{equation}
with the restriction on the inactive subgradients as 
$$
\|z_{h}\|_2\leq \lambda_h, \;\; h\in -E.
$$


\noindent We may choose to parameterize $$\{(y,\omega, \beta,\alpha, z)\in \bar{\mathcal{S}}_0^F(G,\epsilon,\ell,\mathcal{P}): \beta_h = 0 \; \forall h \in {-E} \}$$  in terms of $(y,(\beta_g)_{g\in E}, (z_h)_{h\in -E})$ as
$$
\begin{aligned}
\label{map-gpl}
&\psi_E(y,(\beta_g)_{g\in E}, (z_h)_{h\in -E}) \\
& =  \biggl( y, (X^TX + \epsilon I) \begin{pmatrix}
	(\beta_g)_{g\in E} \\ 0
\end{pmatrix} -X^T y + \begin{pmatrix}
(\lambda_g \beta_g /\| \beta_g \|_2)_{g \in E} \\ (z_h)_{h\in -E} \end{pmatrix} ,
  \begin{pmatrix}
 (\beta_g)_{g\in E} \\ 0 \end{pmatrix}, \\
 & -X^T\left(y-X^T\begin{pmatrix}
 (\beta_g)_{g\in E} \\ 0 \end{pmatrix}\right),
 \begin{pmatrix}
(\lambda_g \beta_g /\| \beta_g \|_2)_{g \in E} \\ (z_h)_{h\in -E} \end{pmatrix}
\biggr)
\end{aligned}
$$
with the restriction on the inactive subgradient
\begin{equation}\label{eq:group_lasso:inactive}
{(z_h)_{h\in -E}}\in \prod_{h\in -E} \lambda_h B_2(\mathbb{R}^{|h|}),
\end{equation}
where $B_2(\mathbb{R}^{|h|})$ is the $\ell_2$ ball in $\mathbb{R}^{|h|}$.
The density to sample from thus, is proportional to
\begin{equation}
\label{density-gpl}
f(y) \cdot g( \phi_E(y, (\beta_g)_{g\in E},(z_h)_{h\in -E}))\cdot \left|D_{((\beta_g)_{g\in E},(z_h)_{h\in -E})} \phi_E\right|,
\end{equation}
restricted to \eqref{eq:group_lasso:inactive}.
\noindent Denoting $M=\cup_{g\in E} g$, the set of active predictors, and their cardinality as $|M|=\sum_{g\in E} |g|$, where $|g|$ is size of each group, the non-trivial Jacobian term is precisely given by
\begin{equation}
\begin{aligned}
\label{gpl:jacobian}
& D_{((\beta_g)_{g\in E},(z_h)_{h\in -E})} \phi_E \\
&=\text{det}\left(\begin{bmatrix}
 	X_M^TX_M+\epsilon I+D_1 & 0_{|M|\times (p-|M|)} \\
 	 X_{-M}^TX_M &  I_{p-|M|} 
	\end{bmatrix}
	\right) \\
& =\textnormal{det}(X_M^TX_M+\epsilon I+D_1),
\end{aligned}
\end{equation}
where $D_1$ is a block diagonal matrices of the form
$$
\label{block diag 2}
D_1=\text{diag}\left(\left(\cfrac{\lambda_g}{\|\beta_g\|_2}\left(I-\cfrac{\beta_g\beta^T_g}{\|\beta_g\|_2^2}\right)\right)_{g\in E}\right) \in\mathbb{R}^{|M|\times |M|}.
$$
\begin{remark}
For groups that are singletons ($|g|=1$ for all $g\in G$), the 
 matrix $D_1$ is zero as $\beta_g/\|\beta_g\|_2 \in \{1,-1\}$.
\end{remark}

\subsection{Conditioning on active directions in group LASSO} \label{sec:gl:active:directions}

If we choose to condition additionally on the active directions $$\left\{u_g=\cfrac{\beta_g}{\|\beta_g\|_2},\; g\in E\right\},$$ 
writing
$$
\begin{aligned}
&\beta_g = \gamma_g \cdot u_g,  \; \gamma_g > 0, & g \in E; \\
&\|z_h\|_2  \leq \lambda_h, & h \in -E,
\end{aligned}
$$
we may express the reconstruction in \eqref{eq:group:lasso:randomization} in terms of $(y,(\gamma_g)_{g\in E},(z_h)_{h\in -E})$ as
\begin{equation*}
\begin{aligned}
\label{eq:group:lasso:embed0}
&\phi_{(E,(u_g)_{g\in E})}(y, (\gamma_g)_{g\in E},(z_h)_{h\in -E})\\
&= (X^TX + \epsilon I) \begin{pmatrix}
(\gamma_g u_g)_{g\in E} \\ 0
\end{pmatrix} -X^T y + \begin{pmatrix}
(\lambda_g u_g)_{g \in E} \\ (z_h)_{h\in -E}
\end{pmatrix}
\end{aligned}
\end{equation*}
restricted to 
\begin{equation*} \label{eq:group_lasso:support}
	(y, (\gamma_g)_{g\in E}, (z_h)_{h\in -E})\in \mathbb{R}^n \times (0,\infty)^{|E|} \times \prod_{h\in -E}\lambda_h B_2(\mathbb{R}^{|h|}).
\end{equation*}
In this case the density we sample from is proportional to
\begin{equation}
\label{active-dirs-cond-gplasso}
	f(y) \cdot g(\phi_{(E,(u_g)_{g\in E})}(y, (\gamma_g)_{g\in E},(z_h)_{h\in -E}))\cdot \left| D_{((\gamma_g)_{g\in E},(z_h)_{h\in -E})}\phi_{(E,(u_g)_{g\in E})}\right|,
\end{equation}
restricted to support above.
The Jacobian of the map in this case can be computed as polynomial in $(\gamma_g)_{g\in E}$ by taking a geometric approach outlined below. Hence, conditioning further on the active directions allows us to sample from a log-concave density.\\
To see an explicit computation of the Jacobian, let $V_g \in \mathbb{R}^{|g| \times (|g|-1)}$, $g\in E$, denote a matrix whose columns are orthonormal and orthogonal to $u_g$ and set
$$
V=\textnormal{diag}\left((V_g)_{g\in E}\right) \in \mathbb{R}^{|M| \times (|M|-|E|)},\; U = \textnormal{diag}((u_g)_{g\in E}) \in \mathbb{R}^{|M| \times |E|}.
$$ 
Finally set
$$
\Gamma = \textnormal{diag}((\gamma_gI_{|g|})_{g\in E}) \in \mathbb{R}^{|M|\times |M|}.
$$


The Jacobian is the derivative of the map
$$
(U,\gamma) \mapsto \nabla \tilde{\ell}(\Gamma U) + \Lambda U,
$$
where 
$$\nabla \tilde{\ell}(\Gamma U)=(X_M^T X_M +\epsilon I_{|M|}) \Gamma U,$$
and
$$
\Lambda = \textnormal{diag}((\lambda_gI_g)_{g\in E}) \in \mathbb{R}^{|M| \times |M|}.
$$

Differentiating first with respect to $V$ (i.e.~tangent to $U$) and then $\gamma$, the 
derivative can be written as
$$
\begin{pmatrix}
(\nabla^2 \tilde{\ell}(\Gamma U) \Gamma + \Lambda) V & \nabla^2 \ell(\Gamma U) U
\end{pmatrix} = \left(\nabla^2 \tilde{\ell}(\Gamma U) \Gamma + \Lambda VV^T\right)
\begin{pmatrix} V & \Gamma^{-1}U \end{pmatrix}.
$$
Since $\tilde{\ell}$ is quadratic, $Q = \nabla^2 \ell(\Gamma U)=X_M^TX_M+\epsilon I_{|M|}$ does not involve $\Gamma U$, the determinant of the above expression is equivalent to 
$$
\det\left(\begin{pmatrix} V^T \\  U^T \end{pmatrix} Q^{-1} \begin{pmatrix}
(Q \Gamma + \Lambda) V & Q U
\end{pmatrix} \right). \footnote{Since we condition on the active directions, $U$ and $V$ are fixed matrices hence $\begin{pmatrix}V & U \end{pmatrix}^T Q^{-1}$ is fixed as well.}
$$ 
Letting
$$
\Gamma^{-} = \textnormal{diag}((\gamma_g I_{|g|-1})_{g\in E}),\;\;
\bar{\Gamma} = \begin{pmatrix}
\Gamma^- & 0  \\
0 & I_{|E|}
\end{pmatrix},
$$
the Jacobian reduces to computing
\begin{equation} \label{eq:gls:active:sets:jacobian}
\begin{aligned}
\det & \left(\bar{\Gamma} \left(I_{|M|} + \bar{\Gamma}^{-1} \begin{pmatrix}
V^TQ^{-1}\Lambda V & 0 \\
Z^TQ^{-1}\Lambda V & 0
\end{pmatrix}
\right)
\right) \\ 
&=\det\left(\Gamma^{-}+ V^TQ^{-1}\Lambda V\right).
\end{aligned}
\end{equation}
Conditioning on $U$ is now straightforward as $V$ will then be fixed in the above density.


\subsection{General geometric approach}
\label{gen:geo:approach}

More generally, one can take a geometric approach to obtain the Jacobian, even when we do not condition on the active directions. 

For the group LASSO, in computing the Jacobian it is convenient to assume that $\ell$ is somewhat
arbitrary but twice-differentiable and to introduce a
new Riemannian metric on 
the fiber of optimization variables $((\gamma_g)_{g\in E}, (z_g)_{g \in E})$ over the
data point $s$. The convenience is that the structure of the Jacobian
is similar for many problems in which the penalty is a seminorm appearing
in Lagrange form. We expect
a similar formula to hold when the penalty  is a seminorm appearing in
bound form. In the interest of space, we do 
not pursue this here.

The Riemmanian metric we use is the pull-back
of the following metric on $\real^p$:
\begin{equation}
\label{eq:riemannian}
\langle V_{\beta}, Z_{\beta} \rangle_{(s,\beta)} = V_{\beta}^T \left (\nabla^2 \ell(\beta; s) + \epsilon I \right)^{-1} Z_{\beta}.
\end{equation}
With $\epsilon=0$ and $\ell$ the log-likelihood of an 
exponential family, in local coordinates the above metric
is inverse of the observed information metric \cite{efron_hinkley} evaluated
at the point estimate $\beta$. It is perturbed by $\epsilon$ times the
Euclidean metric above
as we have modified our convex program in randomizing it.

For the group LASSO, we then now differentiate the parameterization in coordinates $\gamma_E, (z_g)_{g \in G}$.
The main reason for introducing this metric is that the image tangent vectors in the face
$${\cal F} = \left(\prod_{g \in E} \lambda_g \cdot S(\real^{|g|})\right)  \times \left(\prod_{h \in -E} \lambda_h B_2(\real^{|h|})\right)$$
remain orthogonal to the normal vectors $\sum_{g \in E} \gamma_g z_g$. The next theorem derives the Jacobian when curvature component is present in the geometry.
\begin{theorem}[Jacobian meta-theorem]
\label{thm:jacobian} 
Consider the randomized version of the problem
$$
\minimize_{\beta \in \real^p} \ell(\beta;S) + \|\beta\|
$$
for some $F$-a.s. twice-differentiable convex $\ell$ and some seminorm 
$$
\|\beta\| = \sup_{z \in K} \beta^Tz.
$$
Let 
\begin{itemize}
\item ${\cal F}$ be a smooth face of $K$; 
\item $(\eta_i)_{i \in I}$ be an orthonormal (in the Euclidean metric)
frame field normal to ${\cal F}$; 
\item $(V_j)_{j \in A}$ be an orthonormal (in the Euclidean metric) frame field
tangent to ${\cal F}$,
\end{itemize}
and let $\psi$ denote the parameterization
$$
\psi(s, z, \gamma) = \nabla \ell( \beta(z,\gamma); s) + \epsilon \cdot \beta(z,\gamma)
+ z$$
restricted to $N({\cal F})$ with $\beta(z, \gamma) = \sum_{i \in I} \gamma_i \eta_i$.
Then, 
\begin{equation}
\begin{aligned}
\label{eq:orthogonality}
\left \langle \eta_i, \frac{\partial}{\partial \gamma_l} \psi \right \rangle_{(s,\beta)} &= \delta_{il} \\
\left \langle V_j, \frac{\partial}{\partial \gamma_i} \psi  \right \rangle_{(s,\beta)} &= 0 \\
\left \langle V_j, V_k \psi  \right \rangle_{(s,\beta)} &= S_{-\beta}(V_j, V_k) + \langle V_j, V_k  \rangle_{(s,\beta)}, \\
\end{aligned}
\end{equation}
where $S_{\eta}$ is the shape operator of ${\cal F}$ as it sits in $\real^p$.
Finally,
\begin{equation}
\label{eq:jacobian:geometric}
\begin{aligned}
J \psi(s,z,\gamma) &= \det( \nabla^2 \ell(\beta; s) + \epsilon \cdot I)
\cdot \det (G(s,z,\gamma) + C_{-\beta}(z,\gamma)) \\
&= \det(I + G(s,z,\gamma)^{-1}C_{-\beta}(z,\gamma)) / \det (H(s,z,\gamma)),
\end{aligned}
\end{equation}
where
$$
\begin{aligned}
G(s,z,\gamma)_{jk} &= \langle V_{j,z}, V_{k,z} \rangle_{(s,\beta)} \\
H(s,z,\gamma)_{il} &= \langle P^{\perp}_{(s,\beta)}\eta_{i,z}, P^{\perp}_{(s,\beta)}\eta_{l,z} \rangle_{(s,\beta)},
\end{aligned}
$$
where $P^{\perp}_{(s,\beta)}$ is projection orthogonal to the tangent space of ${\cal F}$ in the metric \eqref{eq:riemannian}
and $C_{-\beta}(z,\gamma)$ is the Euclidean curvature matrix of ${\cal F}$ in the basis $V_{i,z}, 1 \leq i \leq \text{dim}({\cal F})$.
\end{theorem}

\begin{proof}
The reconstruction map reads
$$
\omega = \nabla \ell\left(\sum_i \gamma_i \eta_i; s\right) + \epsilon \left(\sum_i \gamma_i \eta_i \right) + z
$$
with the constraint $\sum_i \gamma_i \eta_i$ is normal to $K$ at $z \in {\cal F}$.
We now compute
$$
\begin{aligned}
\frac{\partial}{\partial \gamma_i} \psi(s,z,\gamma) &= \left(\nabla^2 \ell\left(\sum_i \gamma_i \eta_i; s\right) + \epsilon \cdot I \right) \eta_i \\
V_j \psi(s,z,\gamma) &= \left(\nabla^2 \ell\left(\sum_i \gamma_i \eta_i; s\right) + \epsilon \cdot I\right) V_j\left(\sum_i \gamma_i \eta_i\right) + V_j.
\end{aligned}
$$
This is enough to establish \eqref{eq:orthogonality}. Now, the determinant we want to compute is the determinant of the matrix
$$
\begin{pmatrix}
V \psi & \frac{\partial}{\partial \gamma} \psi
\end{pmatrix}
$$
where $V\psi$ is the matrix whose columns are $V_j\psi$, and $\frac{\partial}{\partial \gamma}\psi$ is the matrix
whose columns are $\frac{\partial}{\partial \gamma_i}\psi$.

Factoring out $\nabla^2 \ell\left(\sum_i \gamma_i \eta_i; s\right) + \epsilon I$ and then multiplying by the orthogonal matrix
$$
\begin{pmatrix}
V^T \\ \gamma^T
\end{pmatrix}
$$
shows that the determinant we want to compute is $\det(\nabla^2 \ell(\beta;s) + \epsilon I)$ times the determinant of the matrix
$$
\begin{pmatrix}
S_{-\beta}(V_j, V_k) + \langle V_j, V_k \rangle_{(s,\beta)} & 0 \\
 \dots & I
\end{pmatrix}.
$$
Relation \eqref{eq:jacobian:geometric} now follows, with the second display simply a familiar formula for
the determinant expressed in block form
$$
\det( \nabla^2 \ell(\beta; s) + \epsilon \cdot I)^{-1} = \det(G(s,z,\gamma)) \cdot \det (H(s,z,\gamma)).
$$
When the loss is quadratic, or we use a quadratic approximation for the loss, we see that if we condition on 
$z$, then the term $\det (H(s,z,\gamma))$ is constant and only the first term must be computed.
\end{proof}

\begin{remark} \label{remark:logconcave:cond}
We note that while ${\cal F}$ may not be convex, when $\ell$ is quadratic (or a quadratic approximation is used for inference)
and we condition on $z$ the distribution for inference is log-concave in $\gamma$. More generally, if $f$ is logconcave
in $s$ as well, then the relevant distribution of $(s,\gamma)|z$ is jointly logconcave in $(s,\gamma)$.
\end{remark}

This general theorem can be applied to compute the Jacobian of the group LASSO with a parametrization in terms of the active and inactive directions and the magnitudes of active coefficients. We refer the readers to Appendix \ref{app:Jacobian:meta} for a rederivation of this Jacobian using the above geometric perspective.

\section{Multiple views of the data} \label{sec:multiple:views}

Often, an analyst might try fitting several models to a data set. 
These models might have different numbers of parameters, and different
objective functions. Examples include fitting a
regularization path \cite{glmnet} or stability selection \cite{stability_selection}.

Nevertheless, if
each model is an instance of \eqref{eq:canonical:random:program}, then
there is a straightforward procedure to construct a selective sampler.

Given a collection $(G_i, \epsilon_i, \ell_i, {\cal P}_i)_{i \in {\cal I}}$, 
the appropriate density is proportional to
\begin{equation}
f(s) \cdot \left(\prod_{i \in {\cal I}} g_i\left(\epsilon_i \beta_i + \alpha_i + z_i \right) \cdot J\psi_i(s,\beta_i,\alpha_i,z_i) \right)
\end{equation}
supported on some subset of
$$
\bigsqcup_{s \in \text{supp}(F)} \left(\prod_{i \in {\cal I}} \left\{(\beta_i, \alpha_i, z_i): \beta_i \in \real^{p_i}, \alpha_i \in \partial \ell_i(\beta_i;s), z_i \in \partial {\cal P}_i(\beta_i) \right\} \right).
$$
This set has the form of a bundle with base space $\text{supp}(F)$ and fibers as described above. Typically, each 
$\psi_i$ will be restricted to some canonical event 
${\cal B}_i$. For example, in something similar to stability
selection, each ${\cal B}_i$ might identify
the set of variables and signs chosen by the $i$-th randomized LASSO program.

We see that given $s$, the variables
in the fiber are independent. Hence, these can be sampled
in parallel. For instance, on different machines initiated with the same
random seed, we can sample $s$ IID from density $f$, then run
Gibbs or other samplers to sample $(\beta_i,\alpha_i,z_i)$.

\subsection{Multiple steps of forward stepwise}
\label{sec:fs:multiple}

As an example, we consider taking $K$ steps of forward stepwise.
The Kac-Rice test can be extended to $K$ steps of forward stepwise, where the selection event is characterized by a sequence of indices with the corresponding signs that constitute the active set at step $K$. To make it more explicit, say we consider a sequence of active variables $(j_1,\ldots,j_K)$ with corresponding signs $(s_{1},s_{2},\ldots,s_K)$, in the order in which the randomized variables enter the model. The $k$-th optimization in the set of $K$ optimizations can be written as
\begin{equation}
\label{eq:k_step_fs:objective}
\maximize_{\eta \in \real^{p-k+1}}\eta^T (X_{-\mathcal{A}_{k-1}}^T\mathcal{P}_{\mathcal{A}_{k-1}}^{\perp} y+\omega_k )-I_{\K_k}(\eta), \;\text{ where } y\times \omega_k \sim F\times G_k, 
\end{equation}
$\mathcal{A}_{k}=\{j_1,j_2,\ldots,j_k\}$ is the active set including the $k$-th step, $X_{-\mathcal{A}_{k}}$ are the columns of $X$ except for the ones corresponding to the current active set $\mathcal{A}_{k}$, the characteristic function
\[ I_{\K_k}(\eta) = \begin{cases} 
      0 & \text{if } \eta\in\K_k\\
      \infty & \text{otherwise}
   \end{cases}
\]
 and $$\K_k=\{\eta \in\real^{p-k+1}: \|\eta\|_{1}\leq 1\}.$$ Here, $\{\omega_k\}_{k=1}^K$ is a sequence of independent randomization variables with $\omega_k$ coming from a given distribution $G_k$ in $\mathbb{R}^{p-k+1}$ and corresponding density $g_k$. The projection $\mathcal{P}_{\mathcal{A}_{k}}$ is onto $X_{\mathcal{A}_k}$ and $\mathcal{P}_{\mathcal{A}_{k}}^{\perp}$ is the residual after this projection. The selection event after $K$ steps can be written as
\begin{equation*}
\begin{aligned}
\label{FST-k step}
 \hat{E}_{\{(s_k,j_k)\}_{k=1}^K} &= \biggl\{ \left(y,\{\omega_k\}_{k=1}^K\right)\in\mathbb{R}^n\times\prod_{k=1}^K\mathbb{R}^{p-k+1}:  \text {sign}(X_{j_k}^T{\mathcal{P}_{\mathcal{A}_{k-1}}^{\perp}}y+\omega_{k,j_k})=s_{k},  \\
& s_{k} (X_{j_k}^{T}{\mathcal{P}_{\mathcal{A}_{k-1}}^{\perp}} y+\omega_{k,j_k})\geq \underset{j\in\mathcal{A}_{k-1}^c}{\max}|X_j^T {\mathcal{P}_{\mathcal{A}_{k-1}}^{\perp}} y+\omega_{k,j}|, k=1,\ldots, K \biggr \}.
\end{aligned}
\end{equation*}
The randomization reconstruction map for the $k$-th step, from the subgradient equation is given by
\begin{equation*}
\label{KKT2}
\phi_k(y,z_k)=-X_{-\mathcal{A}_{k-1}}^T\mathcal{P}_{\mathcal{A}_{k-1}}^{\perp} y+z_k,
\end{equation*}
where, sub-differential $z_k\in \real^{p-k+1}$ from the $k$-th step is restricted to 
$$
z_k\in \partial I_{\K_k}(\eta_k^*),
$$
and $\eta_k^*\in\mathbb{R}^{n-k+1}$ is the optimal solution for the $k$-th optimization, as stated in \eqref{eq:k_step_fs:objective}.
 More explicitly, 
 the normal cone is given by 
\begin{equation} \label{eq:k_step_fs:normal:cone}
\partial I_{\K_k}(\eta_k^*)=\{  c\cdot u: u\in\mathbb{R}^{p-k+1}, u_{j_k}=s_k, |u_j|\leq 1\;\forall j\in\mathcal{A}_{k-1}^c, c>0\}.
\end{equation}


The sampler density is thus proportional to
\begin{equation}
\label{eq:fs:ksteps:density}
f(y)\cdot \prod_{k=1}^K g_k\left(z_k-X_{-\mathcal{A}_{k-1}}^T\mathcal{P}_{\mathcal{A}_{k-1}}^{\perp} y\right),
\end{equation}
supported on 
$$
(y, z_1,\ldots, z_k)\in \mathbb{R}^n\times \prod_{k=1}^K \partial I_{\K_k}(\eta_k^*).
$$
For logistic regression, one can replace the $T$ statistics above with the score statistics
as described in Remark \ref{remark:FS:glm}.

Further examples of algorithms that choose a variable based on multiple views of the data are presented in 
Appendix \ref{appendix:multiple}.

\section{Selective sampling via projected Langevin} \label{sec:implementation}


In the examples above, we sample from a joint density of data $S\in \real^n$ and optimization variables $T\in \real^p$, from a probability space on $(S,T)$, achieved via a reparametrization. The selective sampler now samples from a pull back measure on a transformed probability space under a reconstruction map for randomization $\omega$. This idea implemented on randomized convex programs typically allow us to sample from a joint density that is supported on a relatively simpler region. \\
On a canonical selection event $\mathcal{B}_{(E,z_E)}$, defined by the active variables $E$ and their signs/directions $z_E$, the joint sampling density is 
\begin{equation}
\label{sampling density-sampler}
h(s,t)\propto f(s)\cdot g(\phi_{(E,z_E)}(s,t))\cdot \left|D_{t}\phi_{(E,z_E)}(s,t)\right| \cdot 1_{{D}_{(E,z_E)}}(s,t),
\end{equation}
supported on constraint set ${D}_{(E,z_E)}\subset \real^n \times \real^p$, under reconstruction map $\phi_{(E,z_E)}$.\\
 When we consider 
$$f(s)\propto \exp(-\tilde{f}(s)) \text{ and } g(\omega)\propto \exp(-\tilde{g}(\omega)),$$
the negative of logarithm of the sampling density is proportional to
\begin{equation}
\label{log-sampling density}
\tilde{h}(s,\phi_{(E,z_E)}(s,t))=\tilde{f}(s)) +\tilde{g}(\phi_{(E,z_E)}(s,t))-\log \left|D_t\phi_{(E,S_E)}(s,t)\right|
\end{equation}
supported on set ${{D}_{(E,z_E)}}$.\\
We sample from the target density $\tilde{h}$ using updates from a projected Langevin random walk.   
\begin{algorithm}[H] 
\caption{\textbf{Projected Langevin Updates}}
\label{alg:sampling-approx}
\textit{ Iterative update [k+1]:} The $(k+1)$ update based on previous update $(s^{(k)},t^{(k)})$ is given by
\begin{equation*}
\begin{pmatrix} s^{(k+1)}\\ t^{(k+1)} \end{pmatrix}=\mathcal{P}\left(\begin{pmatrix}s^{(k)}\\ t^{(k)}\end{pmatrix}-\eta\begin{pmatrix}\grad_s \tilde{h}(s^{(k)},\phi_{(E,z_E)}(s^{(k)},t^{(k)}))\\ \grad_t \tilde{h}(s^{(k)},\phi_{(E,z_E)}(s^{(k)},t^{(k)}))\end{pmatrix}+\sqrt{2\eta}\begin{pmatrix}\xi_1^{(k)}\\ \xi_2^{(k)}\end{pmatrix}\right)
\end{equation*}
for step-size $\eta$, $\xi_1^{(k)}\times  \xi_2^{(k)}   \sim \mathcal{N}(0,I_n)\times \mathcal{N}(0,I_p)$ and $ \mathcal{P}$ is projection onto set $D_{(E,z_E)}$. Computing the $(k+1)$ update based on previous update $(s^{(k)},t^{(k)})$ thus involves two steps:
\begin{algorithmic}[1]
\STATE  Computing gradient of the negative of log density w.r.t. $(s,t)$, that is
 $$\begin{pmatrix}\grad_s \tilde{h}(s^{(k)},\phi_{(E,z_E)}(s^{(k)},t^{(k)}))\\ \grad_t \tilde{h}(s^{(k)},\phi_{(E,z_E)}(s^{(k)},t^{(k)}))\end{pmatrix}$$
\STATE  Compute projection of update from a noisy version of gradient descent of the log density $$\begin{pmatrix}s^{(k)}\\ t^{(k)}\end{pmatrix}-\eta\begin{pmatrix}\grad_s \tilde{h}(s^{(k)},\phi_{(E,z_E)}(s^{(k)},t^{(k)}))\\ \grad_t \tilde{h}(s^{(k)},\phi_{(E,z_E)}(s^{(k)},t^{(k)}))\end{pmatrix}+\sqrt{2\eta}\begin{pmatrix}\xi_1^{(k)}\\ \xi_2^{(k)}\end{pmatrix}$$ onto constraint set $D_{(E,z_E)}$.
\end{algorithmic}
\end{algorithm}

When $\tilde{h}(s,\phi_{(E,z_E)}(s,t))$ is a convex function on restriction ${{D}_{(E,z_E)}}$ and satisfies smoothness properties
\begin{equation*}
\left|\grad\tilde{h}(s_1,\omega_1)-\grad\tilde{h}(s_2,\omega_2)\right|\leq \beta\left| (s_1,\omega_1)-(s_2,\omega_2)\right|
\end{equation*}
\begin{equation*}
\left|\grad\tilde{h}(s_1,\omega_1)\right|\leq C,
\end{equation*}
for all $ (s_1,\omega_1), (s_2,\omega_2) \in D_{(E,z_E)}$, and the support ${{D}_{(E,z_E)}}$ is convex with non-empty interior and contained in a finite Euclidean ball, the projected Langevin sampler indeed converges to the target density in \eqref{sampling density-sampler} as proved in \cite{bubeck_langevin}. 


\begin{remark} For most examples, data $S$ is sampled from a Gaussian density, in which case, $\grad \tilde{f}$ is unbounded. One should be able to remove this condition by considering a restriction of $\tilde{f}$ to a bounded set of probability close to $1$ under Gaussian density $f$. 
\end{remark}

\begin{remark} The choices for randomizations typically include a Gaussian, Laplacian or Logistic distribution.\footnote{\cite{randomized_response} use heavy-tailed distributions such as Laplace or logistic as choices for randomization.} Under a Logistic density, $\tilde{g}$ is seen to satisfy the smoothness conditions. Qualitatively, our samples do not change much with any of these choices of randomization.
\end{remark}

\begin{remark} Typically, the support set ${{D}_{(E,z_E)}}$ is convex, but not bounded. We could again remove the boundedness condition on ${{D}_{(E,z_E)}}$, by considering a compact subset of probability close to 1 under sampling density $h$.
\end{remark}

\subsection{Examples}

 We revisit few examples like the LASSO with fixed and random design, the Kac-Rice with forward stepwise, the $\ell_1$-penalized Logistic to implement the projected Langevin sampler. We offer inference on the parameters in the selected model $E$ conditional on the canonical selection event $B_{(E,z_E)}$ of selecting active set $E$ and their signs/ directions, obtained upon solving the corresponding randomized programs. We base our tests on the randomized pivot, developed in \cite{randomized_response} for selective inference in a randomized setting \footnote{To construct the pivots from \cite{randomized_response} we can either integrate over the null statistic or condition on it. In the case of LASSO or $\ell_1$-penalized logistic with random design, the null statistic is $T_{-E}$ defined in \eqref{eq:logistic:T:defined}. In the case of LASSO with fixed $X$, the null statistic is $X_{-E}^Ty$. Here we report the results where we integrate over the null statistic (conditioning on the null is easier computationally since it reduces the size of the sampling space).}.


\begin{Example}
\emph{LASSO with fixed design (Section \ref{sec:random:lasso}):}
The response $y$ is modeled as $\mathcal{N}(Xb,\sigma^2)$ with $|\text{supp}(b)|=s$ and $\sigma^2=1$. We test the null hypothesis $H_0: b_{E,j}=0$, $b_{E,j}$ denoting the $j$-th coefficient in the selected model $E$, conditioning on the selection event $B_{(E,z_E)}$ and sufficient statistics $\mathcal{P}_{E\setminus j}y$, corresponding to nuisance parameters as discussed in \ref{sec:conditional}. Thus, the sampling density is proportional to \eqref{eq:randomized:lasso:fixedX:density} on the support 
$$\{(y,\beta_E,u_{-E}): \mathcal{P}_{E\setminus j}y=Y_{\text{obs}}, \text{diag}(z_E)\beta_E>0, \|u_{-E}\|_{\infty}\leq 1\}.$$
The updates from projected Langevin sampler can be written as
\begin{align*}
\begin{pmatrix}
	 \tilde{y}^{(k+1)} \\ \beta_E^{(k+1)} \\ u_{-E}^{(k+1)}
\end{pmatrix}
=\begin{pmatrix}	 
	  y^{(k)}- \eta y^{(k)}/\sigma^2 - \eta \nabla_y\tilde{g}(\phi_{(E,z_E)}(y^{(k)},\beta_E^{(k)}, u_{-E}^{(k)}))+\sqrt{2\eta}\xi_1^{(k)} \\
	  \mathcal{P}_1\left(\beta_E^{(k)}-\eta \nabla_{\beta_E}\tilde{g}(\phi_{(E,z_E)}(y^{(k)},\beta_E^{(k)}, u_{-E}^{(k)}))+\sqrt{2\eta}\xi_2^{(k)}\right)\\
  	 \mathcal{P}_2\left(u_{-E}^{(k)}-\eta\nabla_{u_{-E}}\tilde{g}(\phi_{(E,z_E)}(y^{(k)},\beta_E^{(k)}, u_{-E}^{(k)}))+\sqrt{2\eta}\xi_3^{(k)}\right)
	\end{pmatrix},
\end{align*}
where $(\xi_1^{(k)}, \xi_2^{(k)}, \xi_3^{(k)})\sim\mathcal{N}(0,I_p)\times\mathcal{N}(0, I_{|E|})\times\mathcal{N}(0, I_{p-|E|})$ and independent of everything else and $ \mathcal{P}_1$ is projection onto the orthant $\mathbb{R}^{|E|}_{z_E}$ and $ \mathcal{P}_2$ is the projection onto the cube $[-1,1]^{p-|E|}$. The conditioning on nuisance statistic is implemented by fixing $\mathcal{P}_{E\setminus j}y$ with the update being $$y^{(k+1)}=Y_{\text{obs}}+\mathcal{P}_{E\setminus j}^{\perp}\tilde{y}^{(k+1)},$$ where $\mathcal{P}_{E\setminus j}$ is the projection onto the column space of $X_{E\setminus j}$.
\end{Example}


\begin{Example}
\emph{$\ell_1$-penalized logistic \& LASSO with random design matrix (Section \ref{sec:logistic}):} In $\ell_1$-penalized logistic the response $y$ is modeled as i.i.d. $ \text{Bernoulli}(\pi(x_i^Tb))$, with $|\text{supp}(b)|=s$, with the same null hypothesis as above. In the case of LASSO with random design, the response is modeled as $y\sim\mathcal{N}(Xb,\sigma^2 I_n)$ with $\sigma^2=1$. If we denote as $\bar{\beta}_E$ the unpenalized MLE using only selected covariates $X_E$ in the respective problems, then we can describe the inference for both problems at once. We sample from the selective density in \eqref{eq:logistic:density}, supported on
$$\{(t,\beta_E,u_{-E}): \mathcal{P}_{E\setminus j} t=T_{\text{obs}}, \text{diag}(z_E)\beta_E>0, \|u_{-E}\|_{\infty}\leq 1\},$$
conditioning again on selection event $B_{(E,z_E)}$ and nuisance statistics. Here, $\mathcal{P}_{E\setminus j}$ is the projection onto the column space of $\begin{pmatrix}
	\tilde{\Sigma}_{E,E} & 0 \\ 0 & 0
\end{pmatrix}$ and the nuisance statistic equals 
$$\mathcal{P}_{E\setminus j} T=\tilde{\Sigma}_{E,E}\bar{\beta}_E$$ where $\Sigma_{E,E}$ is the covariance of $\bar{\beta}_E$ (we use a bootstrap estimate of this covariance in the shown results) and $\tilde{\Sigma}_{E,E}=\Sigma_{E,E}-\frac{\Sigma_{E,E}e_{E,j}e_{E,j}^T}{e_{E,j}^T\Sigma_{E,E}e_{E,j}}$.
The updates in the projected Langevin sampler are
\begin{align*}
\begin{pmatrix}
	 \tilde{t}^{(k+1)}\\ \beta_E^{(k+1)} \\ u_{-E}^{(k+1)}
\end{pmatrix}
=\begin{pmatrix}	 
	  t^{(k)}- \eta \Sigma^{-1}t^{(k)}- \eta \nabla_t\tilde{g}(\phi_{(E,z_E)}(t^{(k)},\beta_E^{(k)}, u_{-E}^{(k)}))+\sqrt{2\eta}\xi_1^{(k)} \\
	  \mathcal{P}_1\left(\beta_E^{(k)}-\eta \nabla_{\beta_E}\tilde{g}(\phi_{(E,z_E)}(t^{(k)},\beta_E^{(k)}, u_{-E}^{(k)}))+\sqrt{2\eta}\xi_2^{(k)}\right)\\
  	 \mathcal{P}_2\left(u_{-E}^{(k)}-\eta\nabla_{u_{-E}}\tilde{g}(\phi_{(E,z_E)}(t^{(k)},\beta_E^{(k)}, u_{-E}^{(k)}))+\sqrt{2\eta}\xi_3^{(k)}\right)
	\end{pmatrix},
\end{align*}
where $(\xi_1^{(k)}, \xi_2^{(k)}, \xi_3^{(k)})\sim\mathcal{N}(0,I_p)\times\mathcal{N}(0, I_{|E|})\times\mathcal{N}(0, I_{p-|E|})$ and independent of everything else and $ \mathcal{P}_1$ is projection onto the orthant $\mathbb{R}^{|E|}_{z_E}$ and $ \mathcal{P}_2$ is the projection onto the cube $[-1,1]^{p-|E|}$. The conditioning as usual is implemented by keeping $\mathcal{P}_{E\setminus j}t$ fixed and updating $$t^{(k+1)}=T_{\text{obs}}+\mathcal{P}_{E\setminus j}^{\perp}\tilde{t}^{(k+1)}.$$ 
\end{Example}


\begin{Example}
\emph{Forward stepwise for Kac-Rice (Section \ref{sec:fs:multiple}):} We perform $K$ steps of randomized forward stepwise as in \eqref{eq:k_step_fs:objective} with the response generated as $y\sim\mathcal{N}(Xb, \sigma^2 I_n)$ and $|\text{supp}(b)|=s$. The null hypothesis is $H_0:b_{j_K}=0$, where $j_k$ is the predictor chosen by the algorithm in the $k$th step, $k=1,\ldots, K$. 
We sample from the density \eqref{eq:logistic:density} with $s=K-1$ supported on 
\begin{equation*}
	\left\{(y, z_1,\ldots, z_K): \mathcal{P}_{\mathcal{A}_{K-1}}y =Y_{obs}, z_k\in \partial I_{\K_k}(\eta_k^*), k=1,\ldots, K \right \},
\end{equation*}
where we condition on the sufficient statistic for the nuisance parameters in the selected model consisting of predictors $X_{\mathcal{A}_K}$.
The updates from projected Langevin sampler when one computes the Kac-Rice objective conditional on the selection event determined by forward stepwise with $K$ steps are
\begin{align*}
\begin{pmatrix}
	 \tilde{y} \\ z_1 \\ \vdots \\ z_K 
\end{pmatrix}^{(i+1)}
=\begin{pmatrix}	 
	  y^{(i)}- \eta y^{(i)}/\sigma^2-\eta \nabla_y \left(\sum_{k=1}^K\tilde{g}_k(\phi_k(y^{(i)}, z_k^{(i)}))\right)+\sqrt{2\eta}\xi_1^{(i)} \\
	  \mathcal{P}_1\left(z_1^{(i)}-\eta \nabla_{z_1}\tilde{g}_1(\phi_1(y^{(i)}, z_1^{(i)}))+\sqrt{2\eta}\xi_2^{(i)}\right)\\
  	 \vdots \\
  	 \mathcal{P}_K\left(z_K^{(i)}-\eta \nabla_{z_K}\tilde{g}_K(\phi_K(y^{(i)}, z_K^{(i)}))+\sqrt{2\eta}\xi_{K+1}^{(i)}\right)
	\end{pmatrix},
\end{align*}
where $\mathcal{P}_k$ is the projection onto $\partial I_{\mathcal{K}_k}(\eta_k^*)$, given in \eqref{eq:k_step_fs:normal:cone}, for all $k=1,\ldots, K$.\footnote{This projection is easily done using Remark \ref{remak:cone}.} The conditioning step is then done as above by updating $y$ while keeping $\mathcal{P}_{\mathcal{A}_{K-1}}y$ fixed. 
\end{Example}


\begin{Example}
\emph{Group LASSO with fixed $X$ (Section \ref{sec:gl:active:directions}):}
The response $y$ is modeled as $\mathcal{N}(Xb,\sigma^2)$ with $|\text{supp}(b)|=s$ and $\sigma^2=1$. We test the null hypothesis $H_0: b_{E,g}=0$, $b_{E,g}$ denoting the coefficient corresponding to group $g$, $g\in E$, conditioning on the selection event $B_{(E,(u_g)_{g\in E})}$.\footnote{This includes conditioning on the active directions as well to get a log-concave density (Section \ref{sec:gl:active:directions})} Additionally, we condition on the  sufficient statistics for the nuisance parameters, $\mathcal{P}_{E\setminus j}y$.
The sampling density is proportional to \eqref{active-dirs-cond-gplasso} on the support 
$$
\left\{(y,(\gamma_g)_{g\in E},(z_{h})_{h\in -E})\in \real^n \times (0,\infty)^{|E|}\times \prod_{h\in -E} \lambda_h B_2(\real^{|h|}) :\mathcal{P}_{E\setminus g}y=Y_{obs}\right\}.
$$
The updates for  
$$
\begin{pmatrix}
	 \tilde{y}^{(k+1)} \\ (\gamma_g)_{g\in E}^{(k+1)} \\ (z_{h})_{h\in -E}^{(k+1)}
\end{pmatrix}
$$ 
from projected Langevin sampler can be written as
\begin{align*}
 \begin{pmatrix}	 
	  y^{(k)}- \eta y^{(k)}/\sigma^2 - \eta \nabla_y\tilde{g}(\phi^{(k)})+\sqrt{2\eta}\xi_1^{(k)} \\
	  \mathcal{P}_1\left((\gamma_g)_{g\in E}^{(k)}-\eta \nabla_{(\gamma_g)_{g\in E}}\tilde{g}(\phi^{(k)})+\eta (\textnormal{Tr}(D_{g}^{(k)}))_{g\in E}+ \sqrt{2\eta}\xi_2^{(k)}\right)\\
  	 \mathcal{P}_2\left((z_h)_{h\in -E}^{(k)}-\eta\nabla_{(z_{h})_{h\in -E}}\tilde{g}(\phi^{(k)})+\sqrt{2\eta}\xi_3^{(k)}\right)
	\end{pmatrix},
\end{align*}
where 
$$ \phi^{(k)} = \phi_{(E,(u_g)_{g\in E})}(y^{(k)},(\gamma_g)_{g\in E}^{(k)},(z_{h})_{h\in -E}^{(k)}), $$
$$(\xi_1^{(k)}, \xi_2^{(k)}, \xi_3^{(k)})\sim\mathcal{N}(0,I_{n})\times\mathcal{N}(0, I_{|E|})\times\mathcal{N}(0, I_{p-|M|}),$$ independent of everything else, $ \mathcal{P}_1$ is projection onto the orthant $(0,\infty)^{|E|}$ and $ \mathcal{P}_2$ is the projection onto the product of balls $\prod_{h\in -E} \lambda_h B_2(\real^{|h|})$. Here, term $(\textnormal{Tr}(G_g^{(k)}))_{g\in E}$ comes from differentiating the logarithm of the Jacobian \eqref{eq:gls:active:sets:jacobian} with respect to $(\lambda_g)_{g\in E}$ at iteration $k$:
$$
D_{g}^{(k)}=\left((\Gamma^{-})^{(k)} + V^TQ^{-1}\Lambda V\right)^{-1}G_g,
$$ 
 where
$$
  G_g = \frac{\partial V^TQ^{-1}\Lambda V}{\partial \lambda_g} \;\textnormal{  and  }\; (\Gamma^{-})^{(k)} = \textnormal{diag}\left((\gamma_g^{(k)}\mathbb{I}_{|g|-1})_{g\in E}\right).
$$

Note that $G_g$ is a fixed matrix, thus at every iteration we just need to invert $\left((\Gamma^{-})^{(k)}+V^TQ^{-1}\Lambda V\right)$ in order to compute $D_g^{(k)}$. The conditioning on $\mathcal{P}_{E\setminus g}y$ is done as in the above examples.
\end{Example}





\section*{Acknowledgements}
Jonathan Taylor was supported in part by National Science Foundation grant DMS-1208857 and
Air Force Office of Sponsored Research grant 113039. Jelena Markovic was supported by Stanford Graduate Fellowship.
Jonathan Taylor would like to thank the Berkeley Institute of Data Science, where part of this manuscript was written
while on sabbatical in Fall 2015.


\bibliographystyle{agsm}
\bibliography{selective_sampler}


\clearpage
\appendix
\section{Nuisance parameters \& conditional distributions}
\label{sec:conditional}

 In forming selective hypothesis tests or intervals when
the model ${\cal M}$ is an exponential family, it was noted in
\cite{optimal_inference} that the classical approach of conditioning
on appropriate sufficient statistics can be used to eliminate
nuisance parameters. In this section we describe selective
samplers that can be used to sample from such conditional distributions.

Specifically, suppose we use the
randomized Lasso \eqref{eq:lasso:randomized:program}
with $X$ considered fixed 
and  $\bar{E}$ is a subset of features
that we will use to form the selected model
\begin{equation}
\label{eq:selected:model}
{\cal M}_{\bar{E}} = 
\left\{\mathcal{N}(X_{\bar{E}}\beta_{\bar{E}}, \sigma^2_{\bar{E}}):
\beta_{\bar{E}} \in \real^{|\bar{E}|}, \sigma^2_{\bar{E}} > 0\right\}.
\end{equation}
Often, for a given $(E,s_E)$ observed after
fitting the randomized Lasso, we will choose $\bar{E}=E$ but this
is not strictly necessary.
As
$\sigma^2_{\bar{E}}$ is part of the indexing set for ${\cal M}_{\bar{E}}$, we are assuming $\sigma^2_{\bar{E}}$ is unknown.

Suppose we want a selective test of
$$
H_{0,j|\bar{E}}: \beta_{j|\bar{E}}=0
$$
in the model ${\cal M}_{\bar{E}}$. Standard exponential family calculations
detailed in \cite{optimal_inference} tell us that we can construct
such a test in the presence of the nuisance parameters $(\beta_{\bar{E} \setminus j|\bar{E}}, \sigma^2_{\bar{E}})$ by conditioning on the 
appropriate sigma-algebra:
$$
\sigma\left(X_{\bar{E}\setminus j}^Ty, \|y\|^2_2\right) =
\sigma\left(\mathcal{P}_{\bar{E}\setminus j}y, \|(I - \mathcal{P}_{\bar{E} \setminus j})y\|^2_2\right).
$$

For any fixed values  of the sufficient statistics, say
$(W_{obs}, SSE_{obs})$  (with $\mathcal{P}_{\bar{E} \setminus j}W_{obs}=W_{obs}$)
the conditional
distribution for any $F \in {\cal M}_{\bar{E}}$
is supported on the set
$$
D_{obs} = \left\{y: y=W_{obs} + r, \mathcal{P}_{\bar{E} \setminus j}r=0, \|r\|_2=SSE_{obs} \right\}.
$$
Under $H_{0,j|\bar{E}}$ it is uniformly
distributed over the above set.

We must therefore sample from the
set
$$
\left\{(y,\beta_E, u_{-E}): y \in D_{obs}, \text{diag}(s_E)\beta_E > 0, \|u\|_{-E} \leq 1
\right\}
$$
with a density proportional to
$$
f(y) \cdot g \left(\epsilon \begin{pmatrix} \beta_E \\ 0 \end{pmatrix} - X^T(y-X_E\beta_E) + \lambda \begin{pmatrix} s_E \\ u_{-E} \end{pmatrix} \right).
$$

\section{LASSO and variants} \label{sec:lasso:variants}

In this section, we describe two common variants of the LASSO.

\subsection{LASSO in bound form}

The LASSO program in bound form is defined
as
\begin{equation}
\label{eq:bound:form}
\minimize_{\beta \in \real^p: \|\beta\|_1 \leq \delta } \frac{1}{2} \|y-X\beta\|^2_2
\end{equation}
with its corresponding randomized version 
\begin{equation}
\label{eq:bound:form:random}
\minimize_{\beta \in \real^p: \|\beta\|_1 \leq \delta } \frac{1}{2} \|y-X\beta\|^2_2 +  \frac{\epsilon}{2} \|\beta\|^2_2 - \omega^T\beta.
\end{equation}

In the notation established so far
$$
\begin{aligned}
S &= y \\
\ell(\beta;y) &= \frac{1}{2} \|y-X\beta\|^2_2 \\
{\cal P}(\beta) &= 
\begin{cases}
0 & \|\beta\|_1 \leq \delta \\
\infty & \text{otherwise.}
\end{cases}
\end{aligned}
$$

Typically, we will be interested in doing inference when the constraint
above is tight. In this case, for the canonical event ${\cal B}_{(E,z_E)}$ 
the set $\bar{\solutionset}^F_0(G,\epsilon,\ell,{\cal P})$
can be parametrized by
$$
\begin{aligned}
\lefteqn{\biggl\{(y,\beta_E, u_{-E},c): y \in \real^n, }\\
& \qquad \beta_E \in \real^{|E|}, \|\beta_E\|_1 = \delta, \text{diag}(z_E)\beta_E>0,  \\
& \qquad u_{-E} \in \real^{p-|E|}, \|u_{-E}\|_{\infty} \leq 1,\\
& \qquad  c\in\mathbb{R}, c > 0 \biggr\}.
\end{aligned}
$$
Another possible reparametrization, can be stated as
$$
\begin{aligned}
\lefteqn{\biggl\{(y,\beta_{E\setminus 1}, u_{-E}, c): y \in \real^n, }\\
& \qquad \beta_{E\setminus 1} \in \real^{|E|-1}, s_{E,1}\beta_{E,1} =\delta- s_{E\setminus 1}^T\beta_{E\setminus 1}, \text{diag}(s_{E\setminus 1})\beta_{E\setminus 1}>0,  \\
& \qquad u_{-E} \in \real^{p-|E|}, \|u_{-E}\|_{\infty} \leq 1, \\
& \qquad  c\in\mathbb{R}, c > 0 \biggr\},
\end{aligned}
$$
where $\beta_{E\setminus 1} = (\beta_{E,2},\ldots, \beta_{E,|E|})$
and similarly $s_{E\setminus 1}$. The above parameterization can be expressed as
$$
\begin{aligned}
\psi_{(E,z_E)}(y,\beta_{E\setminus 1},u_{-E},c) 
= \biggl ( y, \: \epsilon &\begin{pmatrix} \beta_{E} \\ 0 \end{pmatrix} - X^T(y-X_E\beta_E) + c \begin{pmatrix} z_E \\ u_{-E} \end{pmatrix}, \\
& \begin{pmatrix} \beta_{E} \\ 0 \end{pmatrix}, 
-X^T(X-X_E\beta_E) ,   c\begin{pmatrix} z_E \\ u_{-E} \end{pmatrix} \biggr ),
\end{aligned}
$$
where $\beta_E$ is expressed in terms of $\beta_{E\setminus 1}$ as above and the corresponding randomization reconstruction map is 
\begin{equation*}
\phi_{(E,z_E)}(y,\beta_{E\setminus 1}, u_{-E}, c)  = \epsilon \begin{pmatrix} \beta_{E} \\ 0 \end{pmatrix} - X^T(y-X_E\beta_E) + c \begin{pmatrix} z_E \\ u_{-E} \end{pmatrix},
\end{equation*}
where
$$
\beta_E = \begin{pmatrix} (\delta -s_{E\setminus 1}^T\beta_{E\setminus 1})s_{E,1} \\ \beta_{E\setminus 1} \\ 0 \end{pmatrix}.
$$

The sampling density is proportional to 
\begin{equation*}
\begin{aligned}
	f(y) \: & \cdot  g\left(\phi_{(E,z_E)}(y,\beta_{E\setminus 1}, u_{-E},c) \right)\cdot \left|\det(D_{(\beta_{E\setminus 1},u_{-E},c)} \phi_{(E,z_E)})\right|,
\end{aligned}
\end{equation*}
where the last determinant equals
\begin{equation*}
\begin{aligned}
-c^{p-|E|}\det \begin{pmatrix}
	- \epsilon s_{E\setminus 1}^T s_{E,1} +X_1^T(-s_{E,1}X_1s_{E\setminus 1}^T+X_{E\setminus 1 }) & s_{E,1} \\
	\epsilon I + X_{E\setminus 1}^T(-s_{E,1}X_1s_{E\setminus 1}^T+X_{E\setminus 1 }) & s_{E\setminus 1}
\end{pmatrix}.
\end{aligned}
\end{equation*}

\subsection{Basis pursuit} \label{sec:basis:pursuit}

The basis pursuit problem \cite{basis_pursuit} is defined
as
\begin{equation}
\label{eq:basis:pursuit}
\minimize_{\beta \in \real^p: \|y-X\beta\|_2 \leq \delta } \|\beta\|_1.
\end{equation}
In the notation established so far
$$
\begin{aligned}
S &= y \\
\ell(\beta;y) &= \begin{cases}
0 & \|y-X\beta\|_2 \leq \delta \\
\infty & \text{otherwise}
\end{cases} \\
 {\cal P}(\beta) &= \|\beta\|_1.
\end{aligned}
$$

The set on which the problem \eqref{eq:basis:pursuit} has a solution is
$$
\left\{y: \hat{\theta}(y) \neq 0 \right\} = \left\{y: \|(I-P_C)y\|_2 \leq \delta \right\},  
$$
where $P_C$ denotes projection onto $\text{col}(X)$.

Its randomized version is
\begin{equation}
\label{eq:basis:pursuit:random}
\minimize_{\beta \in \real^p: \|y-X\beta\|_2 \leq \delta } \frac{\epsilon}{2} \|\beta\|^2_2 - \omega^T\beta + \|\beta\|_1.
\end{equation}
This problem, with non-random choices of $\omega$ is considered in
\cite{tfocs}. We see that
$$
\left\{(y,\omega): \bar{\theta}(y,\omega) \neq \emptyset \right \}
= \left\{(y,\omega): \hat{\theta}(y) \neq \emptyset \right\}.
$$

Generically, when a solution exists, the constraint will be tight, hence
our events of interest will typically condition on $\|y-X\beta\|_2=\delta$.
We see then
$$
\begin{aligned}
\bar{\solutionset}_0^F(G,\epsilon,\ell, {\cal P})
=  \bigl\{& (y,\omega,\beta,\alpha,z) : \: y \in \text{supp}(F), \|y-X\beta\|_2=\delta, \omega \in \text{supp}(G), \\
&\omega = \epsilon\beta+\alpha+z, \alpha = c X^T(X\beta-y), c > 0, z \in \partial {\cal P}(\beta) \bigr\}.
\end{aligned}
$$

For the canonical event ${\cal B}_{(E,z_E)}$, the set 
$\bar{\solutionset}^F_{\mathcal{B}}(G,\epsilon,\ell, {\cal P})$ can be parameterized
by
\begin{equation} \label{eq:basis_pursuit:support}
\begin{aligned}
\lefteqn{ \biggl\{(y,\beta_E, u_{-E}, c): \|(I-P_C)y\|_2 \leq \delta,} \\
& \qquad \beta_E \in \real^{|E|}, \text{diag}(z_E)\beta_E > 0, \|y-X_E\beta_E\|_2=\delta, \\
& \qquad u_{-E} \in \real^{p-|E|}, \|u_{-E}\|_{\infty} \leq 1, \\
& \qquad c \in \real, c > 0 \biggr\}.\\
\end{aligned}
\end{equation}
This set has the form of a bundle over a subset of $\real^n$, 
with fibers that are the product a $|E|-1$ dimensional
ellipse, a  $p-|E|$ dimensional cube and a half-line.

The parameterization is
\begin{equation*}
\begin{aligned}
\psi_{(E,z_E)}(y,\beta_E,u_{-E},c) = \biggl(y,  \:\epsilon &\begin{pmatrix} \beta_E \\ 0 \end{pmatrix} - c X^T(y-X_E\beta_E) + \begin{pmatrix} z_E \\ u_{-E} \end{pmatrix}, \\
 & \begin{pmatrix}\beta_E \\ 0\end{pmatrix}, -cX^T(y-X_E\beta_E), \begin{pmatrix} z_E \\ u_{-E} \end{pmatrix} \biggr).
\end{aligned}
\end{equation*}

Denoting $\phi_{(E,z_E)}(\beta_E, u_{-E}, c) = \epsilon \begin{pmatrix} \beta_E \\ 0 \end{pmatrix} - c X^T(y-X_E\beta_E) + \begin{pmatrix} z_E \\ u_{-E} \end{pmatrix}$,
the sampling density is proportional to 
$$
 f(y) \cdot g \left(\phi_{(E,z_E)}(\beta_E, u_{-E}, c) \right)\cdot 
\left|\text{det}(D_{(\beta_E, u_{-E},c)}\phi_{(E,z_E)})\right|
$$
and restricted to \eqref{eq:basis_pursuit:support}, where $\text{det}(D_{(\beta_E, u_{-E},c)}\phi_{(E,z_E)})$ is the Jacobian of the parameterization
in the coordinates described above, i.e.~coordinates on the product of an $|E|-1$ dimensional
ellipse, an $p-|E|$ dimensional cube and a half-line.

\subsection{Square-root LASSO}

Using the LASSO with an unknown noise level is somewhat of a 
chicken-and-egg problem, as knowing which value of $\lambda$ to choose 
is somewhat difficult. 
Suppose that instead we use the square-root LASSO.
The square-root LASSO program \cite{sqrt_lasso,scaled_lasso} is defined as
\begin{equation}
\label{eq:sqrt:lasso}
\minimize_{\beta \in \real^p} \|y-X\beta\|_2 + \lambda \|\beta\|_1.
\end{equation}
Selective inference for this program was considered
in \citep{selective_sqrt}.
This program has the advantage that it is possible to choose
a reasonable value of $\lambda$ without knowing anything about the noise level.
Its randomized version is
\begin{equation}
\label{eq:sqrt:lasso:randomized}
\minimize_{\beta \in \real^p} \|y-X\beta\|_2 + \frac{\epsilon}{2} \|\beta\|^2_2 - \omega^T\beta + \lambda \|\beta\|_1.
\end{equation}
We can take the law of $\omega$ to be independent of any particular
noise level, though we probably should choose $\epsilon=O(n^{-1/2})$ if thinking of something like the pairs model.

In the notation established so far
$$
\begin{aligned}
S &= y, \\
\ell(\beta;y) &=  \|y-X\beta\|_2, \\
{\cal P}(\beta) &= \lambda \|\beta\|_1.
\end{aligned}
$$

We are most interested in inference when $y-X\hat{\beta}(y,\omega) \neq 0$ on which $\ell$ is differentiable and for the canonical event
${\cal B}_{(E,z_E)}=\{\beta: \text{diag}(z_E)\beta_E > 0, \beta_{-E}=0\}$ we can parametrize $\{(y,\omega, \beta,\alpha,z)\in \bar{\mathcal{S}}_0(G,\epsilon, \ell,\mathcal{P}): y\in D_{obs}, \beta\in \mathcal{B}_{(E,z_E)}\}$ as
\begin{equation*}
\begin{aligned}
\psi_{(E,z_E)}(y,\beta_E, u_{-E}) = \Bigg(y, \epsilon & \begin{pmatrix} \beta_E \\ 0 \end{pmatrix} - \frac{ X^T(y-X_E\beta_E)}{\|y-X_E\beta_E\|_2} +  \lambda \begin{pmatrix} z_E \\ u_{-E} \end{pmatrix}, \\
&\begin{pmatrix} \beta_E \\ 0\end{pmatrix}, - \frac{ X^T(y-X_E\beta_E)}{\|y-X_E\beta_E\|_2},\lambda\begin{pmatrix} z_E \\ u_{-E} \end{pmatrix} \Bigg).
\end{aligned}
\end{equation*}
with the domain $D_{obs}\times \mathbb{R}^{|E|}_{z_E}\times[-1,1]^{p-|E|}$.
We therefore must sample from a density proportional to
$$
f(y)\cdot g \left(\epsilon \begin{pmatrix} \beta_E \\ 0 \end{pmatrix} - \frac{X^T(y-X_E\beta_E)}{\|y-X_E\beta_E\|_2} + \lambda \begin{pmatrix}  z_E \\ u_{-E} \end{pmatrix} \right) \cdot \left|\det \left(\frac{X_E^TR(y,\beta_E)X_E}{\|y-X_E\beta_E\|_2} + \epsilon I \right)\right|
$$
and restricted to to the domain of $\psi_{(E,z_E)}$, where
\begin{equation*}
\label{eq:sqrt:projection}
R(y,\beta_E) = I_n -\frac{ (y - X_E\beta_E) (y - X_E\beta_E)^T}{\|y-X_E\beta_E\|^2_2}.
\end{equation*}

\subsection{More complex penalties: generalized LASSO and other quadratic programs} \label{sec:qp}

Clearly, not every statistical learning problem of interest can be expressed in terms of the 
LASSO or forward stepwise. For example, the generalized LASSO \cite{tibs_taylor} considers a penalty
of the form
$$
{\cal P}(\beta) = \lambda \|D\beta\|_1.
$$
Selective inference for the entire solution path of the generalized LASSO, analgous to the solution path of 
Least Angle Regression and LASSO in \cite{spacings} has been considered in \cite{selective_genlasso}. We consider
a fixed value of $\lambda$ combined with randomization.
For some $D$, the programs can be re-expressed as LASSO problems but not all $D$ (c.f. \cite{tibs_taylor}). 
More generally, we might be confronted with solving a problem of the form
\begin{equation}
\minimize_{\beta} \ell(\beta;S) + \bar{\cal P}(D\beta - \gamma).
\end{equation}
That is, we consider our usual problem with
$$
{\cal P}(\beta) = \bar{\cal P}(D\beta-\gamma).
$$

For example, a linearly constrained quadratic program might take the form
\begin{equation}
\label{QP:general}
\minimize_{\beta} \frac{1}{2} \beta^TQ\beta - S^T\beta
\end{equation}
subject to $D\beta \leq \gamma$ with $S \sim F$. In this case, $\bar{\cal P}$ is the non-positive
cone constraint.

Sometimes, the Fenchel conjugate ${\mathcal{P}}^*$ is simple in the sense that its
naturally associated selection events are easy to parametrize. In this case,
we can proceed as before and simply consider randomized programs of the form
\begin{equation}
\label{QP:general:random}
\minimize_{\beta} \frac{1}{2} \beta^TQ\beta - (S+\omega)^T\beta + \bar{\cal P}(D\beta - \gamma)  + \frac{\epsilon}{2} \|\beta\|^2_2.
\end{equation}

However, when ${\mathcal{P}}^*$ is complex, then our usual approach may be computationally
expensive. In this section, we describe an alternative randomization
scheme that may yield simpler sampling algorithms as described in Section \ref{sec:implementation}.

For concreteness, we consider the quadratic program \eqref{QP:general:random}. Similar calculations
hold for the generalized LASSO by simply replacing one convex conjugate with another.
The issue of parameterization of selection events
arises quickly upon inspection of the subgradient equation for
the randomized program \eqref{QP:general:random}.
It will often be of interest to condition on the set of
tight constraints
$$
\hat{E}(S,\omega) = \left\{j: D_j^T\hat{\beta}(S,\omega)=\gamma_j \right\}.
$$
On this event, the subgradient equation reads
$$
\omega = Q\beta - S + D_{E}^Tz_E
$$
where $z_E \geq 0$ and $D_E\beta=\gamma_E$. 
We see the sub-gradient $D_E^Tz_{E}$ lies in the cone generated by the rows of $D_{E}$. 
If we were to parametrize this selection event, we might write it as
$$
\omega = Q\beta - S + \bar{z}_E
$$
where $\bar{z}_E \in \text{cone}(D_E)$. For arbitrary $D$ and $E$, this cone may be rather complex. Each step of the
projected Langevin implementation described in Section \ref{sec:implementation} requires projection onto 
this set. Projection onto this set can be achieved via a non-negative least squares problem, which would
in principal need to be re-run for every step of the sampler.

We shall instead solve the corresponding randomized dual program. Introducing
variable $v=D\beta-\gamma$, the dual of \eqref{QP:general} which solves 
$$\maximize_{u}\left\{ \minimize_{\beta, v: v\leq 0} \frac{1}{2} \beta^TQ\beta - S^T\beta + \frac{\epsilon}{2} \|\beta\|^2_2 + u^T(D\beta-\gamma-v)\right\},$$
subject to $u\geq 0$ is equivalent to solving
$$\minimize_u \frac{1}{2}(D^Tu-S)^T(Q+\epsilon I)^{-1}(D^Tu-S) + u^T\gamma$$
subject to $u \geq 0$ (for the generalized LASSO, the non-negative cone constraint is replaced with the constraint $\|u\|_{\infty} \leq \lambda$).  
The set of tight constraints can be thought of as corresponding to non-zero $u$'s in this dual problem $E$. When the dual 
problem is strongly convex this is a one-to-one correspondence, though if it is not strongly convex there may be more than one dual solution \cite{tibs_taylor_df}. 

The randomized dual problem solves the program
\begin{equation}
\label{rand:dual:QP}
\minimize_u \frac{1}{2}(D^Tu-S)^T(Q+\epsilon I)^{-1}(D^Tu-S) + u^T(\gamma - \omega)
\end{equation}
subject to $u \geq 0$, which is dual to 
$$\minimize_{\beta} \frac{1}{2} \beta^T Q\beta - S^T\beta +\frac{\epsilon}{2} \|\beta\|^2_2+ I_{\mathcal{K}}(\beta)$$
with $$\mathcal{K}=\{\beta: D\beta-\gamma +\omega\leq 0\}$$ and $I_{\mathcal{K}}$ denoting its characteristic function. However, for the above randomized QP that we propose to solve, the law of $\omega$ has to supported on the set 
$$
\left\{\omega: \cap_j \left\{\beta: D_j^T\beta \leq \gamma_j-\omega_j\right\} \neq 0 \right\}.
$$ to ensure feasibility of the primal problem, though the dual problem always has a solution. 
\begin{remark}
Generally, this is not a problem in examples like the generalized LASSO, where $\mathcal{P}$ is finite everywhere. Further, in the applications of selective inference we have in mind, a data analyst is typically solving statistical learning programs to suggest parameters of interest. For example, Heirnet \cite{heirnet} is a statistical learning
method that seeks hierarchical interactions that can be expressed as a linearly constrained quadratic program. In this case,
a randomization which makes the primal problem infeasible is not catastrophic.
\end{remark}
Denoting
$$Q_{\epsilon}=(Q+\epsilon I)^{-1},$$
one solution to the above problem that ensures feasibility of the primal problem is to instead consider
\begin{equation}
\label{rand:dual:QP:smooth}
\minimize_u \frac{1}{2}(D^Tu-S)^T Q_{\epsilon}(D^Tu-S) + u^T(\gamma - \omega)+\frac{\epsilon'}{2}\|u\|^2_2
\end{equation}
subject to $u \geq 0$ for $\epsilon'>0$. This always yields a feasible primal
$$\minimize_{z\leq 0}\frac{1}{2}z^T(DQ_{\epsilon} D^T +\epsilon' I)^{-1} z-z^T(DQ_{\epsilon} D^T +\epsilon' I)^{-1}(DQ_{\epsilon}S-\gamma+\omega),$$
as $DQ_{\epsilon} D^T +\epsilon' I$ is a positive definite matrix with added positive $\epsilon'$.
The reconstruction map in this case is given by
\begin{equation}
\label{eq:QP:reconstruction}
\omega=\phi_E(S,u_,z_{-E})=DQ_{\epsilon}(D^T_Eu_E-S)+\epsilon' \begin{pmatrix} u_E\\ 0 \end{pmatrix} + \gamma + \begin{pmatrix} 0 \\ z_{-E} \end{pmatrix},
\end{equation}
which allows us to sample $(S,u_{E}, z_{-E})$ from a sampling density proportional to 
$$
f(S)\cdot g(\phi_E(S,u_E,z_{-E}))
$$
and supported on $\textnormal{supp}(F)\times \mathbb{R}^{|E|}_{+}\times \mathbb{R}^{p-|E|}_{-}$, where $\mathbb{R}_{+}=\{x\in\mathbb{R}:x\geq 0\}$ and $\mathbb{R}_{-}=\{x\in\mathbb{R}:x\leq 0\}$.

Another possibility is to consider the dual of the randomized objective \eqref{QP:general:random} given by
$$\minimize_{u\geq 0} \frac{1}{2}(D^T u -S-\omega)^T Q_\epsilon (D^T u -S-\omega)+u^T\gamma,$$
which lead to reconstruction map
$$DQ_{\epsilon}\omega =DQ_{\epsilon}D_E^T u_E- DQ_{\epsilon}S+\gamma +\begin{pmatrix} 0\\ z_{-E}\end{pmatrix}$$
where we have conditioned on the value of the set
$$
\bar{E}(S,\omega) = \left\{j: \hat{u}_j(S,\omega) > 0\right\}
$$
of being $E$. As noted above,  the set $\bar{E}(S,\omega)$ may not be unique even when $E(S,\omega)$ is (c.f. \cite{tibs_taylor_df}). In this
sense, we are conditioning on a particular value of the boundary set, the boundary set determined by the solver we use to solve
this dual problem.

Denoting the density of $DQ_{\epsilon}\omega$ as $\tilde{g}$, the sampling density of $(S,u_E,z_{-E})$ is now proportional to
$$f(S)\cdot\tilde{g}\left(DQ_{\epsilon}D_E^Tu_E-DQ_{\epsilon}S+\gamma+\begin{pmatrix} 0 \\ z_{-E}\end{pmatrix}\right)$$
and restricted to 
$$z_{-E}\leq 0, \; u_E\geq 0\; \text{ and } \; z_{-E}+\gamma_{-E} \in \text{col}(D_{-E}).$$
To avoid enforcing the implicit constraint of $\mathcal{P}_{D_{-E}}^{\perp}(z_{-E}+\gamma_{-E})$ being $0$ at each step of the sampler, we can sample from an approximate density for $(S, u_E, z_{-E})$ that substitutes the Dirac delta operator $\delta_0(\mathcal{P}_{D_{-E}}^{\perp}(z_{-E}+\gamma_{-E}))$ with a smoothed version, an example being 
$$f(S)\cdot\tilde{g}\left(DQ_{\epsilon}D_E^Tu_E-DQ_{\epsilon}S+\gamma+\begin{pmatrix} 0 \\ z_{-E}\end{pmatrix}\right)\cdot \exp\left(-\cfrac{1}{2\epsilon'}\left\|\mathcal{P}_{D_{-E}}^{\perp}(z_{-E}+\gamma_{-E})\right\|_2^2\right)$$ 
supported on $\textnormal{supp}(F)\times \mathbb{R}^{|E|}_{+}\times \mathbb{R}^{p-|E|}_{-}$, a much simpler constraint set.


\section{Neighborhood selection}
\label{sec:neighborhood:selection}

We now consider neighborhood selection (a selective
sampler for which appears in \cite{selective_sampler_nips}).
\citet{neighbourhood_selection} proposed neighborhood selection with the LASSO to achieve this goal. 
The algorithm can be formulated as the following optimization problem, for any node $i$
\begin{equation}
    \label{eq:ns_optimize}
\hbeta^{i, \lambda} = \argmin_{\beta\in\mathbb{R}^p, \beta_i = 0} \|x_i - X\beta\|_2^2 + \lambda \|\beta\|_1,
\end{equation}
where with slight abuse of notation, $X \in \real^{n \times p}$ is the data matrix observed from $n$ i.i.d 
observations, and $x_i$ is the $i$-th column of $X$.
Choice of $\lambda$ is discussed in Chapter 3 of \cite{neighbourhood_selection}.
Denote $$\hat{B} = (\hat{\beta}^1, \hat{\beta}^2, \dots, \hat{\beta}^p),$$ we propose the randomized version of \eqref{eq:ns_optimize},
\begin{equation}
    \hat{B} = \argmin_{B\in\mathbb{R}^{p\times p}, B_{ii} = 0 \:\forall i\in\{1,\ldots, p\}} \|X - XB\|_F^2 + \lambda \|B\|_1 + \frac{\epsilon}{2} \|B\|_F^2 - \Omega B, 
\end{equation}
where $\Omega = (\omega^1, \dots, \omega^p) \overset{i.i.d}{\sim} G$. This is the matrix form of \eqref{eq:canonical:random:program},
and the KKT conditions are decomposable across the nodes.

Suppose for node $i$, $E^i$ is the active set for \eqref{eq:ns_optimize}, $s^i$ is the corresponding signs, $\lambda \cdot u^i$ is the 
subgradient corresponding to the inactive variables except the $i$-th and $X_{-i}$ is the columns of X except the $i$-th column. For every node $i$, the $i$-th coordinate of $\beta^i$ is held to be zero, and \eqref{eq:ns_optimize} is in fact a regression of dimension $p-1$, thus
    \begin{equation} \label{eq:neighborhood:alpha}
    \alpha^i = -X_{-i}^T (x_i - X_{E^i} \beta_{E^i}^i),
    \end{equation}
    and the reparametrization map, 
    $$
    \psi_{(E,z_E)}(X, B^-, u) = (X, \alpha + z + \epsilon B^-,B^-,\alpha, z),
    $$
    where 
    $$
    \begin{aligned}
        \alpha &= (\alpha^1, \alpha^2, \dots, \alpha^p), \; \alpha_i \textnormal{ from }\eqref{eq:neighborhood:alpha}, \\
        z &= (z^1, z^2, \dots, z^p), \; z^i = \lambda \begin{pmatrix} s^i\\ u^i \end{pmatrix}, \\
        B^- & \in \real^{(p-1) \times p} \text{ is } p\times p \text{ matrix without the diagonal elements}.
    \end{aligned}
    $$
    Since $\omega^i$, $1\leq i\leq p$, are independent, and the Jacobian 
    $$
    J\psi_{(E,z_E)}(X,B^-,u) = \prod_{i \in \Gamma} \det(X_{E^i}^T X_{E^i} + \epsilon I),
    $$
    conditioning on $(E^i, s^i)$, the distribution for selective inference has the following density,
    \begin{equation}
        \label{eq:ns_law}
    \begin{aligned}
    f(X) &\cdot \prod_{i \in \Gamma} g\left(\epsilon
    \begin{pmatrix}
        \beta^i_{E^i} \\
        0
    \end{pmatrix} + 
    \lambda \begin{pmatrix}
        s^i \\
        u^i
    \end{pmatrix} - 
    X_{-i}^T (x_i - X_{E^i} \beta^i_{E^i})
    \right) \\
    &\cdot \prod_{i \in \Gamma} \left|\det(X_{E^i}^T X_{E^i} + \epsilon I)\right|.
    \end{aligned}
    \end{equation}

After seeing the active set $E = (E^1, \dots, E^p)$, it is natural to choose the selected model (i.e.~the edge set $\mathcal{E}$) to be $E$. However,
since the active set $E$ is not necessarily symmetric, we choose the edge set $\mathcal{E}$ to be $E^{\lor}$, where
$$
E^{\lor} = \{(i,j) \:|\: E_{ij} = 1 \text{ or } E_{ji} = 1\}.
$$
Under this model, the distribution of $X$ is an
exponential family. More specifically, conditioning on the set of edges $E^{\lor}$, 
$$
\begin{aligned}
f(X) &\propto \exp\left(-\frac{1}{2} \text{Tr}(\Theta X^T X)\right) \\ 
&= \exp\left(- \sum_{(i,j) \in E^{\lor}} \Theta_{ij} x_i^T x_j - \frac{1}{2}\sum_{i \in \Gamma} \Theta_{ii} \|x_i\|^2 \right). 
\end{aligned}
$$
Note that this is an exponential family with sufficient statistics 
$$\{x_i^Tx_j, ~ (i,j) \in E^{\lor}, ~ \|x_i\|^2, i \in \Gamma\}.$$ Therefore,
the law for selective inference \eqref{eq:ns_law} is also an exponential family with the same sufficient statistics. To construct
the UMPU tests as in \cite{optimal_inference} for the null hypothesis $H_{0, ij}: \Theta_{ij} = 0$, we condition on the sufficient 
statistics corresponding to the nuisance parameter and sample from the law \eqref{eq:ns_law}.


\section{Recomputing Jacobian for group lasso} 
\label{app:Jacobian:meta}

We apply the Jacobian meta theorem \ref{thm:jacobian} to the group Lasso to re-derive the Jacobian in the selective sampler density from a geometric perspective. 
For the tangential frame $V_{i,z}$ we can split this over groups as
$$
(\Pi_g^T(V_{g,j,z_g})_{1 \leq j \leq |g|-1})_{g \in E} = (V_{g,z_g})_{g \in E}
$$
where $\Pi_g:\real^p \rightarrow \real^g$ is projection onto
the $g$ coordinates represented by the matrix $\Pi_g \in \real^{|g| \times p}$. The the vectors $V_{g,j,z_g}$ are chosen to be are orthonormal within $T_g(\lambda_g S(\real^g))$ additionally satisfying $V_{g,j,z_g}^Tz_g=0$. For the normal frame we can take
$\eta_g = z_g / \lambda_g$ and
$$
\beta = \sum_{g \in G} \gamma_g z_g / \lambda_g.
$$
In the standard basis of $T_{\beta}\real^p$, the metric \eqref{eq:riemannian} has matrix
$$
X^TX + \epsilon I.
$$
The curvature matrix $C_{-\beta}$ is 0 except on the tangent spaces coming from active groups, on which it is block
diagonal with blocks
$$
\frac{\gamma_g}{\lambda_g} I_{|g|-1}, g \in E.
$$
Hence the matrix of interest is
$$
G(z,\gamma)^{-1}C_{-\beta}(z,\gamma) = \begin{pmatrix} G_E(z,\gamma)^{-1}C_{E,-\beta}(z,\gamma) & 0 \\ 0 & 0 \end{pmatrix}
$$
where the blocks of $G_E$ are
$$
G_{E,gh}(z,\gamma) =  V_{g,z_g}^T\Pi_g (X^TX + \epsilon I)^{-1} \Pi_h^T V_{h,z_h}, \qquad {g,h \in E}.
$$
The eigenvalues of the matrix which show up in the determinant are the same as the reciprocal
of the eigenvalues of the matrix with blocks
\begin{equation}
\label{eq:group:final:jacobian}
\bar{G}_{E,gh}(z,\gamma) =  \frac{\lambda_g}{\gamma_g}^{1/2} V_{g,z_g}^T(X^TX + \epsilon I)^{-1} V_{h,z_h} \frac{\lambda_h}{\gamma_h}, \qquad {g,h \in E}.
\end{equation}
Finally, the matrix
$H(z,\gamma)$ can be ignored if we consider $X$ to be fixed.

\section{Multiple views of the data}
\label{appendix:multiple}

In this section, we consider three more algorithms that choose variables
by querying the data several times.


\subsection{Top $K$ screening} 
\label{sec:topK:screening}

A simple way of variable selection called top $K$ screening selects the $K$ most correlated features $X_j$ with response vector $y$. 
For a given $K$, the unrandomized version of such a screening selects most correlated variables $(j_1,j_2,\ldots,j_K)$ with corresponding signs $(s_1,s_2,\ldots,s_K)$, such that 
$$s_1 X_{j_1}^T y\geq s_2 X_{j_2}^T y \geq \cdots \geq s_K X_{j_K}^T y \geq \max_{j \not \in \{j_1, \dots, j_K\}}|X_j^Ty|.$$
Selective inference in the nonrandomized setting for this problem was considered in 
\citep{lee_screening}.

Here, we consider a randomized version of the optimization above yielding the $k$-th most correlated variable for $1\leq k\leq K$, with added independent randomization variable $\omega_k\sim G_k \in \mathbb{R}^{p-k+1}$ and corresponding density $g_k$ is given by
\begin{equation}
\label{eq:k_step_ms:objective}
\maximize_{\eta \in \real^{p-k+1}}\eta^T (X_{-\mathcal{A}_{k-1}}^T y+\omega_k )-I_{\K_k}(\eta), \;\text{ where } y\sim F.\end{equation}
Here, $\mathcal{A}_{k}=\{j_1,j_2,\ldots,j_k\}$ is the active set including the $k$-th step, $X_{-\mathcal{A}_{k}}$ are the columns of $X$ except for the ones corresponding to the current active set $\mathcal{A}_{k}$, $\omega_k$ is a sequence of randomizations and $I_{\K_k}(\eta)$ is the characteristic function of $$\K_k=\{\eta \in\real^{p-k+1}: \|\eta\|_{1}\leq 1\},$$ that is,
\[ I_{\K_k}(\eta) = \begin{cases} 
      0 & \text{if } \eta\in\K_k\\
      \infty & \text{otherwise.}
   \end{cases}
\] 
Denote the optimal solution of \eqref{eq:k_step_ms:objective} as
\begin{equation}
	\eta_{k,j}^*=\begin{cases}
		s_k & \textnormal{ if } j=j_k\\
		0 & \textnormal{ otherwise, }
	\end{cases} 
\end{equation}
where $j\in \{1,\ldots, p\}\setminus\mathcal{A}_{k-1}=\mathcal{A}_{k-1}^c$\footnote{$\eta_k^*\in\mathbb{R}^{p-k+1}$ is indexed by this set and all $p-k+1$-dimensional vectors in this and the following section will be indexed by $\mathcal{A}_{k-1}^c$.} and $j_k = \underset{j\in\mathcal{A}_{k-1}^c}{\textnormal{argmax}}\left|X_j^Ty+\omega_{k,j}\right|$.
The subgradient equation in $k$-th step leads to a reconstruction map for the randomization given by
\begin{equation*}
\label{KKT1}
\omega_k=\phi_k(y,z_k)=-X_{-\mathcal{A}_{k-1}}^T y+z_k,
\end{equation*}
constraining sub-differential $z_k\in \real^{p-k+1}$ to 
$$
z_k\in \partial I_{\K_k}(\eta_k^*)=\{ c\cdot u: u\in\mathbb{R}^{p-k+1}, u_{j_k}=s_k, |u_j|\leq 1 \; \forall j\in\mathcal{A}_{k-1}^c, c>0\}.
$$

Conditioning on the selection event of choosing the $K$ most correlated variables with their corresponding signs
\begin{equation*}
\begin{aligned}
\label{ms-K step}
\hat{E}_{\{(s_k,j_k)\}_{k=1}^K} = &\biggl\{ \left(y,\{\omega_k\}_{k=1}^K\right)\in\mathbb{R}^n\times\prod_{k=1}^K\mathbb{R}^{p-k+1}:  \text {sign}(X_{j_k}^T y+\omega_{k})=s_{k},  \\
& s_{k} (X_{j_k}^{T}y+\omega_{k})\geq \underset{j\in\mathcal{A}_{k-1}^c}{\max}|X_j^T y+\omega_{k}|, k=1,\ldots, K \biggr \}.
\end{aligned}
\end{equation*}
 we sample $(y, z_1,\ldots, z_k)$ from the selective sampling density proportional to
\begin{equation}
\label{eq:ms:ksteps:density}
f(y)\cdot \prod_{k=1}^K g_k\left(z_k-X_{-\mathcal{A}_{k-1}}^T y\right),
\end{equation}
supported on 
$$
\mathbb{R}^n\times \prod_{k=1}^K \partial I_{\K_k}(\eta_k^*).
$$
For logistic regression, one can replace the $T$ statistics above with the score statistics
as described in Remark \ref{remark:FS:glm}.


\subsection{Stagewise algorithms}

Instead of fully projecting out the current variables at each step as in forward stepwise, one can use an incremental approach
as in a stagewise algorithm \cite{tibs_stagewise}.
The first step of a randomized version of such an algorithm might consist of solving the problem
\begin{equation}
\label{eq:stagewise}
\maximize_{\eta: \|\eta\|_1 \leq 1}\eta^T (X^T(y-X\alpha_0)+\omega_1)
\end{equation}
with $\alpha_0=\eta^*_0=0$. We update
$$
\alpha_1 = \alpha_{0} + \delta \cdot \eta^*_0
$$
for some learning rate $\delta > 0$.
Subsequent
problems are given by
\begin{equation}
\label{eq:stagewise:later}
\maximize_{\eta: \|\eta\|_1 \leq 1}\eta^T (X^T(y-X\alpha_k)+\omega_k)
\end{equation}
with solution $\eta^*_k$ and
$$
\alpha_k = \delta \cdot \sum_{j=0}^{k-1} \eta^*_j.
$$
After $K$ steps, the sampler density is thus proportional to
\begin{equation}
\label{eq:fs:stagewise:density}
f(y)\cdot \prod_{k=1}^K g_k\left(z_k-X^T(y - X \alpha_{k-1})\right),
\end{equation}
$$
(y, z_1,\ldots, z_k)\in \mathbb{R}^n\times \prod_{k=1}^K \partial I_{\K}(\eta_k^*).
$$
with $\K=\{\eta \in\real^{p}: \|\eta\|_{1}\leq 1\}$.


\subsection{Screening via thresholding randomized Post-LASSO}

An alternative way to screen is through the following two-stage procedure where we use randomized LASSO as in \eqref{eq:lasso:randomized:program}  in the first stage to select the model $(E,z_E)$. In the second stage, we solve an unpenalized, randomized program as in \eqref{eq:full:ms:objective} with the selected predictors $X_E$
\begin{equation}
	\minimize_{\gamma\in \real^{|E|}}\frac{1}{2}\|y-X_E\gamma\|_2^2+\frac{\epsilon_1}{2}\|\gamma\|_2^2-\omega_1^T\gamma,  \; \; ((X,y), \omega_1)\sim F\times G_1,
\end{equation}
where $G_1$ is a known distribution on $\mathbb{R}^{|E|}$ and $\omega_1$ is independent from the randomization variable $\omega$ used in the first stage. We perform a second stage of selection based on the output of the above convex program to threshold the resulting coefficients of $\hat{\gamma}((X,y), \omega_1)$ resulting in a further selected model $$\tilde{E}=\{i: |\hat{\gamma}_{i}((X,y),\omega_1)|>a\sigma \}$$ with their signs $\tilde{z}_{\tilde{E}}=\textnormal{sign}(\hat{\gamma}_{\tilde{E}}((X,y),\omega_1))$, where $a$ is a constant and $\sigma$ is again the scaling which can be estimated by the noise variance of the selected model as in Section \ref{sec:full:model:screening}. The canonical event of interest from the two-step procedure becomes
\begin{equation*}
\begin{aligned}
	\mathcal{B}_{(E,z_E, \tilde{E}, \tilde{z}_{\tilde{E}})} =\{(\beta_E, & u_{-E}, \gamma):\: \textnormal{diag}(z_E)\beta_E>0, \|u_{-E}\|_{\infty}\leq 1, \\
	&\textnormal{diag}(\tilde{z}_{\tilde{E}})\gamma_{\tilde{E}}>0, |\gamma_{i}|\geq a\sigma \;\forall\; i\in \tilde{E}, \|\gamma_{-\tilde{E}}\|_{\infty}<a\sigma \}.
\end{aligned}
\end{equation*}

Now the sampling density on $((X,y), \beta_E, u_{-E}, \gamma)$ becomes proportional to
\begin{equation*}
	f(X,y)\cdot g\left(\epsilon \begin{pmatrix} \beta_E \\0 \end{pmatrix}-X^T(y-X_E\beta_E)+\lambda\begin{pmatrix}  z_E \\ u_{-E} \end{pmatrix}\right)\cdot g_1\left(\epsilon_1\gamma-X^T(y-X\gamma)\right)
\end{equation*}
where the optimization variables $(\beta_E, u_{-E},\gamma)$ are restricted to $\mathcal{B}_{(E,z_E, \tilde{E}, \tilde{z}_{\tilde{E}})}$.

\end{document}